\documentclass[11pt,reqno]{amsart}
\usepackage{graphicx,verbatim,lineno,titletoc}
\usepackage{amssymb,mathrsfs}
\usepackage{amsmath}
\usepackage{adjustbox}
\usepackage{calc}
\usepackage{ytableau}
\usepackage{array}
\usepackage{tikz}
\usetikzlibrary{shapes.multipart,patterns,arrows}
\usepackage[font=footnotesize]{caption}
\usepackage[T1]{fontenc} 
\usepackage{fourier} 
\usepackage[english]{babel} 
\usepackage{appendix}
\usepackage[all]{xy}
\usepackage{enumerate}
\usepackage[english]{babel}
\usepackage{multirow}
\usepackage{float}
\usepackage{enumitem}
\usepackage{accents}
\usepackage[numbers]{natbib}
\usepackage{hyperref}
\hypersetup{colorlinks}
\usepackage{algorithm}
\usepackage[noend]{algpseudocode}
\usepackage{mathtools}
\mathtoolsset{showonlyrefs}

\usepackage[top=1.1in, left=1.1in, right=1.1in, bottom=1.1in]{geometry}
\hypersetup{colorlinks,linkcolor={red},citecolor={olive},urlcolor={red}}

\newcommand{\edit}[1]{{\textcolor{blue}{#1}}}
\usepackage{pifont}
\newcommand{\xmark}{\ding{55}}

\newcolumntype{M}[1]{>{\centering\arraybackslash}m{#1}}
\newcolumntype{N}{@{}m{0pt}@{}}

\newcommand\dataset[1]{\textsc{\texttt{#1}}}
\DeclareTextFontCommand{\textbfit}{\bfseries\itshape}

\newtheorem{theorem}{Theorem}[section]
\newtheorem{prop}[theorem]{Proposition}
\newtheorem{definition}[theorem]{Definition}
\newtheorem{lemma}[theorem]{Lemma}
\newtheorem{corollary}[theorem]{Corollary}

\theoremstyle{definition}
\newtheorem{remark}[theorem]{Remark}
\newtheorem{example}[theorem]{Example}

\def\liminf{\mathop{\rm lim\,inf}\limits}

\def\R{\mathbb{R}}

\def\E{\mathbb{E}}
\def\P{\mathbb{P}}

\def\eps{\varepsilon}

\def\x{\mathbf{x}}
\def\y{\mathbf{y}}

\def\U{\mathbf{U}}
\def\D{\mathbf{D}}

\def\param{\boldsymbol{\theta}}
\def\Param{\boldsymbol{\Theta}}

\definecolor{hancolor}{rgb}{0.1, 0.0, 0.9}

\DeclareMathOperator{\Out}{\texttt{Out}}

\DeclareMathOperator*{\argmin}{arg\,min}

\DeclareMathOperator{\surr}{\textup{\texttt{Srg}}}

\newcommand{\tr}{\textup{tr}}

\newlength\myindent
\newtheorem{innercustomassumption}{}
\newenvironment{customassumption}[1]
{\renewcommand\theinnercustomassumption{#1}\innercustomassumption}
{\endinnercustomassumption}

\theoremstyle{definition}

\allowdisplaybreaks

\makeatletter
\newcommand{\addresseshere}{%
	\enddoc@text\let\enddoc@text\relax
}
\makeatother
\begin{document}
	
	\title[Stochastic Reguiarized Majorization-Minimization]{Stochastic Regularized Majorization-Minimization \\ with weakly convex and multi-convex surrogates}
	%\title[online NMF for Markovian data]{Online nonnegative matrix factorization for Markovian data}

	\author{Hanbaek Lyu}
	\address{Hanbaek Lyu, Department of Mathematics, University of Wisconsin - Madison, WI 53706, USA}
	\email{\texttt{hlyu@math.wisc.edu}}
	\thanks{Codes are available at \url{https://github.com/HanbaekLyu/SRMM}}

	%\thanks{The codes for the main algorithm and simulations are provided in \texttt{https://github.com/HanbaekLyu/ONMF$\_$ONTF$\_$NDL}}

	%\keywords{Online matrix factorization, convergence analysis, MCMC, dictionary learning, non-negative matrix factorization, networks}
	%\subjclass[2010]{37K40, 60J10, 60F10}

	\begin{abstract}
		Stochastic majorization-minimization (SMM) is a class of stochastic optimization algorithms that proceed by sampling new data points and minimizing a  recursive average of surrogate functions of an objective function. The surrogates are required to be strongly convex and convergence rate analysis for the general non-convex setting was not available.  In this paper, we propose an extension of SMM where surrogates are allowed to be only weakly convex or block multi-convex, and the averaged surrogates are approximately minimized with proximal regularization or block-minimized within diminishing radii, respectively. For the general nonconvex constrained setting with non-i.i.d. data samples,  we show that the first-order optimality gap of the proposed algorithm decays at the rate $O((\log n)^{1+\eps}/n^{1/2})$ for the empirical loss and $O((\log n)^{1+\eps}/n^{1/4})$ for the expected loss, where $n$ denotes the number of data samples processed. Under some additional assumption, the latter convergence rate can be improved to $O((\log n)^{1+\eps}/n^{1/2})$. As a corollary, we obtain the first convergence rate bounds for various optimization methods under general nonconvex dependent data setting: Double-averaging projected gradient descent and its generalizations, proximal point empirical risk minimization,  and online matrix/tensor decomposition algorithms. We also provide experimental validation of our results. 
	\end{abstract}
	
	${}$
	\vspace{-0.5cm}
	${}$
	\maketitle

	\tableofcontents

	\section{Introduction}
	\label{Introduction}
	
	\textit{Empirical loss minimization} is a classical problem setting regarding parameter estimation with a growing number of observations,  where one seeks to minimize a recursively defined empirical loss function as new data arrives. Some of its well-known applications include maximum likelihood estimation, or more generally, $M$-estimation \cite{geyer1994asymptotics, geer2000empirical, stefanski2002calculus}, as well as the online dictionary learning literature \cite{mairal2010online, mairal2013stochastic, mensch2017stochastic, lyu2020online}. On the other hand, the \textit{expected loss minimization} seeks to estimate a parameter by minimizing the loss function with respect to random data. It provides a general framework for stochastic optimization literature \cite{schneider2007stochastic, marti2005stochastic, bottou2008tradeoffs, nemirovski2009robust}. Optimization algorithms for empirical or expected loss minimization are in nature `online',  meaning that sampling new data points and adjusting the current estimation occurs recursively. Such online algorithms have proven to be particularly efficient in large-scale problems in statistics, optimization, and machine learning \cite{bottou1998online, duchi2009efficient, ghadimi2013stochastic, kingma2014adam}.

	First-order methods for expected loss minimization include (projected or proximal) \textit{Stochastic Gradient Descent} (SGD)  have been extensively studied for many decades, which usually consist of stochastically estimating the full gradient of the expected loss function and performing a gradient descent step followed by a projection onto the parameter space. For a general constrained and nonconvex (weakly convex) setting,  global convergence to stationary (first-order optimal) points of such methods was recently established in \cite{davis2019stochastic} with a rate of convergence  $O(\log n/\sqrt{n})$, where $n$ denotes the number of iterations (processed samples). This result assumes arbitrary initialization, although faster convergence with special initialization is known in a matrix factorization setting \cite{wang2017unified}. Throughout this paper, we are concerned with the convergence of online algorithms with arbitrary initialization, which is often referred to as a `global convergence' in the literature. 
	
	On the other hand, \textit{Stochastic Majorization-Minimization} (SMM) is one of the most popular approaches that directly solve empirical loss minimization 
	\cite{mairal2013stochastic} by sampling data points from a target data distribution and minimizing a recursively defined majorizing surrogate of the empirical loss function. This method is an online extension of the classical majorization-minimization \cite{lange2000optimization} principle and encompasses Expectation-Minimization in statistics \cite{neal1998view, cappe2009line}, and also generalizes the celebrated online matrix factorization algorithm in \cite{mairal2010online}. While it was observed empirically that SMM shows competitive performance while requiring less parameter tuning than SGD for online dictionary learning problems \cite{mairal2010online}, the theoretical convergence guarantee of SMM in the literature only ensures asymptotic convergence to stationary points and lacks any convergence rate bounds \cite{mairal2010online, mairal2013stochastic, mairal2013optimization, mensch2017stochastic}. 
	
	Most theoretical analyses on online optimization algorithms assume the ability to sample i.i.d. data points from the target distribution. This is a classical and convenient assumption for analysis, but it is violated frequently in practice, especially when the samples are accessed through Markov chain-based methods. A common practice is to first sample a single Markov chain trajectory of dependent data samples (after some burn-in period), which is then thinned (subsampled) to reduce the dependence. However, such practice only weakens the dependence between data samples and does make it truly independent in general. One may re-initialize independent Markov chains for every sample to attain true independence, but this approach suffers from a huge computational burden. Optimization algorithms with Markovian data samples were studied in \cite{johansson2007simple, johansson2010randomized} in the context of distributed optimization in networks. More recently, it was shown in \cite{sun2018markov}  that arbitrarily initialized SGD almost surely converges to critical points of unconstrained nonconvex objectives at rate $O( (\log n)^{2}/\sqrt{n} )$, even when the data samples have a Markovian dependence. Block Coordinate Descent with Markovian coordinate selection was also studied recently in \cite{sun2020markov}. 
	
	%\edit{Cite also 
		
		On the other hand, in \cite{lyu2020online}, the online matrix factorization algorithm in \cite{mairal2010online} based on SMM is extended and shown to converge to stationary points for constrained matrix factorization problems in the Markovian setting. Based on the result and combined with an MCMC network sampling algorithm in \cite{lyu2023sampling}, an online dictionary learning algorithm for learning latent motifs in networks is proposed in \cite{lyu2020learning}. More recently, an online algorithm for nonnegative tensor dictionary learning utilizing CANDECOMP/PARAFAC (CP) decomposition is developed in \cite{lyu2020online_CP}, where convergence to stationary points under the Markovian setting is established. The algorithm is based on SMM but uses new components such as block coordinate descent with diminishing radius \cite{lyu2020convergence}. Although the works \cite{lyu2020online,lyu2020online_CP} are the first to establish global convergence of SMM-type methods to stationary points in the Markovian setting, there has not been a convergence rate bounds, which is still the case in the i.i.d. setting. A rate of convergence result is important also in practice since it provides bounds on the number of iterations and data samples sufficient to guarantee a to obtain an approximate solution up to the desired precision.

		Our goal in this paper is twofold. First, we generalize the framework of SMM so that not only strongly convex surrogates can be used, but also the more general class of weakly convex or multi-convex surrogates can be used with suitable regularization. Second, we intend to provide the missing convergence rate analysis of SMM for general constrained nonconvex objectives, both in the context of empirical and expected loss minimization. 
		We call our algorithm \textit{stochastic regularized majorization-minimization} (SRMM), which generalizes the original SMM  \cite{mairal2013optimization}, online matrix factorization  \cite{mairal2010online,lyu2020online}, stochastic approximate majorization-minimization and subsampled online matrix factorization  \cite{mensch2017stochastic}, and the online CP-dictionary learning algorithm  \cite{lyu2020online_CP}.

		\subsection{Contribution}

		Our algorithm and analysis of SMM consider three cases: 1) Strongly convex surrogates without regularization; 2) Multi-convex surrogates with block coordinate descent with diminishing radius; 3) Strongly convex surrogates with proximal regularization. A concise summary of our results is given in Table \ref{table:datasets_big} and in the following bullet points.
		
		%We summarize the highlights of our contribution below: %Our analysis applies to the generalized version of SMM that works with weakly convex or multi-convex surrogates. We also allow using approximate surrogate functions for robustness under inexact computations for solving convex sub-problems.  For a parameter update, averaged weakly convex surrogate functions are inexactly minimized with proximal regularization, or averaged multi-convex surrogates are block-minimized either with cyclic or randomized scheduling for more efficient computation. 
		
		\begin{description}
			\item[SMM with strongly convex surrogates] 
			\item[\qquad $\bullet$] Rate of convergence of classical SMM with strongly convex surrogates and i.i.d. data samples is established: $O((\log n)^{1+\eps} n^{-1/2})$ for empirical loss and $O((\log n)^{1+\eps} n^{-1/4})$ for the expected loss. 
			
			\item[\qquad $\bullet$] The results also hold against inexact computation of surrogates, inexact surrogate minimization, and hidden Markovian dependence in data samples.
			
			\item[SMM with multi-convex surrogates]
			\item[\qquad $\bullet$] SMM with block multi-convex surrogates is developed, where averaged surrogates are approximately block-minimized within a diminishing radius, either with cyclically or randomly chosen block coordinates.
			\item[\qquad $\bullet$] Coordinate descent for surrogate minimization reduces computational cost in high dimension. 
			\item[\qquad $\bullet$] Same convergence rate results as in the previous case hold.
			
			\item[SMM with weakly convex surrogates] 
			\item[\qquad $\bullet$] SMM with weakly convex surrogates and proximal regularization is developed. In particular, only $L$-smooth surrogates can be used without verifying convexity. 
			\item[\qquad $\bullet$] Proximal regularization improves the conditioning of surrogate minimization without changing the surrogates. 
			\item[\qquad $\bullet$] Same convergence rate results as in the previous case hold. 
		\end{description}

		{\small
			\begin{table}[htbp]
				\centering
				%\caption{Topic keywords of the eight topics for experiments on the Twitter data, along the columns of $\vX_1$. These topics were produced by Tensorly and used an initialization for the Neural NNCPD model.}
				%changed node count to the real amount
				\begin{tabular}{ccc||ccccc}
					\hline 
					Surrogates & $\begin{matrix} \textup{Coordinate} \\ \textup{Decent}  \end{matrix}$ 	&$\begin{matrix} \textup{Regularization}  \end{matrix}$ &  $\begin{matrix} \textup{Asymptotic} \\ \textup{Convergence} \end{matrix}$ & $\begin{matrix} \textup{Rate of} \\ \textup{Convergence} \end{matrix}$ & $\begin{matrix} \textup{Iteration} \\ \textup{Complexity} \end{matrix}$  & $\begin{matrix} \textup{Data} \\ \textup{Sampling} \end{matrix}$
					\\
					\hline
					Strongly convex & Full &  None &\checkmark &  \checkmark &  \checkmark  & \textup{Markovian}
					\\[3pt] 
					Multi-convex  & Block  &  $\begin{matrix}  \textup{Diminishing} \\ \textup{Radius} \end{matrix}$ & \checkmark &  \checkmark &  \checkmark  & \textup{Markovian}
					\\[10pt] 
					$\begin{matrix} \textup{Weakly convex} \end{matrix}$ &Full &  Proximal    & \xmark &  \checkmark &  \checkmark & \textup{i.i.d.} \\
					\hline
				\end{tabular}%
				\caption{Overview of theoretical results depending on the types of surrogates, coordinate descent, regularization, and data sampling. For the weakly convex case, asymptotic convergence to stationary points is only subsequential. See Theorem \ref{thm:global_convergence}. } 
				\label{table:datasets_big}
			\end{table}
		}

		\vspace{0.2cm}
		In a unified manner, we provide an extensive convergence analysis on the proposed algorithm in all three cases mentioned above, which we derive under possibly dependent data streams, relaxing the standard i.i.d. assumption on data samples. We obtain global convergence to stationary points of rate $O((\log n)^{1+\eps}/n^{1/2})$, matching the optimal convergence rates for SGD-based methods \cite{sun2018markov,davis2019stochastic}, where $\eps>0$ is an arbitrary constant. Interestingly, our analysis shows that SRMM (and hence SMM) is more adapted to solve empirical loss minimization than expected loss minimization, in the sense that the aforementioned rate of convergence holds for the empirical loss functions, but for the empirical loss function, a slower rate of $O((\log n)^{1+\eps}/n^{1/4})$ is obtained. This is the opposite of SGD-based methods, which converges almost surely for the expected loss at rate $O((\log n)^{2}/n^{1/2})$ and rate $O((\log n)^{1+\eps}/n^{1/4})$ for the empirical loss. However, we also show that the optimal almost sure convergence rate $O((\log n)^{1+\eps}/n^{1/2})$ is obtained when the stationary points of the surrogate functions are in the interior of the constraint set. See Theorems \ref{thm:global_convergence}, \ref{thm:rate_surrogate_gaps}, and \ref{thm:rate_stationarity} for the full statements. 
		
		Our general framework of SRMM can be specialized to various stochastic optimization algorithms. The examples include 
		\begin{description}
			\item{$\bullet$} Double-averaging projected gradient descent \cite{nesterov2015quasi} and its generalizations (Sec. \ref{sec:double_avg}) 
			
			\item{$\bullet$} Proximal point empirial loss minimization  \cite{frostig2015regularizing, lin2015accelerated} (Sec. \ref{sec:PPERM}) 
			
			\item{$\bullet$} Online (Nonnegative) Matrix Factorization \cite{mairal2010online,lyu2020online} (Sec. \ref{subsection:OMF}) 
			
			\item{$\bullet$} Subsampled Online (Nonnegative) Matrix Factorization \cite{mensch2017stochastic} (Sec. \ref{sec:SOMF}) 
			
			\item{$\bullet$} Online Nonnegative CP-dictionary learning \cite{lyu2020online_CP}  (Sec. \ref{subsection:OCPDL})
		\end{description}	
		As an immediate corollary, our general results yield the first convergence rate bounds for the above algorithms in the general setting with nonconvex objectives with constraints and possibly non-i.i.d. data samples. For all of the above algorithms, there have not been any convergence rate results for nonconvex objectives even in the special case of i.i.d. data samples. For the first two examples above, even asymptotic convergence to stationary points for nonconvex objectives was not known. 
		
		In Section \ref{sec:experiments}, we also experimentally validate the efficacy of SRMM for two tasks: Network Dictionary Learning \cite{lyu2021learning}  and image classification with deep convolutional neural networks for the CIFAR-10 dataset \cite{krizhevsky2009learning}.

		\subsection{Related work}

		In \cite{mairal2013stochastic}, convergence analysis of SMM with i.i.d. data samples are given. There, it is shown that when the objective function $f$ is convex, SMM converges to the global minimum with a rate of convergence in expectation of order $O(\log n/\sqrt{n})$ (see \cite[Prop 3.1]{mairal2013stochastic}) and of order $O(1/n)$ when $f$ is strongly convex (see \cite[Prop 3.2]{mairal2013stochastic}). Also, when $f$ is nonconvex, SMM is shown to converge almost surely to the set of stationary points of $f$ over a convex constraint set (see \cite[Prop 3.4]{mairal2013stochastic}) but no result on the rate of convergence is given. The latter nonconvex result was later extended to SMM with approximate surrogate functions in \cite{mensch2017stochastic}, which was applied to developing subsampled online matrix factorization algorithms. 
		
		In \cite{lyu2020online}, the classical online matrix factorization algorithm in \cite{mairal2010online} is shown to converge when the input data matrices are given by a function of some underlying Markov chain. A similar almost sure global convergence in the Markovian setting for online tensor factorization is obtained in \cite{lyu2020online_CP}. In both works, bounds on the rate of convergence are not given.  In \cite{lyu2020convergence}, block coordinate descent with diminishing radius is proposed for deterministic nonconvex problems and shown to converge to the stationary points of the objective function. Also, a rate of convergence of order $O(\log n/\sqrt{n})$ is obtained. 
		
		Stochastic Gradient Descent (SGD) is another popular method for various optimization problems. In \cite{sun2018markov}, a convergence of SGD under the Markovian data assumption is obtained. For the convex case, \cite[Thm. 1]{sun2018markov} shows the convergence of SGD to a global minimum with the rate of convergence of order $O((\log n)^{1+\eps}/\sqrt{n} )$ for each fixed $\eps>0$ . Moreover, \cite[Thm. 2]{sun2018markov} shows that SGD for unconstrained nonconvex optimization converges almost surely to the stationary points with rate $O((\log n)^{1+\eps}/\sqrt{n} )$ for each fixed $\eps>0$. A similar rate of convergence for nonconvex and constrained projected SGD is also known in \cite{davis2019stochastic}. In \cite{zhao2017online}, an online NMF algorithm based on projected SGD with general divergence in place of the squared $\ell_{2}$-loss is proposed, and convergence to stationary points to the expected loss function for i.i.d. data samples is shown.
		
		One of the key features of SMM-type algorithms is that the chosen majorizing surrogates in each iteration are recursively averaged. However, MM-type algorithms without such recursive averaging have also been investigated extensively. For instance, for convex objectives, an iteration complexity of $O(\eps^{-1})$ is known for block successive upper-bound minimization (BSUM) algorithm \cite{hong2017iteration}.

		\subsection{Organization}
		
		In Section \ref{section:statement}, we state the problem settings of empirical and expected loss minimization and introduce background on majorization-minimization (MM) and stochastic MM (SMM). We also introduce the proposed method of stochastic regularized MM (SRMM) at a high level.  In Section \ref{section:algorithm}, we state our SRMM algorithm in Algorithm \ref{algorithm:SMM}. Next in Section \ref{section:results}, we state our main results (Theorems \ref{thm:global_convergence}, \ref{thm:rate_surrogate_gaps}, and \ref{thm:rate_stationarity}) and their corollaries (Corollaries \ref{cor:unconstrained} and \ref{cor:iteration_complexity}). 
		In Section \ref{section:applications}, we discuss applications of our general framework on various special instances of SRMM -- double averaged PSGD and its generalizations, (subsampled) online matrix factorization, and online CP-dictionary learning. In Section \ref{sec:experiments}, we provide numerical experiments of SRMM on network dictionary learning and image classification using deep convolutional neural networks. 
		
		The rest of the sections are devoted to convergence analysis for SRMM. In Section \ref{section:preliminary} we establish some preliminary lemmas. Following in Section \ref{section:key_lemmas}, we state five key lemmas (Lemmas \ref{lem:f_n_concentration_L1_gen}, \ref{lem:pos_variation}, \ref{lem:gradient_finite_sum},  \ref{lem:asymptotic_stationarity}, and \ref{lem:finite_variation_surr}) and derive all main results stated in Section \ref{section:results} from them.  The first three key lemmas will be proved in Section \ref{section:key_lemma_pf1}, whereas the other two key lemmas will be established in Sections \ref{section:key_lemma_pf2} and \ref{sec:lem_asymptotic_stationarity}. 
		
		Lastly, we give some backgrounds on Markov chains and Markov chain Monte Carlo (MCMC) sampling in Appendix \ref{sec:MC_intro}, provide examples of various surrogate functions in Appendix \ref{sec:ex_surrogates},  and give some auxiliary lemmas in Appendix \ref{sec:auxiliary_lemmas}.

		\subsection{Notation}
		\label{subsection:notation}
		In this paper, $\R^{p}$ denote the ambient space for the parameter space $\Param$ equipped with standard dot product $\langle \cdot,\cdot \rangle$ and the induced Euclidean norm $\lVert \cdot \rVert$. We denote $\mathbf{1}(A)$ the indicator function of event $A$, which takes value $1$ on $A$ and $0$ on $A^{c}$.  For each $J\subseteq \{1,\dots,p\}$, we denote $\R^{J}$ by the $|J|$-dimensional subspace of $\R^{p}$ generated by the coordinates in $J$. We identify $\R^{\{1,\dots,p\}}$ and $\R^{p}$. We call a subset $\mathbb{J} \subseteq 2^{\{ 1,\dots,p \}} \setminus \{ \emptyset \}$ a set of  \textit{coordinate blocks} if $\bigcup \mathbb{J}=\{1,\dots,p\}$. In this case, each element $J\in \mathbb{J}$ is called a \textit{coordinate block} (or \textit{block coordinate}). Note that two distinct coordinate blocks do not need to be disjoint. For $\Param\subseteq \R^{p}$, $\param\in \R^{p}$, and a block coordinate $J\subseteq \{1,\dots, p\}$, denote $\Param^{J}:=\textup{Proj}_{\R^{J}}(\Param)$ and $\param^{J}:=\textup{Proj}_{\R^{J}}(\param)$.

		\section{Problem statement}\label{section:statement}
		
		\subsection{Empirical Loss Minimization}

		Suppose we have a loss function $\ell:\mathfrak{X}\times \Param\rightarrow  [0,\infty)$ that measures the fitness of a parameter $\param\in \Param \subseteq \R^{p}$  with respect to an observed data $\x\in \mathfrak{X}$. Consider a sequence of newly observed data $(\x_{n})_{n\ge 1}$ in $\mathfrak{X}$, where $\mathfrak{X}$ can be a general topological space. We would like to estimate a sequence of parameters $(\param_{n})_{n\ge 1}$ such that $\param_{n}$ is in some sense the best fit to all data $\x_{1},\cdots,\x_{n}$ up to time $n$. In the \textit{empirical loss minimization} framework, one seeks to minimize a recursively defined empirical loss function as a new data arrives:
		\begin{align}\label{eq:def_ELM}
			\hspace{-1cm} \textup{Upon arrival of $\x_{n}$:}\quad \param_{n}	\in \argmin_{ \param \in \Param} \big( 	\bar{f}_{n}( \param ) := (1- w_{n} ) \bar{f}_{n-1}( \param ) + w_{n} \, \ell(\x_{n}, \param ) \big), 
		\end{align}
		where $(w_{n})_{n\ge 1}$ is a sequence of \textit{adaptivity weights} in $(0,1]$ and the function $\bar{f}_{n}$ is the \textit{empirical loss function} recursively defined by the weighted average as in \eqref{eq:def_ELM} with $f_{0}\equiv 0$. More explicitly, we can write 
		\begin{align}\label{eq:ELF_closed_form}
			\bar{f}_{n}(\param) = \sum_{k=1}^{n} \ell(\x_{k},\param) w^{n}_{k}, \qquad w^{n}_{k}:=w_{k}\prod_{i=k+1}^{n}(1-w_{i}).
		\end{align}

		The adaptivity weight $w_{n}$ in \eqref{eq:def_ELM} controls how much we want our new estimate $\param_{n}$ deviate from minimizing the previous empirical loss $\bar{f}_{n-1}$ to adapting to the newly observed tensor data $\x_{n}$. In the extreme case of $w_{n}\equiv 1$, $\param_{n}$ is a minimizer of the time-$n$ loss $\ell(\x_{n},\cdot)$ and ignores the past $\bar{f}_{n-1}$. If $w_{n}\equiv \alpha\in (0,1)$ then the history is forgotten exponentially fast, that is, $\bar{f}_{n}(\cdot) = \sum_{k=1}^{n} \alpha(1-\alpha)^{n-k}\, \ell(\x_{k},\cdot)$. On the other hand, the `balanced weight' $w_{n}=1/n$ makes the empirical loss be the arithmetic mean: $\bar{f}_{n}(\cdot) = \frac{1}{n}\sum_{k=1}^{n}  \ell(\x_{s},\cdot)$, which is the canonical choice in the literature including maximum likelihood estimation and online NMF problem in \cite{mairal2010online}. Hence, one may choose the sequence of weights $(w_{n})_{n\ge 1}$ in \eqref{eq:def_ELM} to decay fast for learning average features and decay slow (or keep it constant) for learning trending features). See Figure \ref{fig:ERM_adaptivity} for illustration.

		\begin{figure}[h]
			\centering\hspace{-0.4cm}
			\includegraphics[width=0.75 \linewidth]{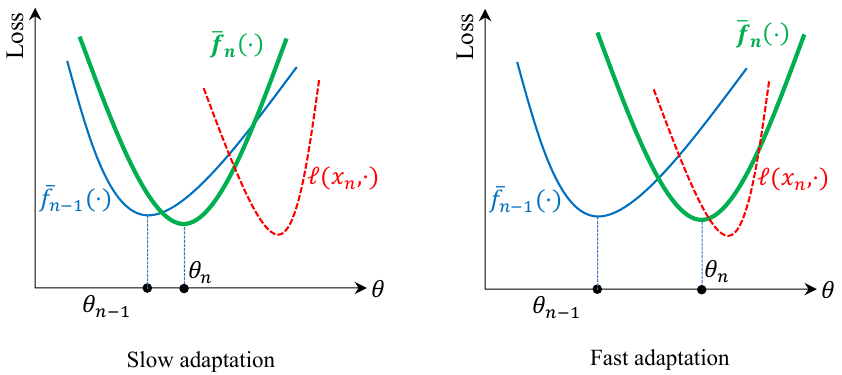}
			\caption{ Schematic plot of empirical loss $\bar{f}_{n}(\cdot)=(1-w_{n})\bar{f}_{n-1}(\cdot) + w_{n}  \ell(\x_{n}, \cdot)$ when the adaptivity weight $w_{n}$ is small (`slow adaptation' regime) or large (`fast adaptation' regime).}
			\label{fig:ERM_adaptivity}
		\end{figure}

		For instance, consider the linear regression problem where $\ell(\x,\param) = \lVert \x - \mathbf{D} \param \rVert^{2}$, where $\x\in \R^{q}$ and $\mathbf{D}\in \R^{q\times p}$ is a fixed matrix of basis features. Then by writing $\ell(\x,\param) = \param^{T} \D^{T}\D \param - 2 \param^{T}\D^{T}\x + \x^{T}\x$,  it is easy to see that 
		\begin{align}\label{eq:f_bar_LR}
			\bar{f}_{n}(\param) =  \lVert \bar{\x}_{n} - \D \param \rVert^{2} + C_{n},
			%\param^{T} \D^{T}\D \param - 2 \param^{T}\D^{T} \bar{\x}_{n} +C_{n},
		\end{align}
		where $\bar{\x}_{n}:=\sum_{k=1}^{n} \x_{k}w^{n}_{k}$ and $C_{n}$ depends only on $\x_{1},\dots,\x_{n}$. Hence $\param_{n}$ in \eqref{eq:def_ELM} in this case is the vector of best linear regression coefficients that fits the basis $\D$ to the averaged data point $\bar{\x}_{n}$. 
		
		\subsection{Expected Loss Minimization} 
		
		Instead of fitting the model underlying the loss function $\ell$ to a sequence of data points $(\x_{n})_{n\ge 1}$, consider fitting it to a single but random data point $\x$ following some probability distribution $\pi$. Then we would be seeking a single estimate $\param^{*}$ that minimizes the expected loss function $f$ as below: 
		\begin{align}\label{eq:expected_loss_minimization}
			\param^{*} \in \argmin_{\param\in \Param} \left( f(\param):= \E_{\x\sim \pi}\left[ \ell(\x,\param)  \right] \right).
		\end{align} 
		We call the above problem setting \textit{expected loss minimization}. This is a popular setting in the optimization literature for stochastic programs. A popular optimization algorithm for  solving \eqref{eq:expected_loss_minimization} is  \textit{projected stochastic gradient descent}, which proceeds by first drawing a sample $\x_{n}\sim \pi$, estimating the full gradient $\nabla f(\param_{n})$ by  the stochastic gradient $\nabla \ell(\x_{n},\param_{n-1})$, and then updating $\param_{n}\leftarrow \textup{Proj}_{\Param}(\param_{n-1} - \alpha_{n} \nabla \ell(\x_{n},\param_{n-1}))$, where $\alpha_{n}$ is a chosen stepsize and $\textup{Proj}_{\Param}$ is the projection operator onto the parameter space $\Param$.

		In the context of linear regression we discussed before, the expected loss function $f$ becomes 
		\begin{align}\label{eq:f_LR}
			f(\param) =  \lVert \E_{\x\sim \pi}[\x] - \D \param \rVert^{2} + C_{n}',
		\end{align}
		where $C_{n}'$ does not depend on $\param$ and $\D$. By comparing the empirical \eqref{eq:f_bar_LR}  and the expected \eqref{eq:f_LR} loss functions for linear regression, one can see that the two problem settings are asymptotically equivalent when $\bar{\x}_{n}\rightarrow \E_{\x\sim \pi}[\x]$ as the sample size $n$ for the empirical loss minimization tends to infinity. This is certainly true when $\x_{k}$'s are i.i.d. from $\pi$ and $w_{k}=1/k$ by the strong law of large numbers. This holds true in a more general setting where $\x_{n}$ form a Markov chain with stationary distribution $\pi$ and the weights $w_{n}$ are non-increasing and decay sufficiently faster than $1/\sqrt{n}$ (see Lemma \ref{lem:uniform_convergence_asymmetric_weights}).

		%The discussion above hints that there should be a close relationship between the two problem formulation of the empirical loss minimization  \eqref{eq:def_ELM} and the expected loss minimization \eqref{eq:expected_loss_minimization}. 

		%In this paper, we propose and analyze an algorithm that solves the stochastic program \eqref{eq:expected_loss_minimization} in a very general setting. 

		%When the objective function $f$ is nonconvex, the convergence of any algorithm for solving \eqref{eq:expected_loss_minimization} to a globally optimal solution can hardly be expected. Instead, global convergence to stationary points of the objective function is desired. Moreover, we are also interested in obtaining the rate of such convergence to stationary points. 
		
		\subsection{Stochastic Regularized Majorization-Minimization}

		Majorization-minimization (MM) is a large class of classical approaches for solving nonconvex optimization problems (see \cite{lange2000optimization} for a recent review) that includes classical gradient descent algorithms as well as expectation-maximization (EM) for solving maximum likelihood estimate (MLE) problems in statistics \cite{cappe2009line, neal1998view}. A key idea is local convex relaxation. For an illustration, suppose we would like to minimize a  differentiable function $f$ with $L$-Lipschitz gradient (i.e., $L$-smooth) and the current estimate is $\param_{n-1}$. In order to compute the new estimate $\param_{n}$, instead of directly minimizing $f$ near $\param_{n-1}$, we may minimize its  quadratic expansion
		\begin{align}
			g_{n}(\param):\param \mapsto f(\param_{n-1}) + \langle \nabla f (\param_{n-1}),\, \param-\param_{n-1} \rangle + \frac{1}{2\eta } \lVert \param - \param_{n-1} \rVert^{2},
		\end{align}
		where $\eta>0$ is a parameter. It is easy to see that $g_{n}$ is minimized at $\param_{n}:=\param_{n-1} - \eta \nabla f(\param_{n-1})$, yielding the classical gradient descent update with stepsize $\eta$. In fact, if one chooses $\eta\le L$, then $g_{n}$ majorizes $f$ (i.e., $g_{n} \ge f$, see Lemma \ref{lem:surrogate_L_gradient}) and the resulting gradient descent algorithm converges to the global minimum at rate $O(1/k)$ when $f$ is convex (see, e.g., \cite{bottou2010large}). In general, MM algorithms proceed by iteratively minimizing a majorizing surrogate (such as $g_{n}$ above) of the objective function $f$.

		Next, consider using a similar idea for solving the ELM problem \eqref{eq:def_ELM}. There, instead of a single objective function $f$, we have a sequence of empirical loss functions $\bar{f}_{n}$ to be minimized online. It would be natural to first find a majorizing surrogate $g_{n}$ of the one-point loss function $f_{n}(\cdot):=\ell(\x_{n},\cdot)$ at the current estimate $\param_{n}$ and then undergo the same recursive averaging step as in \eqref{eq:def_ELM} to compute an averaged surrogate $\bar{g}_{n}$ so that $\bar{g}_{n}$ is a majorizing surrogate of the current objective $\bar{f}_{n}$. We can then minimize $\bar{g}_{n}$ to compute the next estimate $\param_{n}$. This method is called the \textit{stochastic marjorization-minimization} (SMM) \cite{mairal2013stochastic}, which we state concisely as follow:
		\begin{align}\label{eq:def_SMM}
			\hspace{-1cm} (\textbf{\textup{SMM}}) \quad \textup{Upon arrival of $\x_{n}$:}\quad 
			\begin{cases}
				\textup{$g_{n}\leftarrow$ Strongly convex majorizing surrogate of $f_{n}(\cdot)=\ell(\x_{n},\cdot)$;} \\
				\param_{n}	\in \argmin_{ \param \in \Param} \big( 	\bar{g}_{n}( \param ) := (1- w_{n} ) \bar{g}_{n-1}( \param ) + w_{n} \, g_{n}(\param) \big).
			\end{cases}
		\end{align}
		%SMM is an online extension of the classical majorization-minimization principle proposed by Marial \cite{mairal2013stochastic}. Algorithms in this class consist of sampling i.i.d. data samples $(\x_{n})_{n\ge 1}$ from the data distribution $\pi$ and minimizing a recursively defined surrogate loss function $\bar{g}_{t}$ that majorizes the empirical loss function $\bar{f}_{t}$.  
		SMM  has been successful in a variety of settings including online versions of dictionary learning, matrix factorization, proximal gradient, and DC programming (see \cite{mairal2010online,mairal2013optimization,mensch2017stochastic}).

		A premise of SMM is that $\bar{g}_{t}$ is strongly convex so that it is easy to find a unique minimizer, which is the case for online matrix factorization problems \cite{mairal2010online,mairal2013optimization}. When the surrogate $\bar{g}_{n}$ is only block multi-convex, that is, $\bar{g}_{n}:\Param=\Theta^{(1)}\times \dots \times \Theta^{(m)}\rightarrow [0,\infty)$ and it is convex in each of the $m$ blocks, then finding a global minimum of $\bar{g}_{n}$ is not easy, if not impossible. Moreover, even when $\bar{g}_{n}$ is convex, one can still exploit its block multi-convex structure and use a more efficient coordinate descent method to minimize $\bar{g}_{n}$. After all, it may be enough to minimize $\bar{g}_{n}$  only approximately at each $n$ using a cheap coordinate descent method since we are interested in solving an online problem. 
		
		To this end, we propose the following generalization of SMM, where the surrogate $g_{n}$ functions can only be block-convex (not necessarily jointly convex) and the averaged surrogate  $\bar{g}_{n}$ can only be approximately minimized (e.g., by using a single round of block coordinate descent) within a trust region. We call our algorithm \textit{Stochastic Block Marjorization-Minimizaiton} (SRMM):
		\begin{align}\label{eq:def_SRMM}
			&\hspace{-1cm} (\textbf{\textup{SRMM}})  \quad \textup{Upon arrival of $\x_{n}$:} \\
			&\qquad \quad  
			\begin{cases}
				\textup{$g_{n}\leftarrow$ (Weakly convex or multi-convex) Majorizing surrogate of $f_{n}(\cdot)=\ell(\x_{n},\cdot)$;} \\
				\param_{n}	\approx \argmin_{ \param \in \Param} \big( 	\bar{g}_{n}( \param ) := (1- w_{n} ) \bar{g}_{n-1}( \param ) + w_{n} \, g_{n}(\param) \big) + \Psi_{n}( \lVert \param-\param_{n-1} \rVert),
				\qquad\hspace{0.15cm}
			\end{cases}
		\end{align}
		where $\Psi_{n}(\cdot)$ is a regularizer that penalizes having a large value $\lVert \param_{n}-\param_{n-1} \rVert$ of paramter change.

		A motivating example for SRMM is the recent work on online nonnegative  CP\footnote{CANDECOMP/PARAFAC for tensor decomposition}-dictionary learning in  \cite{lyu2020online_CP}, which includes online nonnegative tensor CP-decomposition as a special case. There, the variational surrogate is only convex in each of the $n$ loading matrices. The method developed in \cite{lyu2020online_CP} in order to handle a similar issue for a tensor factorization setting is that, at each round, we only approximately minimize the surrogate $\bar{g}_{n}$ by a single round of cyclic block coordinate descent (BCD) in the $n$ loading matrices. This additional layer of relaxation causes a number of technical difficulties in convergence analysis. One of the key innovations used to ensure convergence of the algorithm is the use of ``search radius restriction'' developed in \cite{lyu2020convergence}, which can be regarded as a trust region method \cite{yuan2015recent, conn2000trust} with diminishing radius. In this paper, we generalize and significantly improve this approach and analysis in \cite{lyu2020online_CP}. Most importantly, we will establish a rate of convergence of SRMM in the general nonconvex, constrained, and Markovian data setting. As a corollary, we obtain a rate of convergence results for the classical SMM in the Markovian data case, which has not been available even in the i.i.d. data case. 
		%One of the reasons that data samples $(\x_{n})_{n\ge 1}$ are required to be independent in the literature is to be able to use martingale-based methods \cite{mairal2010online, mairal2013stochastic, mensch2017stochastic}. Namely, one typically shows that the one-step conditional loss is close to the population loss, and then uses a quasi-martingale convergence theory in the analysis. However, this approach is not applicable as soon as we drop the independence assumption in data samples. Our convergence analysis on dependent data sequences uses the technique of ``conditioning on a distant past", which leverages the fact that while the worst-case 1-step conditional distribution of a Markov chain is always a constant distance away from the stationary distribution $\pi$, the $N$-step conditional distribution is exponentially close to $\pi$ in $N$. This technique has been developed in \cite{lyu2020online} recently to handle dependence in data streams for online NMF algorithms. When $f$ is convex, the analysis we develop here does not even require such mixing conditions. 

	\section{Statement of the algorithm}\label{section:algorithm}

	To state the main algorithm, we first define the class of surrogate functions we use in this paper. Recall the notations in Subsection \ref{subsection:notation}.  Next, we define 
	
	\iffalse
	\begin{definition}[Block multi-convex First-order Surrogates]
		\normalfont
		Let $\param' \in \Param$ and $L\ge 0$. Suppose we can write $\Param=\Theta^{(1)}\times \cdots \times \Theta^{(m)}$, where each $\Theta^{(i)}$ is convex. We denote by $\surr_{L}^{(m)}(f, \param')$ the set of functions $g$ such that $g \ge  f$, 
		$g(\param') = f(\param')$, the approximation error $g - f$ is differentiable, and the gradient $\nabla(g-f)$ is $L$-Lipschitz continuous, and $g$ is convex in each block coordinate in $\Theta^{(i)}$. We call the functions $g$ in $\surr_{L}^{(m)}(f, \kappa)$ \textit{block multi-convex first-order surrogate functions}. Further, we denote by $\surr_{L,\rho}^{(m)}(f,\param')$ the set of functions in $ \surr_{L}^{(m)}(f,\param')$ that are $\rho$-strongly convex in each block-coordinate.
	\end{definition}
	\fi

	\begin{definition}[First-order  $\rho$-multi-convex Surrogates]\label{def:block_surrogate}
		\normalfont
		Fix a set $\mathbb{J}$ of coordinate blocks and parameters $L>0$, $\rho\in \R$, and $\eps\ge 0$. A function $g:\R^{p}\rightarrow \R$ is a (first-order) \textit{$\rho$-multi-convex $\eps$-approximate surrogate} of $f:\R^{p}\rightarrow\R$ at $\param^{*}\in \R^{p}$ on block coordinates in $\mathbb{J}$ if the following hold: 
		\begin{description}
			\item[(i)] ($\eps$-majorization) $g(\param)+\eps \ge f(\param) $ for all $\param\in \R^{p}$; 
			
			\item[(ii)] (Smoothness of error) The approximation error $h:=g-f$ is differentiable and $\nabla h$ is $L$-Lipschitz continuous. Furthermore, $h(\param^{*}) \le \eps$ and $\lVert \nabla h^{*}(\param)\lVert \le \eps$. 
			
			\item[(iii)] (Block $\rho$-multi-convexity) For each  $J\in \mathbb{J}$, $\param\mapsto g(\param)-\frac{\rho}{2} \lVert \param \rVert^{2} $ is convex on $\R^{J}$. (We allow $\rho<0$.)
		\end{description}
		Note that the parameter $\rho$ is not necessarily positive and we allow it to be any value in $\R$. Then \textbf{(iii)} states that in each block $J$,  the function $g$ is \textit{$\rho$-strongly convex} If $\rho>0$; convex if $\rho=0$; \textit{$(-\rho)$-weakly convex} if $\rho<0$. To conveniently refer to all these cases, we call a function $g$ \textit{$\rho$-convex} for each $\rho\in \R$ if the function $\param\mapsto g(\param)-\frac{\rho}{2}\lVert \param \rVert^{2}$ is convex. If this holds on $\R^{J}$ for each coordinate block $J\in \mathbb{J}$, we call $g$ \textit{$\rho$-multi-convex}.  Denote  $\surr_{L,\rho}^{\mathbb{J}}(f, \param^{*},\eps)$ for the set of all $\rho$-multi-convex $\eps$-approximate surrogates of $f$ at $\param^{*}$ with parameters $L$ and $\rho$ and coordinate blocks $\mathbb{J}$. If $\mathbb{J}$ consists of the single coordinate block $ \{1,\dots,p\}$, then we write $\surr_{L,\rho}(f, \param^{*},\eps):=\surr_{L,\rho}^{\mathbb{J}}(f, \param^{*},\eps)$, which consists of $\rho$-convex surrogates $g$ of $f$ at $\param^{*}$.  Lastly, we denote $\surr_{L,\rho}^{\mathbb{J}}(f, \param^{*}):=\surr_{L,\rho}^{\mathbb{J}}(f, \param^{*},0)$.
		%We call the functions $g$ in $\surr_{L}^{(m)}(f, \kappa)$ \textit{block multi-convex first-order surrogate functions}. Further, we denote by $\surr_{L,\rho}^{(m)}(f,\param')$ the set of functions in $ \surr_{L}^{(m)}(f,\param')$ that are $\rho$-strongly convex in each block-coordinate.
	\end{definition}

	Now we state our main algorithm, Algorithm \ref{algorithm:SMM}, which iteratively executes the following: 1) Sample a new data point $\x_{n}\in \mathfrak{X}$ using an a priori sampling algorithm (e.g., MCMC); 2) Choose a new surrogate $g_{n}$ of the loss function $\param\mapsto \ell(\x_{n}, \param)$; 3) Update and aggregate surrogate $\bar{g}_{n}$ by taking a weighted average of $\bar{g}_{n-1}$ and $g_{n}$; 4) Find an approximate minimizer of $\bar{g}_{n}$ plus a regularizer $\Psi_{n}(\lVert \param-\param_{n-1} \rVert)$. 	In the simplest case when $g_{n}$'s are strongly convex, we directly minimize the strongly convex function $\bar{g}_{n}$ over $\Param$; When $g_{n}$'s are $\rho$-weakly convex, then we minimize the strongly convex function $ \bar{g}_{n}\left( \param \right) + \frac{\lambda_{n}}{2} \lVert \param - \param_{n-1} \rVert^{2}$ over $\Param$, where $\lambda_{n}>\rho$. These two cases are covered by Algorithm \ref{algorithm:BSM-PR}. 
	\begin{algorithm}[H]
		\small
		\caption{Stochastic Regularized Majorization-Minimization (SRMM)}
		\label{algorithm:SMM}
		\begin{algorithmic}[1]
			\State \textbf{Input:} $\param_{0} \in  \Param\subseteq\R^{p}$ (initial estimate); \, $N$ (number of iterations); \,$(w_{n})_{n\ge 1}$, (non-increasing weights in $(0,1]$); $(\eps_{n})_{n\ge 1}$, (non-increasing surrogate tolerance in $[0,1]$);  \, $\mathbb{J}\in 2^{\{1,\dots,p\}}$ (set of coordinate blocks);\, $L>0$ (surrogate smoothness parameter);\, $(\rho_{n})_{n\ge 1}$ (surrogate convexity parameters in $[0,\infty)$);
			\State \quad Initialize the approximate surrogate $\bar{g}_{0}:\param \mapsto \frac{\rho}{2} \lVert \param-\param_{0} \rVert^{2}$; $\bar{\param}_{0}=\param_{0}$; $\hat{\param}=\param_{0}$;
			\State \quad \textbf{for} $n=1,\cdots,N$ \textbf{do} 
			\State \quad\quad sample a training point $\x_{n}\in \mathfrak{X}$; define $f_{n}:\param\mapsto \ell(\x_{n},\param)$;
			\State \quad\quad choose a surrogate function $g_{n}\in \surr_{L,\rho_{n}}^{\mathbb{J}}(f_{n},\param_{n-1}, \eps_{n})$;
			\State \quad\quad update the average surrogate: $\bar{g}_{n}=(1-w_{n})\bar{g}_{n-1} + w_{n}g_{n}$;
			\State \quad\quad use Algorithm \ref{algorithm:BSM-PR} or \ref{algorithm:BSM-DR}   to compute updated estimate $\param_{n}$ near $\param_{n-1}$:
			\begin{align}\label{eq:alg1_param_update}
				\qquad 	\qquad 	\param_{n} \approx \argmin_{\param\in \Param} \left(  \bar{g}_{n}(\param)  + \Psi_{n}(\lVert\param-\param_{n-1}\rVert )  \right); \qquad \left( \triangleright \,  \begin{matrix} \textup{Inexact surrogate minimization} \\ \textup{with regularization} \end{matrix} \right) 
			\end{align}
			%\State \quad\quad for option 2, update the averaged iterate: $\bar{\param}_{n} := (1-w_{n+1}) \, \bar{\param}_{n-1} + w_{n+1}\param_{n}$ 
			%\State \quad\quad for option 3, update the averaged iterate: $\hat{\param}_{n} := \left[ (1-w_{n+1}) \, \hat{\param}_{n-1} + w_{n+1}\param_{n} \right]\big/ \left( \sum_{k=1}^{n+1}w_{k}\right)$ 
			\State \quad \textbf{end for}
			\State \textbf{output:}  $\param_{N}$ 
			%\State \textbf{output} \textbf{(option 2):} $\bar{\param}_{N}$ (first averaging scheme)
			%\State \textbf{output} \textbf{(option 3):} $\hat{\param}_{N}$ (second averaging scheme)
		\end{algorithmic}
	\end{algorithm}

	\begin{algorithm}[H]
		\small
		\caption{Surrogate Minimization with Proximal Regularization}
		\label{algorithm:BSM-PR}
		\begin{algorithmic}[1]
			\State \textbf{Input:}  all input of Algorithm \ref{algorithm:SMM}; \, $\param_{n-1}\in \Param$ (current estimate); $m\in \mathbb{N}$ (number of sub-iterations); \, $\hat{\rho}>-\rho$ (proximal regularization parameter); 
			\State \textbf{Require:} $\rho\le 0$ (i.e., $\param\mapsto g_{n}(\param)+\frac{|\rho|}{2}\lVert \param \rVert^{2}$ is convex for each $n\ge 1$)
			\State   \textbf{Do}: \vspace{-0.3cm}
			\begin{align}
				&\qquad \param_{n}\in \argmin_{\param \in \Param}  \, \left[  \bar{g}_{n}\left( \param \right) + \frac{\hat{\rho}}{2} \lVert \param - \param_{n-1} \rVert^{2}\right];  \label{eq:BSM_factor_update_PR}
			\end{align} 
			
			\State \textbf{output}: $\param_{n}\in \Param$ 
		\end{algorithmic}
	\end{algorithm}

	\begin{algorithm}[H]
		\small
		\caption{Block Surrogate Minimization with Diminishing Radius}
		\label{algorithm:BSM-DR}
		\begin{algorithmic}[1]
			\State \textbf{Input:}  all input of Algorithm \ref{algorithm:SMM}; \, $\param_{n-1}\in \Param$ (current estimate); $m\in \mathbb{N}$ (number of sub-iterations); \, $c'\in (0,\infty)$ (search radius constant);  
			\State \textbf{Require:}  $g_{n}$ is block multi-convex with block structure $\mathbb{J}$ for each $n\ge 1$;
			\State \quad $\param_{n}\leftarrow \param_{n-1}$;
			\State \quad   \textbf{for} $i=1,\cdots,m$ \textbf{do}:
			\State \quad    \quad choose a block coordinate $J=J_{i}(n)\in \mathbb{J}$ independently of everything else;

			\State \quad    \quad freeze coordinates in $J^{c}$:  \vspace{-0.2cm}
			\begin{align}\label{eq:param_J_n_def}
				&\qquad \Param^{J}_{n} \leftarrow \left\{ \param \in \Param  \,\bigg|\, \param^{J^{c}}=(\param_{n})^{J^{c}}  \right\} 
			\end{align}
			
			\State \quad \quad update $\param_{n}$ on coordinates in $J$ while holding the rest of coordinates in $J^{c}$:
			\begin{align}
				&\hspace{3cm}  \param_{n}^{(i)}\in \argmin_{\param \in \Param^{J}_{n}}  \, \left[  \bar{g}_{n}\left( \param \right) + \left( \infty\cdot \mathbf{1}( \lVert \param-\param_{n-1} \rVert > c'w_{n}/m ) \right) \right] ; \quad \left( \triangleright \, \textup{Set $\infty\cdot 0 = 0$} \right) \label{eq:BSM_factor_update_DR}
			\end{align}
			
			\State \quad \textbf{end for}
			\State \textbf{output}: $\param_{n}=\param_{n}^{(m)}\in \Param$ 
		\end{algorithmic}
	\end{algorithm}

	%\edit{Edit for PR variant}

	On the other hand, when the surrogates $g_{n}$ are block multi-convex with respect to the coordinate blocks in $\mathbb{J}$, then so is their weight average $\bar{g}_{n}$. Unless $\bar{g}_{n}$ itself is convex, finding an exact minimizer of $\bar{g}_{n}$  over $\Param$ in step \eqref{eq:alg1_param_update} is infeasible. In fact, computing an exact minimizer of $\bar{g}_{n}$ in every step may not be necessary for the convergence of the algorithm, as it was observed for the problem of online CP-dictionary learning \cite{lyu2020online_CP}. Instead, by exploiting the block multi-convex structure of $\bar{g}_{n}$, we may solve a fixed number of convex sub-problems over deterministically or randomly chosen blocks in $\mathbb{J}$.

	In each step of solving \eqref{eq:alg1_param_update}, it is crucial to ensure that the new estimate $\param_{n}$ obtained by approximately minimizing $\bar{g}_{n}$ is not too far from the previous estimate $\param_{n-1}$. When $\bar{g}_{n}$ is strongly convex on full coordinates and if $\param_{n}$ is an exact minimizer of $\bar{g}_{n}$, a simple argument shows that $\lVert \param_{n}-\param_{n-1} \rVert=O(w_{n})$ (see \cite[Lem. B.8]{mairal2013stochastic}). In this case, we may directly find an exact minimizer $\param_{n}$ of $\bar{g}_{n}$ over $\Param$ as in Algorithm \ref{algorithm:BSM-PR} with $\hat{\rho}=0$. This specialization corresponds to the original SMM algorithm in \cite{mairal2013stochastic}. However, such property is not a priori satisfied in the general case when $\bar{g}_{n}$ is nonconvex or $\param_{n}$ is an inexact minimizer of $\bar{g}_{n}$ over $\Param$. 
	
	A key idea behind Algorithms \ref{algorithm:BSM-PR} and \ref{algorithm:BSM-DR} for averaged surrogate minimization is to use an additional regularization that penalizes large values of $\lVert \param-\param_{n-1} \rVert$ in \eqref{eq:alg1_param_update}. We consider two such regularization schemes stated in Algorithm \ref{algorithm:BSM-DR}:  A `hard' regularization of \textit{Diminishing Radius} (DR) and a `soft' regularization of \textit{Proximal Regularization} (PR). DR can also be viewed as a trust region method, where our trust-region takes the form of the Euclidean ball of diminishing radius $r_{n}=O(w_{n})$ in the order of adaptivity weights $w_{n}$ used for iterated averaging of the objective functions. PR is a standard regularization scheme that quadratically penalizes the distance from the old estimate \cite{parikh2014proximal}. The first-order optimality conditions for these two regularization schemes are equivalent assuming the solution of the latter lies in the interior of the trust region, but not necessarily in general. Another difference between the two regularization methods is that DR does not change the objective gradient but PR does. Nonetheless, in the context of block coordinate descent (BCD), both regularization methods guarantee convergence to stationary points \cite{ grippo2000convergence, xu2013block, lyu2020convergence}. 
	
	%There are three options in the algorithm, according to which the final estimate is appropriately averaged. The averaging scheme is used when $f$ is convex or strongly convex.  
	
	Note that each of the block minimization problems in \eqref{eq:BSM_factor_update_PR} and \eqref{eq:BSM_factor_update_DR} is a constrained convex optimization problem so it can be easily solved by a number of known algorithms (e.g., projected gradient descent \cite{beck2017first},  LARS \cite{efron2004least}, LASSO \cite{tibshirani1996regression}, and feature-sign search \cite{lee2007efficient}). Indeed,  for \eqref{eq:BSM_factor_update_PR}, the averaged surrogate $\bar{g}_{n}$ is $-\rho$-weakly convex (recall $\rho<0$) and so its proximal point modification with $\hat{\rho}>-\rho$ is $(\hat{\rho}+\rho)$-strongly convex; Also, \eqref{eq:BSM_factor_update_DR} is equivalent to minimizing $\bar{g}_{n}$  over the convex set $\Param_{n}\cap \{ \param\,|\, \lVert \param-\param_{n-1} \rVert  \le c'w_{n}/m \}$ restricted on $\R^{J}$, where the retriction of $\bar{g}_{n}$ on $R^{J}$ is convex. In particular, using standard projected gradient descent algorithms, one can decrease the optimality gap sub-linearly for convex sub-problems and linearly if the restricted objectives are strongly convex (see, e.g., \cite[Thm. 10.29]{beck2017first}).

	\section{Main results}
	\label{section:results}
	
	\subsection{Optimality conditions and convergence measures} 
	In this subsection, we introduce some notions on optimality conditions and related quantities. Here we denote $f$ to be a general objective function $\Param\rightarrow \R$, but elsewhere $f$ will denote the expected loss function in \eqref{eq:expected_loss_minimization} unless otherwise mentioned. 
	
	Recall that we say $\param^{*}\in \Param$ is a \textit{stationary point} of $f$ over $\Param$ if 
	\begin{align}\label{eq:stationary}
		\inf_{\param\in \Param} \, \langle \nabla f(\param^{*}) ,\,  \param - \param^{*} \rangle \ge 0,
	\end{align}	
	where $\langle \cdot,\, \cdot \rangle$ denotes the dot project on $\R^{p}\supseteq \Param$. This is equivalent to saying that $-\nabla f(\param^{*})$ is in the normal cone of $\Param$ at $\param^{*}$. If $\param^{*}$ is in the interior of $\Param$, then it implies $\lVert \nabla f(\param^{*}) \rVert=0$. For iterative algorithms, such a first-order optimality condition may hardly be satisfied exactly in a finite number of iterations, so it is more important to know how the worst-case number of iterations required to achieve an $\eps$-approximate solution scales with the desired precision $\eps$. More precisely, we say $\param^{*}\in \Param$ is an \textit{$\eps$-approxiate stationary point} of $f$ over $\Param$ if 
	\begin{align}\label{eq:stationary_approximate}
		-\inf_{\param\in \Param} \, \left\langle \nabla f(\param^{*}),\,  \frac{\param - \param^{*} }{\lVert \param-\param^{*} \rVert} \right\rangle  \le \sqrt{\eps}.
	\end{align}	
	This notion of $\eps$-approximate solution is consistent with the corresponding one for unconstrained problems. Indeed, if $\param^{*}$ is an interior point of $\Param$, then \eqref{eq:stationary_approximate} reduces to $\lVert \nabla f(\param^{*}) \rVert^{2}\le \eps$. It is also equivalent to a similar notion in \cite[Def. 1]{nesterov2013gradient}, which is stated for non-smooth objectives using subdifferentials instead of gradients as in \eqref{eq:stationary_approximate}. 
	
	Next, for each $\eps>0$ we define the \textit{iteration complexity} $N_{\eps}$ of an algorithm for minimizing $f$ over $\Param$ with initialization $\param_{0}\in \Param$  as 
	\begin{align}\label{eq:Neps}
		N_{\eps}=N_{\eps}(f, \param_{0}):=  \, \inf\, \{ n\ge 1 \,|\, \text{$\param_{n}$ is an $\eps$-approximate stationary point of $f$ over $\Param$} \}, 
	\end{align}
	where $(\param_{n})_{n\ge 0}$ is a sequence of estimates produced by the algorithm with an initial estimate $\param_{0}$. An upper bound on $N_{\eps}$ can be regarded as the \textit{worst-case} bound on the number of iterations for an algorithm to achieve an $\eps$-approximate solution.

	\subsection{Assumptions}
	
	In this subsection, we state all assumptions we use for establishing the main results. Throughout this paper, we denote by $\mathcal{F}_{n}$ the $\sigma$-algebra generated by the data points $\x_{1},\dots,\x_{n}$ as well as the choice of coordinate blocks $J_{1}(k),\dots,J_{m}(k)$ for $1\le k \le m$. Clearly $(\mathcal{F}_{n})_{n\ge 1}$ defines a filtration, that is, $\mathcal{F}_{0}\subseteq \mathcal{F}_{1}\subseteq \cdots$.

	\begin{customassumption}{(A1)} \label{assumption:A1-ell_smooth}
		(Per-sample loss function) There exists constants $R,L>0$ such that for each data point $\x\in \mathfrak{X}$, the function $\param\mapsto \ell(\x,\, \param)$ over $\param\in \Param$ is $R$-Lipscthiz continuous and its gradient is $L$-Lipschitz continuous. 
	\end{customassumption}
	
	\begin{customassumption}{(A2)}\label{assumption:A2-MC}
		(Data sampling) The observed data points $\x_{n}\in \mathfrak{X}$ are given by $\x_{n}=\varphi(Y_{n})$, where $Y_{n}$ is Markov chain on state space $\Omega$ and $\varphi:\Omega \rightarrow \mathbb{R}^{p}$ is a function with compact image. Furthermore, $Y_{n}$ has a unique stationary distribution $\pi$ and satisfies exponential mixing with rate $\lambda\in [0,1)$:
		\begin{align}
			\sup_{\y\in \Omega} \, \lVert P^{m}(\mathbf{y},\cdot) - \pi \rVert_{TV} \le \lambda^{m}.
		\end{align}
	\end{customassumption}

	\begin{customassumption}{(A3)} \label{assumption:A3-cvx_constraint}
		%For some $m\ge 1$, there exists integers $I_{i}\ge 1$ and compact, convex sets $\Theta^{(i)}\subseteq \R^{I_{i}}$, $i=1,\dots,m$ such that  $\Param = \Theta^{(1)} \times \dots \times \Theta^{(m)}$. 
		(Constraint sets for parameters) The constraint set $\Param$ is a  compact and convex subset of $\R^{p}$. Let $\mathbb{J}\subseteq 2^{\{1,\dots,p\}}\setminus \{\emptyset\}$ denote the set of coordinate blocks used in Algorithm \ref{algorithm:SMM}. Then for each block coodinate $J\in \mathbb{J}$, $\Param^{J}$ contains an open ball in $\R^{J}$. 
	\end{customassumption}

	\begin{customassumption}{(A4)}\label{assumption:A4-w_t}
		(Adaptivity weight decay) The weights $w_{n}\in (0,1]$ are non-increasing and satisfy $w_{n}^{-1} - w_{n-1}^{-1}\le 1$ for all sufficiently large $n\ge 1$, 	$\sum_{n=1}^{\infty}w_{n}=\infty$,  and $w_{n}\sqrt{n}=O(1/(\log n)^{1+\eps})$ for some $\eps>0$. Also, there exists a constant $C>0$ such that for $a_{n}=\lfloor C\log n \rfloor$,  the following hold:
		\begin{align}
			&\sum_{n=1}^{\infty} w_{n}^{2} \sum_{k=1}^{n+1} \E[\eps_{k}]<\infty ,\qquad  \sum_{n=1}^{\infty } a_{n}w_{n} w_{n-a_{n}}<\infty, \qquad \sum_{n=1}^{\infty} w_{n}\lambda^{a_{n}}<\infty,\\
			&	\hspace{2cm} \qquad \left( \textup{optionally, } \sum_{n=1}^{\infty } w_{n} w_{n-a_{n}}\sqrt{n}<\infty  \right),
		\end{align}
		where $(\eps_{n})_{n\ge 1}$ is given in Algorithm \ref{algorithm:SMM} and   $\lambda\in (0,1]$ is given in \ref{assumption:A2-MC}. 
	\end{customassumption}

	\begin{customassumption}{(A5)}\label{assumption:A5_sufficient_surrogate_decay}
		(Surrogate minimization) If Algorithm \ref{algorithm:BSM-PR} is used for \eqref{eq:alg1_param_update}, define 
		\begin{align}\label{eq:def_optimality_gap_PR}
			\hspace{-0.7cm} \Delta_{n}= \Delta_{n}^{(1)}:=\bar{G}_{n}(\param_{n}^{(i)}) - \bar{G}_{n}(\param_{n}^{(i\star)}) , \quad \param_{n}^{(i\star)} := \argmin_{\param\in \Param} \bar{G}_{n}(\param),\quad \bar{G}_{n}(\param):=\bar{g}_{n}(\param) + \frac{\hat{\rho}}{2}\lVert \param-\param_{n-1} \rVert^{2}.
		\end{align}
		If Algorithm \ref{algorithm:BSM-DR} is used for \eqref{eq:alg1_param_update}, define 
		\begin{align}\label{eq:def_optimality_gap}
			\hspace{-0.7cm} \Delta_{n}:=\sum_{i=1}^{m} \Delta_{n}^{(i)},\qquad  \Delta_{n}^{(i)}:=\bar{g}_{n}(\param_{n}^{(i)}) - \bar{g}_{n}(\param_{n}^{(i\star)}) , \qquad \param_{n}^{(i\star)} := \argmin_{\param\in \Param_{n}^{J_{i}}} \bar{g}_{n}(\param),
		\end{align}
		where  $\Param_{n}^{J_{i}}$ is defined in \eqref{eq:param_J_n_def}. 
		Then in both cases, $ \E\left[ \sum_{n=1}^{\infty}	\Delta_{n} \right] <\infty$.
	\end{customassumption}

	\begin{customassumption}{(A6)}\label{assumption:A6-faithful_sampling}
		(Block coordinate sampling) If Algorithm \ref{algorithm:BSM-DR} is used for \eqref{eq:alg1_param_update}, then the joint distribution of coordinate blocks $(J_{1}(n), \dots, J_{m}(n))$  chosen in Algorithm \ref{algorithm:BSM-DR} at iteration $n$ of Algorithm \ref{algorithm:SMM} does not depend on $n$.  Furthermore, the coordinate blocks are disjoint, and the expected number of each coordinate in $\{1,\dots,p\}$ appearing in all coordinate blocks $J_{1}(n), \dots, J_{m}(n)$ is constant. 
	\end{customassumption}

	\iffalse
	\begin{customassumption}{(A5)}\label{assumption:A5_sufficient_surrogate_decay}
		The functions $f_{n}:\param\mapsto \ell(\x_{n}, \param)$ and their surrogates $g_{n} \in \surr^{(m)}_{L}(f_{n}, \param_{n-1})$ have Lipschitz gradient with a uniform Lipschitz constant for all $n\ge 1$. 
	\end{customassumption}
	\fi
	
	\begin{customassumption}{(A7)}\label{assumption:A7_param_surr}
		(Parameterized surrogates) The averaged surrogates $\bar{g}_{n}$  are parameterized by some variable $\kappa_{n}$ in some compact set $\mathcal{K}$. That is, there exists a function $\bar{g}:\mathcal{K}\times \Param\rightarrow [0,\infty)$ such that $\bar{g}_{n}(\param)= \bar{g}(\kappa_{n}, \param,\eps_{n})$ for some $\kappa_{n}\in \mathcal{K}$.  Furthermore, $\bar{g}$ is Lipschitz in the first coordinate. 
	\end{customassumption}
	
	Assumptions \ref{assumption:A1-ell_smooth} and \ref{assumption:A3-cvx_constraint} are standard in the literature of constrained stochastic nonconvex optimization and online dictionary learning \cite{mairal2010online, mairal2013stochastic, mensch2017stochastic, lyu2020online, lyu2020online_CP}. Relaxing the standard i.i.d. data sampling assumption in \cite{mairal2010online, mairal2013stochastic, mensch2017stochastic}, the Markovian data assumption with exponential mixing was considered in \cite{lyu2020online, lyu2020online_CP}, which is trivially satisfied when the data samples $\x_{n}$ are i.i.d. from the target distribution $\pi$. 
	
	Assumption \ref{assumption:A4-w_t} states that the sequence of weights $w_{n}\in (0,1]$ we use to recursively define the empirical loss \eqref{eq:def_ELM} and surrogate loss \eqref{eq:def_SRMM} does not decay too fast so that $\sum_{n=1}^{\infty} w_{n}=\infty$ but  decay fast enough so that $\sum_{n=1}^{\infty} w_{n}^{2}<\infty$. This is analogous to requirements for stepsizes in stochastic gradient descent algorithms, where the stepsizes are usually required to be non-summable but square-summable (see, e.g., \cite{sun2018markov}). Note that our general results do not require the stronger assumption $\sum_{n=1}^{\infty} w_{n}^{2}\sqrt{n}<\infty$, which is standard in the literature \cite{mairal2010online, mairal2013stochastic, mensch2017stochastic,lyu2020online, lyu2020online_CP}. Also, the condition $w_{n}^{-1}-w_{n-1}^{-1}\le 1$ for all suficiently large $n$ is equivalent for the recursively defined weights $w^{n}_{k}$ in \eqref{eq:ELF_closed_form} being non-decreasing in $k$ for all sufficiently large $k$, which is required to use Lemma \ref{lem:uniform_convergence_asymmetric_weights}. We also remark that \ref{assumption:A4-w_t} is implied by the following simpler condition: 
	\begin{customassumption}{(A4')}\label{assumption:A4'-w_t}
		The sequence of non-increasing weights $w_{n}\in (0,1]$ in Algorithm \ref{algorithm:SMM} satisfy either $w_{n}=n^{-1}$ for $n\ge 1$ or $w_{n}=\Theta(n^{-\beta}(\log n)^{-\delta})$ for some $\delta >1$ and $\beta\in [1/2, 1)$ (Optionally, $\beta\in [3/4,1)$).  Also, the sequence $(\eps_{n})_{n\ge 1}$ in Algorithm \ref{algorithm:SMM} satisfies $\sum_{n=1}^{\infty} w_{n+1}^{2} \sum_{k=1}^{n+1} \E[\eps_{k+1}]<\infty$.
	\end{customassumption}
	
	\noindent If we consider the weight $w_{n}=n^{-\beta}/(\log n)^{\delta}$ for some $\beta\in [1/2,1)$ and $\delta>1$ as in \ref{assumption:A4'-w_t}, we  have 
	\begin{align}\label{eq:bound_typical_choice}
		\left( \sum_{k=1}^{n} w_{n}\right)^{-1} = O(n^{\beta-1} (\log n)^{\delta}).
	\end{align}
	Hence the above bound is optimized when $\beta=1/2$ for each fixed $\delta>1$.
	
	Next, we give some remarks on \ref{assumption:A5_sufficient_surrogate_decay}. When the averaged surrogate $\bar{g}_{n}$ is strongly convex on each coordinate block $J_{i}(n)$, then the expected optimality gap in \ref{assumption:A5_sufficient_surrogate_decay} decays exponentially fast in the number of iterations of standard constrained convex optimization algorithms such as projected gradient descent (see, e.g., \cite[Thm. 10.29]{beck2017first}). See also stochastic gradient descent or random coordinate descent (see, e.g., \cite[Thm 4.6]{bottou2018optimization} and \cite[Thm. 1]{wright2015coordinate}) for unconstrained cases. Hence, for instance, one can ensure $\E[\Delta_{n}] = O(n^{-2})$ in $O(m\log n)$ sub-iterations for computing $\param_{n}^{(1)},\dots,\param_{n}^{(m)}$ in Algorithm \ref{algorithm:BSM-DR}. Consequently, the total computational cost of Algorithm \ref{algorithm:SMM} would be the iteration complexity times a log factor, which is negligible. 
	
	On the other hand, when Algorithm \ref{algorithm:BSM-PR} is used for \eqref{eq:alg1_param_update} in Theorem \ref{thm:global_convergence}, we will always be minimizing a strongly convex function over $\Param$ to find $\param_{n}$, so the same remark applies.  In a special case when the surrogates $g_{n}$ are strongly convex and Algorithm \ref{algorithm:BSM-PR} is used with $\hat{\rho} =0$, then one can ensure $\E[\Delta_{n}]=O(w_{n}^{2})$ in $O(1)$ sub-iterations, instead of $O(\log n)$ (see Lemma \ref{lem:strongly_convex_surrogate_A6'}). More precisely,  \ref{assumption:A5_sufficient_surrogate_decay} is implied by the following simple condition, which was also used in \cite[Assumption (\textbf{I})]{mensch2017stochastic} to analyze SMM with inexact surrogate minimization:
	\begin{customassumption}{(A5')}\label{assumption:A5'_sufficient_surrogate_decay}
		(Surrogate optimality gap decay) Suppose the surrogates $g_{n}$ are $\rho$-strongly convex for some $\rho>0$ Algorithm \ref{algorithm:BSM-PR} is used with  $\hat{\rho}= 0$. Then there exists a constant $\mu\in (0,1]$ such that for all $n\ge 1$,
		\begin{align}\label{eq:surrogate_linear_decay}
			\E\left[\bar{g}_{n}(\param_{n}) - \bar{g}_{n}(\param_{n}^{\star}) \,\bigg|\, \mathcal{F}_{n-1}  \right]  \le (1-\mu) \, \left[\bar{g}_{n}(\param_{n-1}) -  \bar{g}_{n}(\param_{n}^{\star})  \right],
		\end{align}
		where $\param_{n}^{(\star)}$ denotes the exact minimizer of $\bar{g}_{n}$  over $\Param$. 
	\end{customassumption}
	
	Next, \ref{assumption:A6-faithful_sampling} asserts some properties of a random sampling of block coordinates $J_{1}(n),\dots, J_{m}(n)$ in Algorithm \ref{algorithm:BSM-DR}, which are crucially used in the proof of Lemma \ref{lem:first_order_optimality} that is pivotal to establishing a rate of convergence results in Theorems \ref{thm:rate_surrogate_gaps} and \ref{thm:rate_stationarity}. We note that \ref{assumption:A6-faithful_sampling}  is trivially satisfied if the deterministic cyclic block coordinate descent is used. Namely, if each surrogate $g_{n}$ is block multi-convex with block structure $\mathbb{J}=\{J_{1},\dots,J_{m}\}$ where $J_{1}\cup \dots \cup J_{m}$ gives a partition of full coordinates $\{1,\dots,p\}$, then we can deterministically cycle through the $m$ coordinate blocks by setting $J_{i}(n)=J_{i}$ for $1\le i \le m$. Then \ref{assumption:A6-faithful_sampling} is satisfied. Such cyclic block coordinate descent was used in the online CP-dictionary learning \cite{lyu2020online_CP}. Hence the present work generalizes such a deterministic block coordinate schedule to a possibly randomized schedule. 
	
	Lastly, \ref{assumption:A7_param_surr} asserts that the averaged surrogates $\bar{g}_{n}$ can be parameterized by a compact index set $\mathcal{K}$, which is satisfied by most practical use cases of SMM-type algorithms \cite{mairal2013stochastic} including online matrix factorization (see Subsection \ref{subsection:OMF}) and online CP-dictionary learning (see Subsection \ref{subsection:OCPDL} as well as \cite[Alg. 2]{lyu2020online_CP}). While \ref{assumption:A7_param_surr} is crucially used in deriving some of the main results (Theorem \ref{thm:global_convergence}), it is also of practical importance since it allows one to store averaged surrogates $\bar{g}_{n}$ only by storing some sufficient statistics living in a compact set $\mathcal{K}$, without needing to store all past data $\x_{1},\dots,\x_{n}$. %Sometimes this parameterization assumption of the averaged surrogates limits 

	\subsection{Statement of main results}
	
	Throughout this section, let $(\param_{n})_{n\ge 1}$ denote the output of Algorithm \ref{algorithm:SMM} with arbitrary initialization $\param_{0}$. We will consider one of the following instances:
	\begin{customassumption}{C1}\label{C1}
		(Strongly convex surrogates without regularization) $g_{n}\in \surr_{L,\rho}(f_{n},\param_{n-1}, \eps_{n})$ for $n\ge 1$ for some $L,\rho>0$. Use Algorithm \ref{algorithm:BSM-PR} with $\hat{\rho}= 0$ for solving \eqref{eq:alg1_param_update}. In this case, assume \ref{assumption:A5'_sufficient_surrogate_decay}.
	\end{customassumption}
	
	\begin{customassumption}{C2}\label{C2}
		(Multi-convex surrogates with radius restriction) $g_{n}\in \surr_{L,0}^{\mathbb{L}}(f_{n},\param_{n-1},\eps_{n})$ for $n\ge 1$ for some set $\mathbb{J}$ of coordinate blocks and for some $L>0$. Use Algorithm \ref{algorithm:BSM-DR}  for solving \eqref{eq:alg1_param_update}. Assume \ref{assumption:A5_sufficient_surrogate_decay}-\ref{assumption:A6-faithful_sampling}.
	\end{customassumption}

	\begin{customassumption}{C3}\label{C3}
		(Weakly convex surrogates with proximal regularization) $g_{n}\in \surr_{L,\rho}(f_{n},\param_{n-1}, \eps_{n})$ for $n\ge 1$ for some $L>0$ and $\rho\le 0$. Use Algorithm \ref{algorithm:BSM-PR} with $\hat{\rho}>-\rho$ for solving \eqref{eq:alg1_param_update}. Assume \ref{assumption:A5_sufficient_surrogate_decay} and that the data sequence $(\x_{n})_{n\ge 1}$ is i.i.d. from the stationary distribution $\pi$.
	\end{customassumption}

	Now we state our first main result, Theorem \ref{thm:global_convergence}, which states that Algorithm \ref{algorithm:SMM} converges globally (w.r.t. initialization) to the set of stationary points of both the empirical loss $\bar{f}_{n}$ and the expected loss $f$. Moreover, it also states that the averaged surrogate $\bar{g}_{n}$ is asymptotically an accurate approximation of the empirical and the expected loss functions at $\param_{n}$ both in the function values and gradients.

	\begin{theorem}[Global Convergence]\label{thm:global_convergence}
		
		Assume \ref{assumption:A1-ell_smooth}-\ref{assumption:A4-w_t} and \ref{assumption:A7_param_surr} hold. Then for cases \textup{\ref{C1}}-\textup{\ref{C2}}, the following hold:
		\begin{description}
			\item[(i)] (Empirical Loss Minimization) Both $|\bar{g}_{n}(\param_{n}) - \bar{f}_{n}(\param_{n})|$ and  $\lVert \nabla \bar{g}_{n}(\param_{n}) - \nabla \bar{f}_{n}(\param_{n}) \rVert $ converge to zero almost surely. Furthermore,  $\param_{n}$ is asymptotically a stationary point of $\bar{f}_{n}$ over $\Param$ almost surely if $\Delta_{n}=o(1)$ for \ref{C1}  and $\Delta_{n}=o(\lVert \param_{n}-\param_{n-1} \rVert)$ for \ref{C2}. 		
			
			\vspace{0.2cm}
			\item[(ii)] (Expected Loss Minimization) All of $ | \bar{g}_{n}(\param_{n}) - f(\param_{n})|$,  $ \lVert  \nabla \bar{g}_{n}(\param_{n})  - \nabla f(\param_{n}) \rVert$ and $\lVert \E[ \nabla \bar{g}_{n}(\param_{n})]  - \nabla f(\param_{n}) \rVert$ converge to zero alsmot surely. Furthermore,  $\param_{n}$ converges to the set of a stationary points of $f$ over $\Param$ almost surely if $\Delta_{n}=o(1)$ for \ref{C1}  and $\Delta_{n}=o(\lVert \param_{n}-\param_{n-1} \rVert)$ for \ref{C2}. 		
		\end{description}
		For the case \textup{\ref{C3}}, the following subsequential versions of \textup{\textbf{(i)}-\textbf{(ii)}} hold:
		\begin{description}
			\item[(iii)] All five quantities in \textup{\textbf{(i)}}-\textup{\textbf{(ii)}} converge to zero almost surely on some subsequence of $(\param_{n})_{n\ge 1}$. Furthermore, $\param_{n}$ is asymptotically a stationary point of $\bar{g}_{n}$ over $\Param$ almost surely provided $\Delta_{n}=o(\lVert \param_{n}-\param_{n-1} \rVert)$. 
		\end{description}
	\end{theorem}

	We note that standard convergence results in the literature of SMM \cite{mairal2010online, mairal2013optimization, lyu2020online} asserts that the SMM algorithm converges globally to the set of the stationary points of the expected loss function $f$ almost surely, which is recovered by  Theorem \ref{thm:global_convergence} \textbf{(ii)} for the first case of strongly convex surrogates without radius restriction. Note that we establish Theorem \ref{thm:global_convergence} without the optional condition in \ref{assumption:A4-w_t}, which essentially states that $\sum_{n=1}^{\infty} w_{n}^{2}\sqrt{n}<\infty$. Such condition was standard in the literature (see, e.g., \cite[Assumption (E)]{mairal2013stochastic}). The same statement for the second case of block multi-convex surrogates with radius restriction was recently obtained in \cite{lyu2020online_CP} for the context of online CP-dictionary learning. 
	
	The convergence of gradient norms $\lVert  \nabla \bar{g}_{n}(\param_{n}) - \nabla \bar{f}_{n}(\param_{n})\rVert$ and $\lVert  \nabla \bar{g}_{n}(\param_{n}) - \nabla f(\param_{n})\rVert$ are new and the asymptotic stationarity for the empirical loss functions in Theorem \ref{thm:global_convergence} \textbf{(i)} has not been elaborated very much in the aforementioned literature, since the main focus of using SMM was to solve the expected loss minimization. It is worth noting that the hypothesis for the expected loss minimization stated in Theorem \ref{thm:global_convergence} \textbf{(ii)} is a bit stronger than that for the empirical loss minimization stated in Theorem \ref{thm:global_convergence} \textbf{(i)}. This is an indication that SRMM (and hence SMM) is generically more suited to solve the empirical loss minimization than the expected loss minimization, which we will elaborate on in the forthcoming results.  
	
	In the following two results, Theorems \ref{thm:rate_surrogate_gaps} and \ref{thm:rate_stationarity}, we establish bounds on the rate of convergence of surrogate gaps as well as approximate optimality. Below, we denote by $\E[\cdot]$ the coditional expectation $\E[\cdot \,|\, \mathcal{F}_{0}]$ with respect to the time-0 information that contains the initial estimate $\param_{0}$. 
	%First, Theorem \ref{thm:rate_surrogate_gaps} gives how fast the gaps between the averaged surrogate $\bar{g}_{n}$ and the empirical ($\bar{f}_{n}$) and the expected ($f$) loss functions converge to zero:
	
	\begin{theorem}[Rate of Convergence of Surrogate Gaps and Variation]\label{thm:rate_surrogate_gaps}
		Let $(\param_{n})_{n\ge 1}$ be an output of Algorithm \ref{algorithm:SMM}. 
		Make the same assumption as in Theorem \ref{thm:global_convergence}. Then the following hold:
		\begin{description}
			\item[(i)] (Empirical Loss Minimization)  Asymptotically almost surely, 
			\begin{align}\label{eq:thm_nonconvex_empirical_rate}
				\min_{1\le k \le n}  \begin{pmatrix}  \left|\bar{g}_{k}(\param_{k}) - \bar{f}_{k}(\param_{k})\right| + \lVert \nabla \bar{g}_{k}(\param_{k})-\nabla \bar{f}_{n}(\param_{k}) \rVert^{2}   \end{pmatrix} &= O \left( \left(\sum_{k=1}^{n} w_{k}\right)^{-1}  \right).
			\end{align}
			Furthermore, if the optional condition in \ref{assumption:A4-w_t} holds, then  asymptotically almost surely, 
			\begin{align}\label{eq:thm_nonconvex_empirical_variation}
				\min_{1\le k \le n}  \begin{pmatrix} 
					\,\, \left| \E[\bar{g}_{k}(\param_{k})] - \bar{f}_{k}(\param_{k}) \right| +  \lVert \E[\nabla \bar{g}_{k}(\param_{k})]- \nabla \bar{f}_{k}(\param_{k}) \rVert^{2} \end{pmatrix} &= O \left( \left(\sum_{k=1}^{n} w_{k}\right)^{-1}  \right).
			\end{align}

			\vspace{0.1cm}
			\item[(ii)] (Expected Loss Minimization) Asymptotically almost surely, 
			\begin{align}\label{eq:thm_nonconvex_expected_variation}
				\min_{1\le k \le n}  \begin{pmatrix} 
					\,\, \left| \E[\bar{g}_{k}(\param_{k})] - f(\param_{k}) \right| +  \lVert \E[\nabla \bar{g}_{k}(\param_{k})]- \nabla f(\param_{k}) \rVert^{2} \end{pmatrix} &= O \left( \left(\sum_{k=1}^{n} w_{k}\right)^{-1}  \right).
			\end{align}
			Furthermore, if the optional condition in \ref{assumption:A4-w_t} holds, then asymptotically almost surely, 
			\begin{align}\label{eq:thm_nonconvex_expected_rate}
				\min_{1\le k \le n}  \left(  \left| \bar{g}_{k}(\param_{k}) - f(\param_{k}) \right| + \lVert  \nabla \bar{g}_{k}(\param_{k})- \nabla f(\param_{k}) \rVert^{2} \right) = O \left( \left(\sum_{k=1}^{n} w_{k}\right)^{-1}  \right).
			\end{align}
		\end{description}
	\end{theorem}

	\begin{theorem}[Rate of Convergence to Stationarity]\label{thm:rate_stationarity}
		Let $(\param_{n})_{n\ge 1}$ be an output of Algorithm \ref{algorithm:SMM}. 
		Make the same assumption as in Theorem \ref{thm:global_convergence}.   Then the following hold:
		\begin{description}
			\item[(i)] (Surrogate and Empirical Loss Stationarity) Asymptotically almost surely, 
			\begin{align}
				&\min_{1\le k \le n}  \,\, \left[ -\inf_{\param\in \Param} \left\langle \nabla \bar{g}_{k}(\param_{k}),\, \frac{(\param - \param_{k}) }{\lVert \param - \param_{k}\rVert} \right\rangle \right]  = O \left( \left(\sum_{k=1}^{n} w_{k}\right)^{-1}  \right). \label{eq:thm_convergence_bd_surrogate} \\
				& \min_{1\le k \le n}  \,\, \left[ -\inf_{\param\in \Param} \left\langle \nabla \bar{f}_{k}(\param_{k}),\, \frac{(\param - \param_{k}) }{\lVert \param - \param_{k}\rVert} \right\rangle \right]  = O \left( \left(\sum_{k=1}^{n} w_{k}\right)^{-1/2}  \right). \label{eq:thm_convergence_bd_f_t}
			\end{align}
			\item[(ii)] (Expected Loss Stationarity)   It holds that 
			\begin{align}
				&\min_{1\le k \le n}  \,\,\E \left[  -	\inf_{\param\in \Param} \left\langle \nabla f(\param_{k}),\, \frac{(\param - \param_{k}) }{\lVert \param - \param_{k}\rVert}\right\rangle \right] =  O \left( \left(\sum_{k=1}^{n} w_{k}\right)^{-1/2} + w_{n}\sqrt{n}  \right). \label{eq:thm_convergence_bd_f_E}
			\end{align} 
			Further assume the optional condition in \ref{assumption:A4-w_t} holds. Then asymptotically almost surely, 
			\begin{align}
				&\min_{1\le k \le n}  \,\, \left[  -	\inf_{\param\in \Param} \left\langle \nabla f(\param_{k}),\, \frac{(\param - \param_{k}) }{\lVert \param - \param_{k}\rVert}\right\rangle \right] =  O \left( \left(\sum_{k=1}^{n} w_{k}\right)^{-1/2}  \right). \label{eq:thm_convergence_bd_f}
			\end{align}
		\end{description}
	\end{theorem}

	To our best knowledge, the rate of convergence results in Theorems \ref{thm:rate_surrogate_gaps} and \ref{thm:rate_stationarity} are entirely new for SMM-type algorithms even under the classical setting with i.i.d. data and strongly convex surrogates. Only almost sure convergence to stationary points was known before \cite{mairal2010online, mairal2013stochastic,zhao2017online,  lyu2020online, lyu2020online_CP} in such cases. Moreover, \eqref{eq:thm_nonconvex_empirical_variation} and \eqref{eq:thm_nonconvex_expected_variation} in Theorem \ref{thm:rate_surrogate_gaps} give bounds on the variation of the random objective values $\bar{f}_{n}(\param_{n})$ and $f(\param_{n})$ against the deterministic quantity $\E[\bar{g}_{n}(\param_{n})]$ with respect to the randomness of data samples $\x_{1},\dots,\x_{n}$ and of Algorithm \ref{algorithm:SMM} (e.g., possibly randomized choice of blocks in Algorithm \ref{algorithm:BSM-DR}). We remark that \ref{assumption:A7_param_surr} in fact is not necessary for deriving Theorems \ref{thm:rate_surrogate_gaps} and  \ref{thm:rate_stationarity}.
	
	Next, we state a corollary of Theorems \ref{thm:rate_surrogate_gaps} and \ref{thm:rate_stationarity}, which specializes these results to a more familiar setting of unconstrained nonconvex optimization. 
	\begin{corollary}\label{cor:unconstrained}
		Let $(\param_{n})_{n\ge 1}$ be an output of Algorithm \ref{algorithm:SMM}. 
		Make the same assumption as in Theorem \ref{thm:global_convergence}. Further, assume that $\param_{n}$ is in the interior of $\Param$ for all $n\ge 1$. Then asymptotically almost surely, 
		\begin{align}\label{eq:cor_rates_bounds}
			&	\min_{1\le k \le n} 	\left\lVert  \nabla \bar{g}_{k}(\param_{k}) \right\rVert^{2} = O \left( \left(\sum_{k=1}^{n} w_{k}\right)^{-2}  \right), \quad 	\min_{1\le k \le n} 	\left\lVert  \nabla \bar{f}_{k}(\param_{k}) \right\rVert^{2} = O \left( \left(\sum_{k=1}^{n} w_{k}\right)^{-1}  \right),\\
			&\hspace{2.7cm} \quad 	\min_{1\le k \le n} 	\left\lVert  \nabla f(\param_{k}) \right\rVert^{2} = O \left( \left(\sum_{k=1}^{n} w_{k}\right)^{-1}  \right).
		\end{align}
	\end{corollary}
	
	Below we unpack the above results and give some remarks. For a more direct interpretation of our results, we will consider the weight $w_{n}=n^{-\beta}/(\log n)^{1+\eps}$ for some $\beta\in [1/2,1)$ and $\eps>0$ as in \ref{assumption:A4'-w_t}. Then the bound in \eqref{eq:bound_typical_choice} is optimized at $\beta=1/2$ for each fixed $\eps>0$ to 
	\begin{align}
		\left( \sum_{k=1}^{n} w_{n}\right)^{-1} = O(n^{-1/2} (\log n)^{1+\eps}).
	\end{align}
	Suppose all iterate $\param_{n}$ are in the interior of $\Param$. This would be a reasonable assumption when all stationary points of the averaged surrogates $\bar{g}_{n}$ are in the interior of $\Param$. Then	by Corollary \ref{cor:unconstrained}, 
	\begin{align}\label{eq:cor_rates_bounds1}
		&	\min_{1\le k \le n} 	\left\lVert  \nabla \bar{g}_{k}(\param_{k}) \right\rVert^{2} = O\left(  \frac{(\log n)^{2+2\eps}}{n} \right),  \quad 	\min_{1\le k \le n} 	\left\lVert  \nabla \bar{f}_{k}(\param_{k}) \right\rVert^{2} = O\left(  \frac{(\log n)^{1+\eps}}{\sqrt{n}} \right) ,\\
		&\hspace{2.7cm} \quad 	\min_{1\le k \le n} 	\left\lVert  \nabla f(\param_{k}) \right\rVert^{2} = O\left(  \frac{(\log n)^{1+\eps}}{\sqrt{n}} \right).
	\end{align}
	The last asymptotic expresses a well-known rate of convergence bound for nonconvex and unconstrained SGD  \cite{sun2018markov, ward2019adagrad}. A similar rate of convergence for nonconvex and constrained projected SGD is also known in \cite{davis2019stochastic}, although a different measure (using Moreu envelope) of the rate of convergence for the constrained nonconvex problems was used.
	
	For the general case when $\bar{g}_{n}$ may have its stationary points at the boundary of $\Param$, we cannot measure first-order optimality simply by the gradient norms as above. In this case, for the empirical loss minimization, \eqref{eq:thm_convergence_bd_f_t} in Theorem \ref{thm:rate_stationarity} states that 
	\begin{align}
		\min_{1\le k \le n}  \,\,\left[ -\inf_{\param\in \Param} \left\langle \nabla \bar{f}_{k}(\param_{k}),\, \frac{(\param - \param_{k}) }{\lVert \param - \param_{k}\rVert} \right\rangle \right]^{2}  &= O\left(  \frac{(\log n)^{1+\eps}}{\sqrt{n}} \right),
	\end{align}
	which obtains the same asymptotic rate of convergence as in the interior stationary points case above. As for the expected loss minimization, we need to assume the strong er condition $\sum_{n=1}^{\infty} w_{n}^{2}\sqrt{n}<\infty$ (i.e., the optional condition in \ref{assumption:A4-w_t}). In this case, the optimal value of $\beta$ is $3/4$, in which case we obtain 
	\begin{align}
		\min_{1\le k \le n}  \,\, \left[  -	\inf_{\param\in \Param} \left\langle \nabla f(\param_{k}),\, \frac{(\param - \param_{k}) }{\lVert \param - \param_{k}\rVert}\right\rangle \right]^{2} = O\left(  \frac{(\log n)^{1+\eps}}{n^{1/4}} \right).
	\end{align}
	Hence we lose a factor of $n^{-1/4}$ from the optimal bound in \eqref{eq:cor_rates_bounds}. While the $n^{-1/4}$ bound on the rate of convergence for the expected loss in the general case may not be sharp, this discussion gives a more quantitative indication that SRMM (and hence SMM) is generically more suited to solve the empirical loss minimization than the expected loss minimization for constrained nonconvex objectives.

	Lastly, we state a corollary of the earlier results on the iteration complexity of Algorithm \ref{algorithm:SMM}. Corollary  \ref{cor:iteration_complexity} gives a ``worst-case'' rate of convergence of Algorithm \ref{algorithm:SMM} until reaching an $\eps$-stationary points of the objective functions. 
	
	\begin{corollary}[Iteration Complexity]\label{cor:iteration_complexity}
		Let $(\param_{n})_{n\ge 1}$ be an output of Algorithm \ref{algorithm:SMM}. 
		Make the same assumption as in Theorem \ref{thm:global_convergence}.   Then the following hold:
		\begin{description}
			\item[(i)] (Empirical Loss Minimization) Suppose $w_{n}=n^{-1/2}(\log n)^{\delta}$ for some $\delta>1$. Then we have the following worst-case iteration complexity for Algorithm \ref{algorithm:SMM} with the empirical loss objective:
			\begin{align}
				N_{\eps}(\bar{f}_{n}) = O(\eps^{-2} (\log \eps^{-1})^{\delta} ).
			\end{align}
			
			\item[(ii)] (Expected Loss Minimization) Suppose $w_{n}=n^{-3/4}(\log n)^{\delta}$ for some $\delta>1$.  
			Then we have the following worst-case iteration complexity  for Algorithm \ref{algorithm:SMM} with the expected loss objective:
			\begin{align}
				N_{\eps}(f) = O(\eps^{-4} (\log \eps^{-1})^{\delta} ).
			\end{align}
			Furthermore, suppose $\param_{n}$ is in the interior of $\Param$ for $n\ge 1$. Then we may choose $w_{n}=n^{-1/2}(\log n)^{\delta}$ for some $\delta>1$ and obtain
			\begin{align}
				N_{\eps}(f) = O(\eps^{-2} (\log \eps^{-1})^{\delta} ).
			\end{align}
		\end{description}
	\end{corollary}
	
	\subsection{Remarks on main results} 
	
	Here we give some remarks regarding the main results stated in the previous subsection.

	\begin{remark}[SRMM vs. PSGD on empirical and expected loss minimization]
		
		Recall that our Theorem \ref{thm:rate_surrogate_gaps} gives the rate of convergence of our proposed algorithm of SRMM (Algorithm \ref{algorithm:SMM}) both for empirical and expected loss minimization. It is interesting to compare the rate of convergence bounds for SRMM and projected stochastic gradient descent (PSGD) algorithms with respect to both empirical and expected loss minimization. In this discussion, we omit log factors in all error bounds for simplicity and use $\widetilde{O}$ notation for big-O modulo log factors. 
		
		\begin{figure*}[h]
			\centering
			\includegraphics[width=0.55 \linewidth]{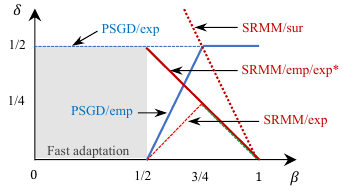}
			%\vspace{-0.5cm}
			\caption{Diagram of upper bounds of the rate of convergence $\widetilde{O}(n^{-\delta})$ of algorithms SRMM (red) and PSGD (blue) for empirical loss minimization (denoted `emp' in solid lines) with adaptivity weight $w_{n}=n^{-\beta}$,  for expected loss minimization (denoted `exp' in dashed lines), and for surrogate loss minimization (denoted `sur' in dotted line). The red solid line SRMM/exp* is the case when the iterates are assumed to be in the interior of the constraint set. The red dotted line linearly extends to $\beta=1/2$. 
			}
			\label{fig:diagram_rates}
		\end{figure*}
		
		It is known that the optimal asymptotic rate of convergence of SGD for nonconvex problems measured in terms of gradient norm squared is $\widetilde{O}(1/\sqrt{n})$ \cite{bottou2018optimization, xu2019convergence}. Similarly, PSGD for constrained nonconvex problems also has the same optimal rate of convergence \cite{davis2019stochastic}, although a direct comparison is not immediate in this case since a near-stationarity measure using gradient norm squared of Moreau envelope is used. It should be noted that all these convergence rate bounds are with respect to expected loss functions and considering (P)SGD rate of convergence for empirical loss is less standard in the literature. Nonetheless, one can transfer the rate of convergence between these two objectives, at least in expectation, using the following inequality:
		\begin{align}\label{eq:conv_rate_transfer1}
			\E\left[ \left| \inf_{\param\in \Param} \left\langle \nabla \bar{f}_{k}(\param_{k}),\, \frac{\param - \param_{k} }{\lVert \param - \param_{k}\rVert} \right\rangle  -  \inf_{\param\in \Param} \left\langle \nabla f(\param_{k}),\, \frac{\param - \param_{k} }{\lVert \param - \param_{k}\rVert} \right\rangle  \right|  \right] 
			&\le \E\left[ \sup_{\param\in \Param} \lVert \nabla \bar{f}(\param) - \nabla f(\param) \rVert \right] \\
			&= O(w_{n}\sqrt{n}). 
		\end{align}
		Indeed, the inequality follows from the Cauchy-Schwarz inequality and the equality is from Lemma \ref{lem:f_n_concentration_L1_gen}. Assuming $w_{n}=n^{-\beta}$ for $\beta \in [0,1]$, this implies 
		\begin{align}
			\min_{1\le k \le n}  \,\, \E \left[ -\inf_{\param\in \Param} \left\langle \nabla \bar{f}_{k}(\param_{k}),\, \frac{\param - \param_{k} }{\lVert \param - \param_{k}\rVert} \right\rangle \right]  &\le \min_{1\le k \le n}  \,\, \E \left[ -\inf_{\param\in \Param} \left\langle \nabla f(\param_{k}),\, \frac{\param - \param_{k} }{\lVert \param - \param_{k}\rVert} \right\rangle \right]    + O(n^{-\beta+1/2}).
		\end{align}
		
		For PSGD (also for SGD), the first term on the right-hand side is of $\widetilde{O}(n^{-1/4})$, so this implies that the rate of convergence for PSGD with respect to the empirical loss function is bounded by $\tilde{O}(n^{ -\delta  })$ with $\delta=\min(1/2, 2\beta-1)$. Note that $\delta$ stays constant $1/2$ for $\beta\in [3/4, 1]$ and decays linearly to 0 when $\beta$ decreases from $3/4$ to $1/2$ (see the solid blue line in Figure \ref{fig:diagram_rates}). This means that PSGD becomes inefficient in minimizing the empirical loss with adaptivity weight $w_{n}=n^{-\beta}$  when $\beta$ is closed to $1/2$. 
		
		On the contrary, according to Theorem \ref{thm:rate_stationarity}, SRMM rate of convergence for the empirical loss function in fact improves linearly as $\beta$ decays from 1 to $1/2$, with a better rate than PSGD over the interval $\beta\in [1/2, 2/3)$, achieving rate $\widetilde{O}(n^{-1/2})$ at $\beta=1/2$. (See the solid red line in Figure \ref{fig:diagram_rates}). In our analysis, for SRMM, we obtain twice the rate of convergence $\widetilde{O}(n^{-2(1-\beta)})$ when the averaged surrogates are minimized approximately (see the red dotted line in Figure \ref{fig:diagram_rates}), and then we transferred it to a rate bound for the empirical loss and then to the expected loss. This comparison indicates that SRMM is more adapted to and also more effective in minimizing empirical loss function than PSGD is, but PSGD may in general be more adapted to minimizing the expected loss than SRMM is. 
		
		Our analysis of SRMM critically relies on the fact that we can compare the averaged surrogate loss $\bar{g}_{n}$ with the empirical loss $\bar{f}_{n}$, as stated in Lemma \ref{lem:f_n_concentration_L1_gen}. This result is only valid in what we call the `slow adaptation regime' $\beta\in (1/2, 1]$. This is where we can use CLT-type arguments to connect the empirical and the expected loss minimization. Hence, it is an interesting open problem to analyze the rate of convergence of SRMM in the `fast adaptation regime' of $\beta\in [0,1/2]$, which is depicted as the grey region in Figure \ref{fig:diagram_rates}. Empirical loss in this regime will fast adapt to newly observed data, so minimizing it will capture more fast-paced, short-time-scale features from streaming data. We speculate that SRMM is more effective than PSGD for such a purpose. 
	\end{remark}

	\begin{remark}[Iterate stability and regularization]
		We give some remarks on the use of proximal regularization and diminishing radius in Algorithms \ref{algorithm:BSM-DR} and \ref{algorithm:BSM-PR}, and also why we need to assume i.i.d. data sampling for the case \ref{C3} of weakly convex surrogates with proximal regularization. 
		
		As we mentioned earlier in Section \ref{section:algorithm}, a key property of Algorithm \ref{algorithm:SMM} for cases \ref{C1}-\ref{C2} stated in Theorem \ref{thm:global_convergence} is the iterate stability $\lVert \param_{n}-\param_{n-1} \rVert=O(w_{n})$ (see Lemma \ref{lem:stability}), which was critically used in the SMM literature \cite{mairal2010online, mairal2013stochastic, mensch2017stochastic, lyu2020online, lyu2020online_CP}. It is crucial in deducing asymptotic convergence to stationary points stated in Theorem \ref{thm:global_convergence} by using Lemma \ref{lem:positive_convergence_lemma}. Since $w_{n}$'s are square summable (see \ref{assumption:A4-w_t}), it implies the  `weak iterate stability', $\sum_{n=1}^{\infty} \lVert \param_{n}-\param_{n-1} \rVert^{2}<\infty$. Our analysis shows that we can still retain the rate of convergence results stated in Theorems \ref{thm:rate_surrogate_gaps} and \ref{thm:rate_stationarity} if we only had this weak iterate stability. While in the classical block coordinate descent setting of non-stochastic optimization the weak iterate stability is sufficient to derive asymptotic stationarity of the iterates (see, e.g., \cite{grippo2000convergence}, \cite{xu2013block}), it does not seem to be immediate for the stochastic setting we consider in this work.

		The technique of diminishing radius constraints we used in Algorithm \ref{algorithm:BSM-DR} to handle multi-convex surrogates was first introduced in \cite{lyu2020convergence} to analyze convergence and complexity of cyclic BCD algorithms for minimizing block multi-convex functions on convex constraint sets. Enforcing additional trust region condition $\lVert \param_{n}-\param_{n-1}\rVert =O(w_{n})$ with diminishing radii $O(w_{n})$ bakes iterate stability directly into the algorithm, although that this auxiliary constraint is weak enough to retain asymptotic stability with respect to the original objective function needs to be argued (Lemma \ref{lem:asymptotic_stationarity}). In comparison, with proximal regularization in Algorithm \ref{algorithm:BSM-PR}, we were able to only show weak asymptotic stability of the form $\E[\sum_{n=1}^{\infty} \lVert \param_{n}-\param_{n-1} \rVert^{2}]<\infty$ (see Lemma \ref{lem:pos_variation} \textbf{(vi)}). This was assuming i.i.d. data sampling assumption in comparison to the more general Markovian data assumption for cases \ref{C1}-\ref{C2}. In order to handle Markovian dependence (especially in proving Lemma \ref{lem:pos_variation}), we use the technique of `conditioning on distant past' first introduced in \cite{lyu2020online}. It is to bound the error involving $\param_{n+1}$ by conditioning on the past $\sigma$-algebra $\mathcal{F}_{n-N}$ and using Markov chain mixing during the interval $[n-N, n]$. This requires to control $\lVert \param_{n} - \param_{n-N} \rVert$, which is difficult without a priori iterate stability with proximal regularization. However, when the data samples are i.i.d., then we can condition directly on $\mathcal{F}_{n}$ so that we avoid using iterate stability.

		\begin{remark}[Asymptotic stationarity and inexact surrogate minimization]
			We discuss the effect of using inexact minimization of block multi-convex surrogates on convergence rates. First, consider the case \ref{C1} in Theorem \ref{thm:global_convergence} together with identically zero optimality gaps (see \ref{assumption:A5_sufficient_surrogate_decay}). In this case, each $\param_{n}$ is the minimizer of the strongly convex surrogates $\bar{g}_{n}$ over $\Param$ so that $-\nabla \bar{g}_{n}(\param_{n})$ is in the normal cone of $\Param$ at $\param_{n}$. Hence Theorem \ref{thm:rate_surrogate_gaps} can be directly used to obtain \eqref{eq:thm_convergence_bd_f_t} and \eqref{eq:thm_convergence_bd_f}. In particular, $\lVert \nabla \bar{g}_{n}(\param_{n}) \rVert$ is identically zero when there are no boundary stationary points of $\bar{g}_{n}$.
			
			In all other cases covered in \ref{C1}-\ref{C3} in Theorem \ref{thm:global_convergence}, each estimate $\param_{n}$ is only an approximate minimizer of the averaged surrogate $\bar{g}_{n}$ over $\Param$ due to the combination of the following factors:  non-convexity of $\bar{g}_{n}$; the use of regularization in surrogate minimization in Algorithms \ref{algorithm:BSM-PR} and \ref{algorithm:BSM-DR};  possible nonzero optimality gap (see \ref{assumption:A5_sufficient_surrogate_decay}) in solving convex subproblems in surrogate minimization. Hence, $-\nabla \bar{g}_{n}(\param_{n})$ is not necessarily in the normal cone of $\Param$ at $\param_{n}$ in the general case. However, \eqref{eq:thm_convergence_bd_surrogate} shows that $-\nabla \bar{g}_{n}(\param_{n})$ should subsequentially converge to some vector in the normal cone of $\Param$ at $\param_{n}$ at rate $O((\sum_{k=1}^{n} w_{k})^{-1})$, and by \eqref{eq:thm_convergence_bd_f}, the bound on the rate of subsequential convergence to stationary points of $f$ over $\Param$ is slower at order $O((\sum_{k=1}^{n} w_{k})^{-1/2})$. We can ensure that these two convergence rates are achieved over the same subsequence. Hence, the asymptotic rate of convergence is unchanged when we relax strongly convex surrogates to weakly convex or block multi-convex surrogates.
		\end{remark}

	\end{remark}
	
	%For SGD methods with unconstrained nonconvex objective, bounds of order $O(\log n/\sqrt{n})$ are known \cite{sun2018markov, ward2019adagrad}. Comparing to our bound $O((\log n)^{2}/n^{1/4})$, it seems that the lost factor of $n^{-1/4}$ in our result seems to be a trade-off we made to handling constraints. To our best knowledge, we are not aware of any convergence rate bound for SGD in the constrained nonconvex case. On the contrary, SMM-type methods, including our Algorithm \ref{algorithm:SMM}, are inherently well-suited to handle constraints since convergence to stationary points over the constraint $\Param$ can be obtained by comparing the true gradients $\nabla f (\param_{n})$ with the surrogate gradient $\nabla \bar{g}_{n}(\param_{n})$ as in Theorem \ref{thm:nonconvex_function} \textbf{(ii)}. However, while in SGD one requires $\sum_{n=1}^{\infty} (\log n)^{2}w_{n}^{2}<\infty$ when $w_{n}$ are interpreted as the step sizes (see \cite{sun2018markov}), convergence of SMM requires $\sum_{n=1}^{\infty} w_{n}^{2}\sqrt{n}<\infty$ (see \cite{mairal2013stochastic}), which costs additional factor of $n^{-1/4}$. This is the same requirement we make for Theorem \ref{thm:nonconvex_function} (see \ref{assumption:w_t}). 

	\section{Applications and Experiments}
	\label{section:applications}

\subsection{Applications}  

In this subsection, we discuss some applications of our general results in the setting of online matrix and tensor factorization problems. 

\subsubsection{Double-averaging PSGD and its generalization}
\label{sec:double_avg}

In this section, we will apply our general framework of SRMM \eqref{eq:def_SRMM} to derive convergence results for variants of PSGD such as the `double-averaging PSGD' due to Nesterov and Shikhman \cite{nesterov2015quasi} as well as its generalization.

Suppose we have a prescribed weight sequence $(w_{n})_{m\ge 1}$ and hyperparameters $L,\lambda\ge 0$. Consider the following iterates 
\begin{align}\label{eq:two_way_PSGD_gen}
	\begin{cases}
		\overline{\nabla}_{n} &\leftarrow (1-w_{n}) 	\overline{\nabla}_{n-1} + w_{n} \nabla_{\param} \ell(\x_{n}, \param_{n-1}) \qquad (\overline{\nabla}_{0}:=\mathbf{0}) \\
		\bar{\param}_{n-1} &\leftarrow (1-w_{n}) 	\bar{\param}_{n-2} + w_{n} \param_{n-1} \qquad (\bar{\param}_{0}:=\param_{0}) \\
		\tilde{\param}_{n-1} &\leftarrow \frac{L}{L+\lambda} 		\bar{\param}_{n-1} + \frac{\lambda}{L+\lambda} \param_{n-1} \qquad (\widetilde{\param}_{0}:=\param_{0}) \\
		\param_{n} &\leftarrow \textup{Proj}_{\Param} \left(  \widetilde{\param}_{n-1} - \frac{1}{L+\lambda} \overline{\nabla}_{n}\right).
	\end{cases}
\end{align}
Setting $\lambda=0$ reduces \eqref{eq:two_way_PSGD_gen} to the following `double-averaging PSGD', which was first investigated by  Nesterov and Shikhman \cite{nesterov2015quasi}:
\begin{align}\label{eq:two_way_PSGD}
	\begin{cases}
		\overline{\nabla}_{n} &\leftarrow (1-w_{n}) 	\overline{\nabla}_{n-1} + w_{n} \nabla_{\param} \ell(\x_{n}, \param_{n-1}) \qquad (\overline{\nabla}_{0}:=\mathbf{0}) \\
		\bar{\param}_{n-1} &\leftarrow (1-w_{n}) 	\bar{\param}_{n-2} + w_{n} \param_{n-1} \qquad (\bar{\param}_{0}:=\param_{0}) \\
		\param_{n} &\leftarrow \textup{Proj}_{\Param} \left(  \bar{\param}_{n-1} - \frac{1}{L} \overline{\nabla}_{n}\right).
	\end{cases}
\end{align}

Compared to the standard PSGD update $\param_{n} \leftarrow \textup{Proj}_{\Param} \left(  \param_{n-1} - \frac{1}{L}\nabla_{\param} \ell(\x_{k}, \param_{k-1})  \right)$ with fixed step size $1/L$, the algorithm \eqref{eq:two_way_PSGD} uses the recursively averaged gradient $\overline{\nabla}_{n}$ as well as the recursively averaged iterate $\bar{\param}_{n-1}$. %Algorithms of the form of \eqref{eq:two_way_PSGD} was first investigated by  Nesterov and Shikhman \cite{nesterov2015quasi} under the name of `subgradient methods with double averaging'. We will refer to \eqref{eq:two_way_PSGD} as a \textit{double averaging method}. 
The update $\param_{n} \leftarrow \textup{Proj}_{\Param} \left(  \param_{n-1} - \frac{1}{L} \overline{\nabla}_{n}  \right)$, where one uses the averaged gradients in place of the stochastic gradient, is known as `dual-averaging'  \cite{nesterov2009primal, xiao2009dual}. For convex objectives in the online setting with i.i.d. observations, \cite{xiao2009dual} obtained a bound on the regret of this method of order $O(\sqrt{t})$. Compared to such dual-averaging methods, the double-averaging method \eqref{eq:two_way_PSGD} uses additional inline averaging of the iterates, which is known to be equivalent to using a momentum  \cite{defazio2020understanding}. 

\iffalse
due to the additional averaging with the latest estimate $\param_{n-1}$ 

Note that even though a constant stepsize of $1/L$ is employed, due to the recursive averaging,

Compared to \eqref{eq:two_way_PSGD}, \eqref{eq:two_way_PSGD_gen}

One can verify that $\lVert \param_{n}-\param_{n-1} \lVert = O(w_{n} \lVert \overline{\nabla }_{n} \rVert )$, so one can regard $w_{n}$ as the `effective stepsize' of the algorithm \eqref{eq:two_way_PSGD}. However, if $\lambda>0$ in  \eqref{eq:two_way_PSGD_gen}, 

\begin{align}
	\param_{n}-\param_{n-1} =(1-w_{n}) \bar{\param}_{n-2} - (1-w_{n}) \param_{n-1} - \frac{1}{L}\overline{\nabla}_{n}
\end{align}
\fi

In \cite{nesterov2015quasi}, Nestrove and Shikhman showed that for convex (possibly non-smooth) objective functions in the offline setting, the double-averaging method \eqref{eq:two_way_PSGD} generates a sequence of iterates that decreases the optimality gap in the objective value as $O(t^{-1/2})$. Moreover, such a rate of convergence is justified for the \textit{whole sequence} of test points for the first time for subgradient methods. However, to the author's best knowledge, the double-averaging method has not been analyzed beyond this setting. This is in contrast to the standard PSGD case, which has been analyzed for general non-smooth nonconvex objectives under i.i.d. \cite{davis2019stochastic} and Markovian \cite{alacaoglu2022convergence} data assumption. As an application of our general SRMM analysis, we will derive very general convergence and complexity results for the iterates \eqref{eq:two_way_PSGD_gen} (and hence for its specialization \eqref{eq:two_way_PSGD} as well). 

Assuming the per-sample loss function $f_{n}(\param):= \ell(\x, \param)$ is $L$-smooth for each data point $\x$, the following prox-linear surrogate (see Ex. \ref{ex:prox_linear}) is indeed a majorizing surrogate of $f_{n}$:
\begin{align}\label{eq:def_prox_linear_PSGD}
	g_{n}(\param) := \ell(\x_{n}, \param_{n-1}) + \langle \nabla_{\param} \ell(\x_{n}, \param_{n-1}),\, \param - \param_{n-1}  \rangle  + \frac{L}{2} \lVert \param- \param_{n-1} \rVert^{2}.
\end{align}
Now we claim that the generalized double-averaging scheme \eqref{eq:two_way_PSGD_gen} is in fact equivalent to the following SRMM iterate 
\begin{align}\label{eq:PSGD_SRMM1}
	\param_{n} &= \argmin_{\param\in \Param} \left( \bar{g}_{n}(\param) + \frac{\lambda}{2}\lVert \param-\param_{n-1} \rVert^{2} \right), 
\end{align}
where $\bar{g}_{n}(\param):=(1-w_{n}) \bar{g}_{n-1}(\param) + w_{n} g_{n}(\param)$ denotes the recursive average of the surrogates as in \eqref{eq:def_SRMM}. Indeed, letting  $\widetilde{\param}_{n-1}:=\frac{L}{L+\lambda} \bar{\param}_{n-1} + \frac{\lambda}{L+\lambda} \param_{n-1} $, one can easily see that \eqref{eq:PSGD_SRMM1} is equivalent to 
\begin{align}
	\param_{n} 
	&= \argmin_{\param\in \Param}\,\,  \langle \overline{\nabla}_{n},\, \param   \rangle + \frac{L+\lambda}{2} \lVert \param \rVert^{2} - (L+\lambda) \left\langle \param,\, \widetilde{\param}_{n-1}  \right\rangle \\
	&= \argmin_{\param\in \Param}\,\,  \left\lVert  \param - (\widetilde{\param}_{n-1} - \frac{1}{L+\lambda} \overline{\nabla}_{n} ) \right\rVert^{2}\\
	&= \textup{Proj}_{\Param} \left(  \widetilde{\param}_{n-1} - \frac{1}{L+\lambda} \overline{\nabla}_{n}\right). 
\end{align}

Interestingly, the above derivation also tells us that the double-averaging scheme \eqref{eq:two_way_PSGD} is equivalent to the SMM update $\param_{n} = \argmin_{\param\in \Param}  \bar{g}_{n}(\param)$ with the prox-linear surrogates in \eqref{eq:def_prox_linear_PSGD}. Using the additional proximal regularization with parameter $\lambda \ge 0$ as in \eqref{eq:PSGD_SRMM1} has the effect of additional averaging of the parameters, giving a constant weight $\frac{\lambda}{L+\lambda}$ to the most recent iterate $\param_{n-1}$ (see \eqref{eq:two_way_PSGD_gen}) instead of the possibly decaying weight $w_{n}$ (see \eqref{eq:two_way_PSGD}).

\iffalse
We claim  that the double-averaging method \eqref{eq:two_way_PSGD} is in fact equivalent to the corresponding SMM update $\param_{n} = \argmin_{\param\in \Param}  \bar{g}_{n}(\param)$. We will justify this claim by deriving the more general iterate  \eqref{eq:two_way_PSGD} from SRMM \eqref{eq:def_SRMM}. Namely, consider the following SRMM update 

where we now minimize the averaged surrogate $\bar{g}_{n}$ with and added proximal regularization term with a fixed coefficient $\lambda\ge 0$. 
Thus, the SRMM update \eqref{eq:PSGD_SRMM1} with prox-linear surrogates is equivalent to the following update: 

Clearly, setting $\lambda=0$ in \eqref{eq:two_way_PSGD_gen} leads to \eqref{eq:two_way_PSGD}. Hence \eqref{eq:two_way_PSGD} (resp., \eqref{eq:two_way_PSGD_gen}) can be viewed as a special instance of Algorithm \ref{algorithm:SMM} in case \ref{C1} (resp., \ref{C3}). 

\fi

Now that we have verified the generalized double-averaging PSGD \eqref{eq:two_way_PSGD_gen} is a special instance of SRMM \eqref{eq:def_SRMM} with prox-linear surrogate and proximal regularizer, we obtain the following general convergence and complexity results for the \eqref{eq:two_way_PSGD_gen} as well as \eqref{eq:two_way_PSGD}. 

\begin{corollary}[Convergence and complexity of double-averaging  PSGD]
	Theorems \ref{thm:global_convergence}, \ref{thm:rate_surrogate_gaps}, and \ref{thm:rate_stationarity} hold for the double-averaging PSGD \eqref{eq:two_way_PSGD} (case \ref{C1} with prox-linear surrogates) and its generalization \eqref{eq:two_way_PSGD_gen} (case \ref{C3} with prox-linear surrogates and proximal regularization). 
\end{corollary}

\iffalse

\iffalse
\begin{align}
	\param_{n} &= \argmin_{\param\in \Param} \left( \bar{g}_{n}(\param):=(1-w_{n}) \bar{g}_{n-1}(\param) + w_{n} g_{n}(\param) \right) \\
	&= \argmin_{\param\in \Param}\,\,  \langle \overline{\nabla}_{n},\, \param   \rangle + \frac{L}{2} \lVert \param \rVert^{2} - L\langle \param,\, \bar{\param}_{n-1} \rangle \\
	&= \argmin_{\param\in \Param}\,\,  \left\lVert  \param - (\bar{\param}_{n-1} - \frac{1}{L} \overline{\nabla}_{n} ) \right\rVert^{2}\\
	&= \textup{Proj}_{\Param} \left(  \bar{\param}_{n-1} - \frac{1}{L} \overline{\nabla}_{n}\right). 
\end{align}
\fi

\begin{align}
	\left\lVert  \param - (\bar{\param}_{n-1} - \frac{1}{L} \overline{\nabla}_{n} ) \right\rVert^{2} = \lVert \param-\bar{\param}_{n-1} \rVert^{2} - \frac{2}{L} \langle \param ,\, \overline{\nabla}_{n} \rangle + C
\end{align}
\fi

Next, our general framework enables us to consider a block coordinate version of the double-averaging PSGD \eqref{eq:two_way_PSGD}, which will be an example of Algorithm \ref{algorithm:SMM} in case \ref{C2}. For instance, instead of updating the entire parameter $\param_{n}$, we may only update its $i$th block $\theta_{n}^{(i)}$ where $i$ is chosen uniformly at random independently at each iteration:
\begin{align}\label{eq:two_way_PSGD_block}
	\begin{cases}
		i &\sim \textup{Uniform}(\{1,\dots,m\}) \\
		\overline{\nabla}_{n}^{(i)} &\leftarrow (1-w_{n}) 	\overline{\nabla}_{n-1}^{(i)} + w_{n} \nabla_{\theta^{(i)}} \ell(\x_{n}, \param_{n-1}) \qquad (\overline{\nabla}_{0}^{(i)}:=\mathbf{0}) \\
		\bar{\theta}_{n-1}^{(i)} &\leftarrow (1-w_{n}) 	\bar{\theta}_{n-2}^{(i)} + w_{n} \theta_{n-1}^{(i)} \qquad (\bar{\theta}_{0}^{(i)}:=\theta_{0}^{(i)}) \\
		\theta_{n}^{(i)} &\leftarrow \textup{Proj}_{\Theta^{(i)} \cap \{ \theta\,:\, \lVert \theta - \theta_{n-1}^{(i)} \rVert \le w_{n} \} } \left(  \bar{\theta}_{n-1}^{(i)} - \frac{1}{L} \overline{\nabla}_{n}^{(i)}\right).
	\end{cases}
\end{align}
Note that in \eqref{eq:two_way_PSGD_block}, we used the diminishing radius regularization to update $\theta_{n}^{(i)}$. Also notice that the above block version of  \eqref{eq:two_way_PSGD} has the computational advantage that one only needs to compute the partial stochastic gradient $\nabla_{\theta^{(i)}} \ell(\x_{k}, \param_{k-1})$ instead of the full stochastic gradient $\nabla_{\param} \ell(\x_{k}, \param_{k-1})$, which could be significant when the parameter dimension $p$ is large. For instance, here we can choose individual coordinates as single blocks so the per-iteration computational cost of \eqref{eq:two_way_PSGD_block} is of $O(1)$ plus the cost of sampling the new data point $\x_{n}$ and for the projection onto the convex set  $\Theta^{(i)} \cap \{ \theta\,:\, \lVert \theta - \theta_{n-1}^{(i)} \rVert \le w_{n} \}$.

As before, our general result immediately implies the following convergence and complexity result for the randomized block variant of the double-averaging PSGD \eqref{eq:two_way_PSGD_block}. 

\begin{corollary}[Convergence and complexity of double-averaging PSGD]
	Theorems \ref{thm:global_convergence}, \ref{thm:rate_surrogate_gaps}, and \ref{thm:rate_stationarity} hold for the randomized block variant  of the double-averaging PSGD \eqref{eq:two_way_PSGD_block} (case \ref{C2} with prox-linear surrogates and randomized coordinate descent)
\end{corollary}

We remark that the above corollary holds if we used cyclic block coordinate descent in \eqref{eq:two_way_PSGD_block} instead of randomly choosing block coordinates to update. 

\subsubsection{Proximal point empirical loss minimization} 
\label{sec:PPERM}

Consider the following iterates \eqref{eq:prox_ERM}, where we recursively update the empirical loss function $\bar{f}_{n}$ and minimize is proximal point modification $\tilde{f}_{n}$ over the parameter space $\Param$. 
\begin{align}\label{eq:prox_ERM}
	\begin{cases}
		\bar{f}_{n}(\param) &\leftarrow (1-w_{n}) \bar{f}_{n-1}(\param) + w_{n} \ell(\x_{n}, \param) \\
		\param_{n} &\in \argmin_{\param\in \Param} \,\left(  \tilde{f}_{n}(\param):=\bar{f}_{n}(\param) + \frac{\lambda}{2} \lVert \param-\param_{n-1} \rVert^{2} \right)
	\end{cases}
\end{align}
Under \ref{assumption:A1-ell_smooth}, each $\bar{f}_{n}$ is $L$-smooth so if $\lambda>L$, then the proximal point modification $\tilde{f}_{n}$ is $(\lambda-L)$-strongly convex (see Lemma \ref{lem:L_smooth_weak_convex}). The proximal point method has been used in empirical loss minimization problems with convex objectives \cite{frostig2015regularizing, lin2015accelerated}.

It is easy to see that \eqref{eq:prox_ERM} is an instance of case \ref{C3}, where we use the trivial surrogate $g_{n}(\cdot)=\ell(\x_{n},\cdot)$. Hence our general results yield that the proximal point ERM \eqref{eq:prox_ERM} in the general non-convex constrained setting with dependent data. 

\begin{corollary}[Convergence and complexity of Proximal Point ERM]
	Theorems \ref{thm:global_convergence}, \ref{thm:rate_surrogate_gaps}, and \ref{thm:rate_stationarity} hold for the proximal point ERM \eqref{eq:prox_ERM}.
\end{corollary}

\subsubsection{Online (Nonnegative) Matrix Factorization }
\label{subsection:OMF}

Consider the matrix factorization loss $f(W,H):=\lVert X - WH \rVert^{2} + \lambda \lVert H \rVert_{1}$, where $X\in \R^{p\times d}$ is a given data matrix to be factorized into the product of dictionary $W\in \R^{p\times r}$ and the code $H\in \R^{r\times d}$ with $\lambda\ge 0$ being the $L_{1}$-regularization parameter for $H$. Clearly $f$ is two-block multi-convex and differentiable with respect to $W$ with gradient $\nabla_{W} f(W,H) = (X-WH)H^{T}$. Fix compact and convex constraint sets $\Theta_{1}\subseteq \R^{p\times r}$ and $\Theta_{2}\subseteq \R^{r\times d}$.  Then $W\mapsto \nabla f(W,H)$ and $H\mapsto \nabla f(W,H)$ are both Lipschitz over the compact and convex set $\Theta_{1}\times \Theta_{2}$. 

Suppose we have a sequence of data matrices $(X_{n})_{n\ge 1}$ and a sequence of weights $(w_{n})_{n\ge 1}$. For each $n\ge 1$, the function $f_{n}:W\mapsto \ell(X_{n},W)= \inf_{H\in \Theta_{2}}\lVert X_{n}-WH \rVert^{2} + \lambda \lVert H \rVert_{1}$ is the minimum reconstruction error for factorizing $X_{n}$ using the dictionary matrix $W$. The corresponding empirical loss minimization problem is 
\begin{align}
	W_{n} \in \argmin_{W\in \Theta_{1}} \bar{f}_{n}(W):= (1-w_{n}) \bar{f}_{n-1}(W) + w_{n} \ell(X_{n}, W).  
\end{align}
If we have a target distribution $\pi$ for the data matrices $X_{n}$, then  the corresponding expected loss minimization problem is 
\begin{align}\label{eq:OMF_expectation}
	W \in \argmin_{W\in \Theta_{1}} \, f(W):=\E_{X\sim \pi}\left[ \ell(X, W) \right]. 
\end{align}
The latter is known as the online matrix factorization problem introduced in \cite{mairal2010online}. A well-known instance is the online nonnegative matrix factorization, which corresponds to \eqref{eq:OMF_expectation} with nonnegativity constraints for $W$ and $H$. 

In order to apply Algorithm \ref{algorithm:SMM} in this setting, denote  $H_{n}\in \argmin_{H\in \Theta_{2}}\lVert X_{n} - W_{n-1}H \rVert_{F}^{2} + \lambda \lVert H\rVert_{1}$, which is an optimal code for factorizing $X_{n}$ using the previous dictionary $W_{n-1}$. Then the function $g_{n}(W)=\lVert X_{n}-W H_{n} \rVert_{F}^{2} + \lambda \lVert H_{n} \rVert_{1}$ is a  surrogate of $f_{n}:W\mapsto \ell(X_{n},W)$ at $W_{n-1}$ and belongs to $\surr_{L'',0}(f_{n})$ for some $L''>0$ (see Ex. \ref{ex:cvx_variational_surrogate}.) A simple calculation shows that the resulting averaged surrogate function $\bar{g}_{n}$ becomes 
\begin{align}
	\bar{g}_{n}(W) =  \tr(W A_{n} W^{T})  - 2\tr(W B_{n})  + C_{n}, 
\end{align}
where the matrices $A_{n}, B_{n},C_{n}$ are recursively defined as 
\begin{align}
	&A_{n} = (1-w_{n})A_{n-1}+w_{n}H_{n}H_{n}^{T}; \quad B_{n} = (1-w_{n})B_{n-1}+w_{n}H_{n}X_{n}^{T}; \\ 
	&C_{n} = (1-w_{n})\bar{X}_{n-1}+w_{n} X_{n}X_{n}^{T}.
\end{align}
Thus Algorithm \ref{algorithm:SMM} in this case reduces to 
\begin{align}\label{eq:OMF1}
	\textit{Upon arrival of $X_{n}$:\quad }
	\begin{cases}
		H_{n} = \argmin_{H\in \Theta_{2}\subseteq \R^{r\times d}} \lVert X_{n} - W_{n-1}H \rVert_{F}^{2} + \lambda \lVert H \rVert_{1}\\
		A_{n} = (1-w_{n})A_{n-1}+w_{n}H_{n}H_{n}^{T} \\
		B_{n} = (1-w_{n})B_{n-1}+w_{n}H_{n}X_{n}^{T} \\
		W_{n} = \argmin_{W\in \Theta_{1} \subseteq \R^{p\times r}} \left(  \tr(W A_{n} W^{T})  - 2\tr(W B_{n})\right),
	\end{cases}
\end{align}
which is the online matrix factorization algorithm (OMF) proposed in  \cite{mairal2010online} for i.i.d. data matrices. 
Later in \cite{lyu2020online}, this algorithm was analyzed in a Markovian data setting. The following corollary is a direct consequence of our general results. 

\begin{corollary}[Rate of convergence of OMF]
	Theorems \ref{thm:rate_surrogate_gaps} and \ref{thm:rate_stationarity} hold for the online matrix factorization algorithm \eqref{eq:OMF1} (case \ref{C1}).
\end{corollary}

To the author's best knowledge, such a result was not known for online matrix factorization even under the i.i.d. data setting.

\vspace{0.2cm}
\subsubsection{Subsampled Online Matrix Factorization}
\label{sec:SOMF}

While the OMF algorithm \eqref{eq:OMF1} is efficient in handling matrices with a large number of columns $d$, its computational cost with respect to the data dimension $p$ (i.e., the number of rows) is not reduced. In \cite{mensch2017stochastic}, subsampled OMF was proposed to improve the computational efficiency with respect to $p$, by using only a random subsample of rows. For a preliminary version, consider the following variant of the online matrix factorization algorithm \ref{eq:OMF1}, which uses a random coordinate descent on subsampled rows of $W_{n}$:
\begin{align}\label{eq:OMF2}
	&\textit{Upon arrival of $X_{n}$:\qquad } \nonumber \\
	&\begin{cases}
		H_{n} = \argmin_{H\in \Theta_{2}\subseteq \R^{r\times d}} \lVert X_{n} - W_{n-1}H \rVert_{F}^{2} + \lambda \lVert H \rVert_{1}\\
		A_{n} = (1-w_{n})A_{n-1}+w_{n}H_{n}H_{n}^{T} \\
		B_{n} = (1-w_{n})B_{n-1}+w_{n}H_{n}X_{n}^{T} \\
		J \leftarrow \textup{Random subset of $\{1,\dots, p\}$}\\
		W_{n} = \argmin_{W\in \Theta_{1} \subseteq \R^{p\times r}} \left(  \tr(W A_{n} W^{T})  - 2\tr(W B_{n})\right) \\
		\hspace{1cm}  \textup{while freezing rows of $W_{n}$ not in $J$ to the rows of $W_{n-1}$}.
	\end{cases}
\end{align}
For instance, $J$ may be taken as the uniform subset of $\{1,\dots,p\}$ of a fixed size $p'<p$ or it may contain each index $i\in \{1,\dots,p\}$ independently with a fixed probability. All our main results in this paper (Theorems \ref{thm:global_convergence}, \ref{thm:rate_surrogate_gaps}, and \ref{thm:rate_stationarity}) apply to the OMF algorithm in \eqref{eq:OMF2}. 

However, note that the code computation for $H_{n}$ in \eqref{eq:OMF2} still involves solving a least squares problem with $p$ dimensional data matrix $X_{n}$. In order to fully reduce the dependency on $p$, one may also use row subsampling to compute only $p'$ rows of $H_{n}$. The approach used in \cite{mensch2017stochastic} was to replace $X_{n}$ with an averaged $p'$-dimenisonal matrix $\bar{X}_{n}^{(i)}$ assuming $X_{n}=X^{(i)}$, where a finite pool of data matrices $\mathfrak{X}=\{X^{(1)},\dots,X^{(l)}\}$ was assumed and $\bar{X}^{(i)}_{n}$ is a recursively defined matrix of size $p'\times d$ based on previous occurances of the matrix $X^{(i)}$. See \cite[Alg. 3]{mensch2017stochastic} for more details. 

An almost sure convergence to stationary points of subsampled OMF algorithm under i.i.d. data samples was shown in \cite{mensch2017stochastic}.  The analysis there is based on a general convergence result on stochastic approximate majorization-minimization algorithm \eqref{eq:def_SMM} using strongly convex $\eps$-approximate surrogate functions (see \cite[Prop. 3]{mensch2017stochastic}). A similar convergence result is retained by Theorem \ref{thm:global_convergence}, although the assumptions are slightly different. In addition, Theorems \ref{thm:rate_surrogate_gaps}-\ref{thm:rate_stationarity} also provide a rate of convergence, as stated in the following corollary.

\begin{corollary}[Rate of convergence of Subsampled OMF]
	Theorems \ref{thm:rate_surrogate_gaps} and \ref{thm:rate_stationarity} hold for the subsampled online matrix factorization algorithm \eqref{eq:OMF2} (case \ref{C1} with randomized block coordinate descent).
\end{corollary}

\vspace{0.2cm}
\subsubsection{Online CP-dictionary Learning }
\label{subsection:OCPDL}

Suppose we have  $\mathbf{X}_{1},\dots,\mathbf{X}_{N}\in \R_{\ge 0}^{I_{1}\times \dots \times I_{m}}$, which are $N$ observed $m$-mode tensor-valued signals. Consider  the following tensor-valued dictionary learning problem
\begin{align}\label{eq:NTF_intro_1}
	[\mathbf{X}_{1},\dots, \mathbf{X}_{b}] \approx [\mathbf{D}_{1},\dots, \mathbf{D}_{r}]\times_{m+1} H,
\end{align}
where $\times_{m+1}$ denotes the mode-$(m+1)$ tensor-matrix product and we impose the tensor dictionary atoms $\D_{1},\dots,\D_{r}\in \R^{I_{1}\times \dots \times I_{m} \times r}$ to be of rank-1 and $H\in \R^{R\times b}$ is called a code matrix. Equivalently, we assume that there exist \textit{loading matrices} $[U^{(1)},\dots,U^{(n)}]\in \R_{\ge 0}^{I_{1}\times r} \times \dots \times \R_{\ge 0}^{I_{n}\times r}$ such that 
\begin{align}\label{eq:def_OUT}
	[\mathbf{D}_{1},\dots, \mathbf{D}_{r}] &= \Out(U^{(1)},\dots,U^{(m)}) \\
	&:= \left[\bigotimes_{k=1}^{m} U^{(k)}(:,1),\, \bigotimes_{k=1}^{m} U^{(k)}(:,2),\, \dots \,, \bigotimes_{k=1}^{m} U^{(k)}(:,r)  \right] \in \R_{\ge 0}^{I_{1}\times \dots \times I_{m} \times r},
\end{align} 
where $U^{(k)}(:,j)$ denotes the $i^{\textup{th}}$ column of the $I_{k}\times r$ matrix $U^{(k)}$ and $\otimes$ denotes the outer product. Since we impose  a CANDECOMP/PARAFAC (CP) \cite{tucker1966some, harshman1970foundations, carroll1970analysis} structure for the tensor-valued dictionary $[\D_{1},\dots,\D_{r}]$, the above is called a CP-dictionary learning problem introduced in \cite{lyu2020online_CP}. When $m=1$, it reduces to the standard vector-valued dictionary learning problem. 

In \cite{lyu2020online_CP}, the following online CP-dictionary learning problem was proposed and analyzed. Fix compact and convex constraint sets for code and loading matrices $\mathcal{C}^{\textup{code}}\subseteq \R^{R\times b}$ and $\mathcal{C}^{(i)}\subseteq \R^{I_{i}\times r}$, $i=1,\dots,n$, respectively. Write  $\Param:=\mathcal{C}^{(1)}\times \cdots \times \mathcal{C}^{(n)}$. For each $\mathcal{X}\in\R_{\ge 0}^{I_{1}\times \dots\times I_{n}\times b}$, $\U:=[U^{(1)},\dots,U^{(n)}]\in \R^{I_{1}\times r}\times \dots \times \R^{I_{n}\times r}$, $H\in \R^{R\times b}$, define  
\begin{align}
	\ell(\mathcal{X},\U,H) &:= \lVert 
	\mathcal{X} - \Out(\U) \times_{n+1} H
	\rVert_{F}^{2} + \lambda \lVert H \rVert_{1},  \label{eq:def_loss_full} \\
	\ell(\mathcal{X},\U) &:= \inf_{H\in \mathcal{C}^{\textup{code}}}\,\, \ell(\mathcal{X},\U,H),\label{eq:def_loss}
\end{align}
where $\lambda\ge 0$  is a regularization parameter. Fix a sequence of non-increasing weights $(w_{t})_{t\ge 0}$ in $(0,1]$. Here $\mathcal{X}$ denotes a minibatch of $b$ tensors in $\R^{I_{1}\times \dots \times I_{n}}$, so minimizing $\ell(\mathcal{X}, \U)$ with respect to $\U$ amounts to fitting the CP-dictionary $\Out(\U)$ to the minibatch of $b$ tensors in $\mathcal{X}$. The corresponding empirical loss minimization problem is 
\begin{align}\label{eq:CPDL_empirical}
	\mathbf{U}_{n} \in \argmin_{\mathbf{U} \in \Param} \left( \bar{f}_{n}(\mathbf{U}):= (1-w_{n}) \bar{f}_{n-1}(\U) + w_{n} \ell(\mathcal{X}_{n}, \U) \right).  
\end{align}
If we have a target distribution $\pi$ for the data tensors $\mathcal{X}_{n}$, then  the corresponding expected loss minimization problem is 
\begin{align}\label{eq:CPDL_expected}
	\U\in \argmin_{\U\in \Param} \,\left( f(W):=\E_{X\sim \pi}\left[ \ell(\mathcal{X}, \U) \right] \right).
\end{align}

In order to solve \eqref{eq:CPDL_empirical} and \eqref{eq:CPDL_expected}, the following online CP-dictionary learning algorithm was proposed and analyzed in \cite{lyu2020online_CP}:
\begin{align}\label{eq:CPDL1}
	&\textit{Upon arrival of $\mathcal{X}_{n}$:\quad } \\
	&\begin{cases}
		H_{n} = \argmin_{H\in \mathcal{C}^{\textup{code}}\subseteq \R^{r\times b}} \ell(\mathcal{X}_{n}, \U_{n-1}, H) \\
		\bar{g}_{n}(\U) = (1-w_{n}) \bar{g}_{n-1}(\U)  + w_{n} \ell(\mathcal{X}_{n}, \U, H_{n}) \\
		\textup{for $i=1,\dots,m$}: \nonumber \\
		\quad U_{n}^{(i)} \in \argmin_{\substack{U\in \mathcal{C}^{ (i) } \\ \lVert U-U_{n-1}^{(i)} \rVert \le c'w_{n} } }  \,\, \bar{g}_{n}(U_{n}^{(1)}, \dots, U_{n}^{(i-1)}, U, U_{n-1}^{(i+1)}, \dots, U_{n-1}^{(m)}). 
	\end{cases}
\end{align}
An almost sure convergence to the stationary points of the above algorithm under the Markovain data setting was obtained in \cite{lyu2020online_CP}. In fact, the algorithm \eqref{eq:CPDL1} is a special case of Algorithm \ref{algorithm:SMM} with $m$-block multi-convex variational surrogate corresponding (see Example \ref{ex:multiconvex_variational}) to the function $\ell$ in \eqref{eq:def_loss_full} with $\mathbb{J}=\{ J_{1},\dots,J_{m} \}$, where $J_{i}$ denotes the coordinate block corresponding to the $i$th loading matrix in $\R^{I_{i}\times r}$ with cyclic block coordinate descent in Algorithm \ref{algorithm:BSM-DR}.

Our generalized algorithm and refined analysis improve the results in \cite{lyu2020online_CP} in multiple ways. First, Theorems \ref{thm:rate_surrogate_gaps} and \ref{thm:rate_stationarity} provide a rate of convergence results for the online CP-dictionary learning algorithm \eqref{eq:CPDL1}, which has not been known before.

\begin{corollary}[Rate of convergence of OCPDL]
	Theorems \ref{thm:rate_surrogate_gaps} and \ref{thm:rate_stationarity} hold for the online CP-dictionary learning algorithm \eqref{eq:CPDL1} (case \ref{C2}).
\end{corollary}

Second, due to the flexibility of using approximate surrogate functions, the convergence results of  \eqref{eq:CPDL1} hold under inexact code or factor matrix computation, following a similar argument as in \cite{mensch2017stochastic}. Third, one can only optimize a small number of subsampled rows when updating the factor matrices $U_{n}^{(i)}$ in \eqref{eq:CPDL1} by using a refined block coordinate structure (see Example \ref{ex:refine_block}). A similar idea of using subsampling for tensor CP-decomposition was also used recently in  \cite{kolda2020stochastic}. 

Lastly, we remark that random row subsampling can be utilized for the code computation for $H_{n}$ in \eqref{eq:CPDL1}, which essentially involves solving a least-squares problem with the mode-$(m+1)$ unfolding of $\mathcal{X}_{n}$, which has size $(I_{1}\times \cdots \times I_{m}) \times b$. Since the number of rows of a such matrix $I_{1}\times \cdots \times I_{m}$ can be very large, the computational gain in this approach should be significant. We believe it would be straightforward to adapt the approach in \cite{mensch2017stochastic} for subsampled online matrix factorization for a subsampled online CP-dictionary learning setting. For our theoretical results to apply, one only needs to verify that the resulting inexact code computation for $H_{n}$ yields $\eps_{n}$-approximate surrogates $\bar{g}_{n}$ that verifies the assumption \ref{assumption:A4-w_t}. We do not proceed with this line of research in the present paper.

	\subsection{Experiments}
	\label{sec:experiments}
	
	In this section, we provide experimental results of SRMM on two tasks - Network Dictionary Learning \cite{lyu2021learning} and Image classification with Deep Convolutional Neural Networks for the CIFAR-10 dataset \cite{krizhevsky2009learning}. 
	
	\subsubsection{Network Dictionary Learning}
	
	Network Dictionary Learning (NDL) \cite{lyu2020online, lyu2021learning} is the task of learning a fixed number of `latent motifs' from a large number of connected $k$-node subgraphs in a given large and possibly sparse networks. The learned latent motifs provide a concise description of a network's mesoscale structure and can be used for reconstructing and denoising networks (see \cite{lyu2021learning} for more details). NDL is naturally formulated as a nonconvex and constrained stochastic optimization problem, where the input data is a stream of $k$-node subgraphs sampled by the MCMC motif-sampling algorithm \cite{lyu2023sampling}. 
	
	The standard optimization algorithm for NDL is based on SMM for online nonnegative matrix factorization (see Sec. \ref{subsection:OMF}). We compare the performance of SMM (with adaptivity weights $w_{t}=1/t$) with that of various stochastic optimization algorithms --  PSGD, PSGD with momentum (PSGD-HB) (with step size schedule $\alpha_{t}=1/t$) and AdaGrad \cite{ward2020adagrad}.

	\begin{figure*}[h]
		\centering
		\includegraphics[width=1 \linewidth]{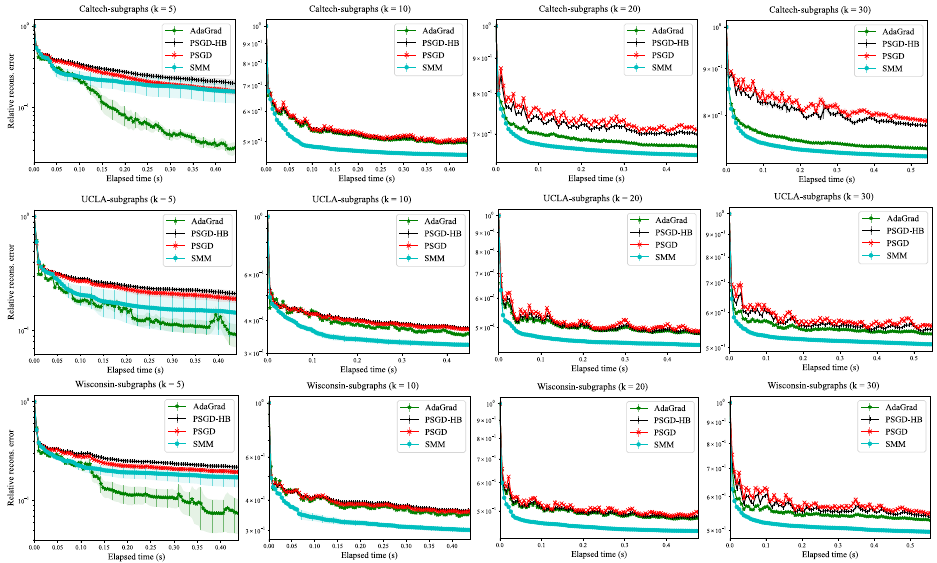}
		%\vspace{-0.5cm}
		\caption{Plot of reconstruction error vs. elapsed time for four algorithms for online NMF: AdaGrad, PSGD-Heavy Ball, PSGD, and SRMM. The data stream is a sequence of 4-node subgraph adjacency matrices sampled by an MCMC motif-sampling algorithm in \cite{lyu2023sampling} from three college Facebook networks \cite{traud2012social}. %Six consecutive Markovian samples of subgraphs are shown in each plot. 
			The shaded region represents one standard deviation from ten runs.
		}
		\label{fig:NDL_exp1}
	\end{figure*}

	We consider three facebook networks of schools  \dataset{Caltech}, \dataset{UCLA},  and \dataset{Wisconsin} from the Facebook100 dataset~\citep{traud2012social}, following a similar setup to~\citep{lyu2020online}.
	We then used the MCMC motif-sampling algorithm of~\citep{lyu2023sampling} to generate $300$ correlated subgraphs from the networks. This gives a stream of 300 $k\times k$ binary subgraph adjacency matrices, from which the latent motifs need to be learned by solving an online nonnegative matrix factorization problem. We ran each stochastic optimization algorithm on this streaming data (in the order that each subgraph are sampled) exactly once (e.g., one epoch). We used four subgraph sizes: $k\in \{ 5,10,20,30 \}$. 
	%For comparison, we used the stochastic majorization-minimization (SMM) algorithm from~\citep{mairal2010online,lyu2020online}, which is the state-of-the-art algorithm for ODL problems. 
	
	In Fig.~\ref{fig:NDL_exp1}, we see the convergence of all the algorithms with respect to the normalized reconstruction error, which is in line with our theoretical results. We observe that SMM shows significantly faster convergence in all cases except the smallest subgraph size $k=4$ is used, in which case AdaGrad seems to converge faster than all the other methods. Overall SMM seems to be providing robust performance for NDL with various choices of subgraph sizes. 
	
	%converges significantly faster than other methods, especially for the sequence of subgraphs from \dataset{Caltech}. The difference in speed of convergence between all methods is marginal for the \dataset{UCLA} and \dataset{Wisconsin}. We suspect that this different behavior is related to the fact that subgraphs in \dataset{Caltech} induced on random paths of $k=4$ nodes are more likely to contain more edges than those from the other two (much sparser) networks. 

	\subsubsection{Image classification using CNN}
	\label{sec:cifar10}
	
	Using DenseNet-121 \cite{huang2017densely} and ResNet-34 \cite{he2016deep}, we then consider the task of image classification on the standard CIFAR-10 dataset \cite{krizhevsky2009learning}. We compared six different optimization algorithms: (1) The generalized double-averaging scheme \eqref{eq:two_way_PSGD_gen} (labeled as SRMM);  (2) The double-averaging scheme \eqref{eq:two_way_PSGD} (labeled as SMM); (3) SGD-HB (SGD with momentum); (4) AdaGrad \cite{duchi2011adaptive}; (5) Adam \cite{kingma2015adam}; and (6) AMSGrad \cite{Reddi2018on}. We use these algorithms for a total of 130 epochs and report the train and test accuracies in Figure \ref{fig:cifar10}. We see that SRMM achieves the highest train and test accuracies (around 93$\%$) the most rapidly both for DenseNet and ResNet. SMM also shows a competitive performance against the benchmark methods for DenseNet, and outperforms SGD-HB for ResNet in test accuracy. 
	%shows final test accuracy around 90$\%$ but the convergence is noticeably slower than other methods except SGD-HB. 
	
	\begin{figure*}[h]
		\centering
		\includegraphics[width=1 \linewidth]{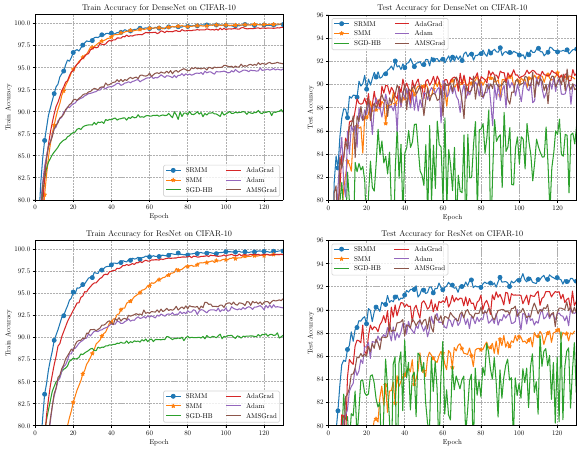}
		%\vspace{-0.5cm}
		\caption{Plot of training and test accuracies for CIFAR-10 image classification using DenseNet-121 and ResNet-34 with various optimization algorithms -- SRMM, SMM, SGD, AdaGrad, Adam, and AMSGrad. 
		}
		\label{fig:cifar10}
	\end{figure*}
	
	In the experiments reported in Figure \ref{fig:cifar10}, the hyperparameters are chosen as follows. (1) SRMM: $w_{n}=n^{-1/2}$ for $n\ge 1$, $L\in \{10, 1\}$, and $\lambda\in \{10, 1, 10^{-1}, 10^{-2}, 10^{-3}\}$; (2) SMM: $w_{n}=n^{-1/2}$ for $n\ge 1$ and  $L\in \{10, 1\}$; (3) SGD-HB: step sizes in $\{ 10^{-2}, 10^{-1}, 1, 10, 10^{2} \}$ and momentum parameter of $0.9$; (4) AdaGrad: stepsizes  in $\{ \textup{5e-4}, \textup{1e-3},  \textup{5e-3},  \textup{5e-2},  \textup{1e-1} \}$ and accumulator value 0; (5)-(6) Adam and AMSGrad: step sizes in $\{ \textup{1e-4}, \textup{5e-4},  \textup{1e-3},  \textup{5e-3},  \textup{1e-2} \}$, $\beta_{1}\in \{0.9, 0.99\}$, and $\beta_{2}\in \{0.99, 0.999\}$. For SRMM and SMM, we reset the iteration counter to zero at the beginning of each epoch in order to avoid the adaptivity weight $w_{n}$ becomes too small after many epochs.

	\vspace{0.2cm}
	
	\section{Preliminary analysis}
	\label{section:preliminary}
	
	In this section, we give some preliminary lemmas that we will use in our analysis in the following sections. Recall the definition of the optimality gaps $\Delta_{n}^{(i)}$  in \ref{assumption:A5_sufficient_surrogate_decay}. Also, we assume throughout that $g_{n}\in \surr_{L,\rho}^{\mathbb{L}}(f_{n},\param_{n-1},\eps_{n})$ for $n\ge 1$ for some fixed set $\mathbb{J}$ of coordinate blocks, $L>0$, and $\rho\in \R$.

	We first observe some regularity properties of the surrogate gradients $\nabla g_{n}$ and $\nabla \bar{g}_{n}$. 
	
	\begin{prop}\label{prop:g_grad_Lipschitz}
		Assume \ref{assumption:A1-ell_smooth} and \ref{assumption:A3-cvx_constraint}. Suppose that the data sequence $(\x_{n})_{n\ge 1}$ is contained in some compact subset $\mathfrak{X}_{0}\subseteq \mathfrak{X}$. Then the following hold for all $n\ge 1$:
		\begin{description}[itemsep=0.1cm] 
			\item[\textbf{(i)}] Both $\nabla g_{n}$ and $\nabla \bar{g}_{n}$ are $L$-Lipschitz over $\Param$; 
			\item[\textbf{(ii)}] $g_{n},\bar{g}_{n}$, $\nabla g_{n}$, and  $\nabla \bar{g}_{n}$ are uniformly bounded over $\Param$;
			\item[\textbf{(iii)}] There exists a constant $L'>0$ such that for all $n\ge 1$, $g_{n}, \bar{g}_{n}$ are $L'$-Lipschitz continuous over $\Param$ and $ \lVert \nabla \bar{g}_{n+1}(\param) - \nabla \bar{g}_{n}(\param) \rVert\le L'w_{n+1}$. 
		\end{description}
	\end{prop}
	
	\begin{proof}
		
		Denote  $f_{n}=\ell(\x_{n},\cdot)$ for $n\ge 1$. Note that $\nabla \ell(\cdot,\cdot)$ is $L$-Lipscthiz continuous (see \ref{assumption:A1-ell_smooth}). Also, $\nabla (g_{n}-f_{n})$ is $L$-Lipscthiz by Definition \ref{def:block_surrogate}. It follows that $\nabla g_{n}$ is also $L$-Lipschitz over $\Param$. This holds for all $n\ge 1$, so by an induction, $\nabla \bar{g}_{n}$ is also $L$-Lipschitz over $\Param$ for $n\ge 1$. This shows \textbf{(i)}.

		Note that by Definition \ref{def:block_surrogate}, $\nabla (g_{n}-f_{n})$ is $L$-Lipschitz continuous and has norm bounded by $\eps_{n}\le 1$  at $\param_{n-1}$. Since the parameter space $\Param$ is compact by \ref{assumption:A3-cvx_constraint}, it follws that $\nabla (g_{n}-f_{n})$ is uniformly bounded by some constant $L'>0$ over $\Param$. Also, by the assumption, the continuous map $(\x,\param)\mapsto \lVert \nabla \ell(\x,\param) \rVert$ over the compact domain $\mathfrak{X}_{0} \times \Param$  assumes bounded values. Hence $\nabla \ell$ is uniformly bounded over $\Param$ by \ref{assumption:A1-ell_smooth} and \ref{assumption:A3-cvx_constraint}. It follows that $\nabla f_{n}$ is uniformly bounded, so $\nabla g_{n}$ is also uniformly bonded. Then by induction, it follows that $\nabla \bar{g}_{n}$ is also uniformly bounded. Similarly, as $\lVert g_{n}(\param_{n-1})-f_{n}(\param_{n-1}) \rVert \le \eps_{n}\le 1$ and $f_{n}$ is uniformly bounded, it follows that $g_{n}$ is bounded over $\Param$. By induction, it follows that $\bar{g}_{n}$ is also uniformly bounded over $\Param$. 
		
		To show \textbf{(iii)}, we first let $L'>0$ be a uniform bound on $\lVert \nabla g_{n} \rVert$ and $\lVert \nabla \bar{g}_{n} \rVert$ over $\Param$. Fix $\param,\param'\in \Param$ and consider the linear curve $t\mapsto \gamma(t):=t\param' + (1-t)\param$ in $\Param$. Then 
		\begin{align}
			\lVert g_{n}(\param) -  g_{n}(\param') \rVert = \left| \int_{0}^{1} \langle \nabla g_{n}(\gamma(t)),\, \param-\param' \rangle \,dt \right| \le L' \lVert \param-\param' \rVert. 
		\end{align}
		This shows that $g_{n}$ is $L'$-Lipscthiz for all $n\ge 1$. From this, one also concludes that $\bar{g}_{n}$ is $L'$-Lipscthiz for all $n\ge 1$. 
		
		Lastly, note  that $\nabla \bar{g}_{n+1} = (1-w_{n+1}) \nabla \bar{g}_{n} + w_{n+1} \nabla g_{n}$. Hence by \textbf{(ii)}, there exists $L'\in (0,\infty)$  such that 
		\begin{align}\label{eq:g_bar_grad_lipschitz_variation}
			\lVert \nabla \bar{g}_{n+1}(\param) - \nabla \bar{g}_{n}(\param) \rVert \le w_{n+1} \lVert \nabla \bar{g}_{n}(\param) - \nabla g_{n}(\param) \rVert \le L' w_{n+1}
		\end{align}
		for all $\param\in \Param$. This shows \textbf{(iii)}.
	\end{proof}

	Define a sequence $(\bar{\eps}_{n})_{n\ge 1}$ recursively as 
	\begin{align}
		\bar{\eps}_{n} = (1-w_{n}) \bar{\eps}_{n-1} + w_{n} \eps_{n},\qquad \bar{\eps}_{0}=0,
	\end{align}
	where $(\eps_{n})_{n\ge 1}$ is the sequence of surrogate error tolerance in Algorithm \ref{algorithm:SMM}. Then 
	\begin{align}\label{eq:bar_eps_bound}
		\bar{\eps}_{n} = \sum_{k=1}^{n} w^{n}_{k} \eps_{k} 	\le w_{n} \sum_{k=1}^{n} \eps_{k},
	\end{align}
	where  $w^{n}_{k}$ is defined in \eqref{eq:ELF_closed_form} and the inequality above uses $w^{n}_{1}\le w^{n}_{2}\le \dots \le w^{n}_{n}=w_{n}$ under  \ref{assumption:A4-w_t}. Since $g_{n}-f_{n} \ge -\eps_{n}$ for all $n\ge 1$, it follows that 
	\begin{align}\label{eq:g_f_bar_eps_bound}
		\bar{g}_{n} -\bar{f}_{n} \ge -\bar{\eps}_{n} \ge -w_{n} \sum_{k=1}^{n} \eps_{k}.
	\end{align}

	\begin{lemma}\label{lem:stability}
		Let $(\param_{n})_{n\ge 1}$ be an output of Algorithm \ref{algorithm:SMM}. Assume \ref{assumption:A5_sufficient_surrogate_decay}. Then for all $n\ge 1$, the following hold:
		\begin{description}[itemsep=0.1cm]
			\item[(i)] (Forward Monotonicity) \quad $\bar{g}_{n}(\param_{n-1}) - \bar{g}_{n}(\param_{n}) \ge \Psi_{n}(\lVert \param_{n}-\param_{n-1} \rVert)  -\Delta_{n}$;

			\item[(ii)] (Stability of Estimates I) \quad Under case \ref{C1} in Theorem \ref{thm:global_convergence}, $\lVert \param_{n}-\param_{n-1}\rVert = O(w_{n})$;
			
			\item[(iii)] (Stability of Estimates II) \quad Under case \ref{C2} in Theorem \ref{thm:global_convergence}, $\lVert \param_{n}-\param_{n-1}\rVert \le  c'w_{n}$. 
		\end{description}
	\end{lemma}
	
	\begin{proof}
		First suppose Algorithm \ref{algorithm:BSM-PR} is used for \eqref{eq:alg1_param_update}. By the definition of the optimality gap $\Delta_{n}$ in \ref{assumption:A5_sufficient_surrogate_decay}, 
		\begin{align}
			\bar{g}_{n}(\param_{n}) + \Psi_{n}(\lVert \param_{n}-\param_{n-1} \rVert)  \le \bar{g}_{n}(\param_{n-1})  + \Delta_{n}. 
		\end{align}
		On the other hand, suppose Algorithm \ref{algorithm:BSM-DR} is used for \eqref{eq:alg1_param_update}. Consider the computation of $\param_{n}$ in Algorithm \ref{algorithm:BSM-DR}. Let $J_{1},\dots,J_{m}\in \mathbb{J}$ denote the coordinate blocks used in order, and let  $\param_{n}^{(1)},\dots,\param_{n}^{(m)}(=\param_{n})$ denote the outputs of \eqref{eq:BSM_factor_update_DR} after each block minimization in Algorithm \ref{algorithm:BSM-DR}.  Denote  $\param_{n}^{(0)}=\param_{n-1}$. Note that $\param_{n}^{(i)}$ is an approximate  minimizer of the convex function $\param\mapsto \bar{g}_{n}(\param) $ over the convex set $\Param^{J_{i}}$. Also note that $\param_{n}^{(i-1)}\in \Param^{J_{i}}$, by definition of the optimality gap $\Delta_{n}^{(i)}$ in \ref{assumption:A5_sufficient_surrogate_decay}, for $1\le i \le m$, 
		\begin{align}
			\bar{g}_{n}(\param_{n}^{(i)}) + \Psi_{n}(\lVert \param_{n}^{(i)}-\param_{n-1} \rVert)  \le \bar{g}_{n}(\param_{n}^{(i-1)}) + \Psi_{n}(\lVert \param_{n}^{(i-1)}-\param_{n-1} \rVert) + \Delta_{n}^{(i)},
		\end{align}
		Summing over all $i=1,\dots,m$ then gives $\bar{g}_{n}(\param_{n-1}) - \bar{g}_{n}(\param_{n}) \ge -\Delta_{n}$.  Recall that in this case we take the regularizer $\Psi_{n}(\lVert \param-\param_{n-1} \rVert)$ equals zero if $\lVert \param-\param_{n-1} \rVert\le c'w_{n}/m$ and $\infty$ otherwise. Hence this shows \textbf{(i)}. Note that \textbf{(iii)} is trivial by the search radius restriction in Algorithm \ref{algorithm:SMM}. 
		
		It remains to show \textbf{(ii)}. We show the assertion under zero surrogate optimality gap $\Delta_{n}\equiv 0$ (see \ref{assumption:A5_sufficient_surrogate_decay}) so that each $\param_{n}$ is the exact minimizer of the $\rho$-strongly convex function $\bar{g}_{n}$ over $\Param$. Indeed, by the second-order growth property (Lemma \ref{lem:second_order_growh_univariate}) and using $L'$-Lipschitz continuity of $g_{n}$ (see Lemma \ref{prop:g_grad_Lipschitz}), almost surely,
		\begin{align}
			\frac{\rho}{2} \lVert \param_{n}-\param_{n-1} \rVert^{2}&\le \bar{g}_{n}(\param_{n-1} ) - \bar{g}_{n}(\param_{n}) \\
			&= (1-w_{n} ) \left( \bar{g}_{n-1}(\param_{n-1} ) - \bar{g}_{n-1}(\param_{n} ) \right) + w_{n} \left( g_{n}(\param_{n-1} ) - g_{n}(\param_{n} ) \right) \\
			&\le w_{n} \left( g_{n}(\param_{n-1} ) - g_{n}(\param_{n} ) \right) \le w_{n} L' \lVert \param_{n}-\param_{n-1} \rVert. 
		\end{align}  
		This shows $\lVert \param_{n}-\param_{n-1} \rVert= O(w_{n})$, as desired. The proof for the general nonzero surrogate optimality gap can be found in Appendix \ref{sec:auxiliary_lemmas} (see Lemma \ref{lem:strongly_convex_surrogate_A6'}). 
	\end{proof}

	\begin{prop}\label{prop:Wt_bd}
		Let $(\param_{n})_{n\ge 1}$ be an output of Algorithm \ref{algorithm:SMM}. Assume \ref{assumption:A4-w_t}. Then for each $n\ge 0$, the following hold:	
		\begin{description}
			\item[(i)] $\bar{g}_{n+1}(\param_{n+1}) - \bar{g}_{n}(\param_{n}) \le w_{n+1}\left(  \ell(\x_{n+1},\param_{n}) - \bar{f}_{n}(\param_{n})  \right) +  w_{n}^{2}\left(\sum_{k=1}^{n+1}\eps_{k} \right) + \sum_{i=1}^{m} \Delta_{n+1}^{(i)}$.
			\vspace{0.2cm}
			\item[(ii)] $0  \le w_{n+1}\left( \bar{g}_{n}(\param_{n}) - \bar{f}_{n}(\param_{n}) \right)  \le    w_{n+1}\left( \ell(\x_{n+1},\param_{n}) - \bar{f}_{n}(\param_{n}) \right)  + \bar{g}_{n}(\param_{n}) - \bar{g}_{n+1}(\param_{n+1})  + w_{n}^{2}\left(\sum_{k=1}^{n+1}\eps_{k} \right)$.
			
		\end{description}
	\end{prop}

	\begin{proof}
		We begin by observing that 
		\begin{align}\label{eq:surrogate_loss_recursion0}
			\bar{g}_{n+1}(\param_{n}) &=  (1-w_{n+1}) \bar{g}_{n}(\param_{n}) + w_{n+1}\ell(\x_{n+1},\param_{n}) +w_{n+1} (g_{n+1}(\param_{n}) - \ell(\x_{n+1},\param_{n}) )
		\end{align}
		for all $t\ge 0$. Hence    
		\begin{align}\label{eq:surrogate_loss_recursion}
			&\bar{g}_{n+1}(\param_{n+1}) - \bar{g}_{n}(\param_{n}) \\
			&\quad =  \bar{g}_{n+1}(\param_{n+1}) - \bar{g}_{n+1}(\param_{n}) + \bar{g}_{n+1}(\param_{n}) - \bar{g}_{n}(\param_{n})\\
			&\quad=\bar{g}_{n+1}(\param_{n+1}) - \bar{g}_{n+1}(\param_{n}) + (1-w_{n+1})\bar{g}_{n}(\param_{n}) + w_{n+1}\ell(\x_{n+1},\param_{n}) - \bar{g}_{n}(\param_{n})\\
			&\quad=\bar{g}_{n+1}(\param_{n+1}) - \bar{g}_{n+1}(\param_{n}) + w_{n+1}(\ell(\x_{n+1},\param_{n})-\bar{f}_{n}(\param_{n})) + w_{n+1}(\bar{f}_{n}(\param_{n}) - \bar{g}_{n}(\param_{n})) \\
			&\hspace{5cm} + w_{n+1} (g_{n+1}(\param_{n}) - \ell(\x_{n+1},\param_{n}) ).
		\end{align} 
		Now note that $\bar{g}_{n+1}(\param_{n+1}) - \bar{g}_{n+1}(\param_{n})  \le  \sum_{i=1}^{m} \Delta_{n+1}^{(i)}$ by Lemma \ref{lem:stability} \textbf{(i)}. Also, recall that $\bar{g}_{n} - \bar{f}_{n} \ge -w_{n}\sum_{k=1}^{n} \eps_{k}$ from \eqref{eq:g_f_bar_eps_bound} and $0\le g_{n+1}(\param_{n})-\ell(\x_{n+1},\param_{n})\le \eps_{n+1}$ since $g_{n+1}$ is a multi-convex $\eps_{n}$-approximate  surrogate of $f_{n+1}=\ell(\x_{n+1},\cdot)$ at $\param_{n}$. Thus the inequalities in both \textbf{(i)} and \textbf{(ii)} follow using $w_{n+1}\le w_{n}$. 
	\end{proof}

	\iffalse
	\edit{Maybe not necessary. Argue everything for general $\bar{g}_{n}\ge \bar{f}_{n}$, $\bar{g}_{n} = (1-w_{n})\bar{g}_{n-1} + w_{n}g_{n}$, where $g_{n}(\param_{n-1})=f_{n}(\param)=\ell(\x_{n},\param_{n-1})$. Specialize for ERM (or $M$-estimator) later.  } 
	\begin{prop}\label{prop:ELM_ineq}
		Let $\param_{n}\in \argmin_{\param\in \Param} \bar{f}_{n}(\param)$ for $n\ge 1$. Then the following hold for all $n\ge 1$:
		\begin{align}
			\bar{f}_{n+1}(\param_{n+1}) - \bar{f}_{n}(\param_{n}) \le w_{n+1}\left(  \ell(\x_{n+1},\param_{n}) - \bar{f}_{n}(\param_{n})  \right).
		\end{align}
	\end{prop}

	\begin{proof}
		By the recursive definition of the empirical loss function $\bar{f}_{n}$, recall that $\bar{f}_{n+1}(\param_{n}) =  (1-w_{n+1}) \bar{f}_{n}(\param_{n}) + w_{n+1}\ell(\x_{n+1},\param_{n})$.  Hence  
		\begin{align}\label{eq:surrogate_loss_recursion_ERM}
			&\bar{f}_{n+1}(\param_{n+1}) - \bar{f}_{n}(\param_{n}) \\
			&\quad =  \bar{f}_{n+1}(\param_{n+1}) - \bar{f}_{n+1}(\param_{n}) + \bar{f}_{n+1}(\param_{n}) - \bar{f}_{n}(\param_{n})\\
			&\quad=\bar{f}_{n+1}(\param_{n+1}) - \bar{f}_{n+1}(\param_{n}) + (1-w_{n+1})\bar{f}_{n}(\param_{n}) + w_{n+1}\ell(\x_{n+1},\param_{n}) - \bar{f}_{n}(\param_{n})\\
			&\quad=\bar{f}_{n+1}(\param_{n+1}) - \bar{f}_{n+1}(\param_{n}) + w_{n+1}(\ell(\x_{n+1},\param_{n})-\bar{f}_{n}(\param_{n})).
		\end{align} 
		Now note that $\bar{f}_{n+1}(\param_{n+1}) - \bar{f}_{n+1}(\param_{n})  \le 0$ by the definition of $\param_{n}$ in the assertion. Thus the inequalities in both \textbf{(i)} and \textbf{(ii)} follow.
	\end{proof}
	
	\fi

	\begin{lemma}\label{lem:positive_convergence_lemma}
		Let $(a_{n})_{n\ge 0}$ and $(b_{n})_{n \ge 0}$ be sequences of nonnegative real numbers such that $\sum_{n=0}^{\infty} a_{n}b_{n} <\infty$. Then the following hold.
		
		\begin{description}
			\item[(i)] $\displaystyle \min_{1\le k\le n} b_{k} \le \frac{\sum_{k=0}^{\infty} a_{k}b_{k}}{\sum_{k=1}^{n} a_{k}}  = O\left( \left( \sum_{k=1}^{n} a_{k} \right)^{-1} \right)$.
			
			\vspace{0.2cm}
			\item[(ii)] Further assume $\sum_{n=0}^{\infty} a_{n} = \infty$ and  $|b_{n+1}-b_{n}|=O(a_{n})$. Then $\lim_{n\rightarrow \infty}b_{n} = 0$.
		\end{description}
	\end{lemma}
	
	\begin{proof}
		\textbf{(i)} follows from noting that
		\begin{align}
			\left( \sum_{k=1}^{n}a_{k} \right) \min_{1\le k \le n} b_{k}\le \sum_{k=1}^{n} a_{k}b_{k} \le  \sum_{k=1}^{\infty} a_{k}b_{k} <\infty.
		\end{align}
		The proof of \textbf{(ii)} is omitted and can be found in \cite[Lem. A.5]{mairal2013stochastic}.
	\end{proof}

	\vspace{0.2cm}
	\section{Key lemmas and proofs of main results}
	\label{section:key_lemmas}
	
	In this section, we state all key lemmas without proofs and derive the main results (Theorems \ref{thm:global_convergence}, \ref{thm:rate_surrogate_gaps}, and \ref{thm:rate_stationarity}) assuming them. The key lemmas stated in this section will be proved in the subsequent sections.

	\subsection{Key Lemmas}
	
	In this subsection, we state all key lemmas that are sufficient to derive the main results in this paper.

	First, Lemma \ref{lem:f_n_concentration_L1_gen} states some general concentration inequalities of recursively defined functions similar to the empirical and the surrogate loss functions in \eqref{eq:def_ELM} and \eqref{eq:def_SRMM}. One may regard them as the classical Glivenko-Cantelli theorem for a general weighting scheme. We state the lemma in a self-contained manner so that it may be more convenient to be used for other purposes. 
	
	\begin{lemma}[Uniform concentration of parameterized empirical observables]\label{lem:f_n_concentration_L1_gen}
		Fix compact subsets $\mathcal{X}\subseteq \R^{q}$, $\Param\subseteq \R^{p}$ and a bounded Borel measurable function $\psi:\mathcal{X}\times \Param\rightarrow \R^{r}$. Let $(\x_{n})_{n\ge 1}$ denote a sequence of points in $\mathcal{X}$ such that $\x_{n}=\varphi(X_{n})$ for $n\ge 1$, where $(X_{n})_{n\ge 1}$ is a Markov chain on a state space $\Omega$ and $\varphi:\Omega \rightarrow \mathcal{X}$ is a measurable function. Fix a sequence of weights $w_{n}\in (0,1]$, $n\ge 0$ and define functions $\bar{\psi}(\cdot):= \E_{\x\sim \pi}\left[ \psi(\x,\cdot) \right] $ and $\bar{\psi}_{n}:\Param\rightarrow \R^{r}$ recursively as $\bar{\psi}_{0}\equiv \mathbf{0}$ and 
		\begin{align}
			\bar{\psi}_{n}(\cdot) = (1-w_{n})\bar{\psi}_{n-1}(\cdot) + w_{n} \psi(\x_{n}, \cdot).
		\end{align}
		Assume the following: 
		\begin{description}
			\item[(a1)] The Markov chain $(X_{n})_{n\ge 1}$ mixes exponentially fast to its unique stationary distribution and the stochastic process $(\x_{n})_{n\ge 1}$ on $\mathcal{X}$ has a unique stationary distribution $\pi$. 
			\item[(a2)] $w_{n}$ is non-increasing in $n$ and $w_{n}^{-1} - w_{n-1}^{-1}\le 1$ for all sufficiently large $n\ge 1$.
		\end{description}
		
		\noindent Then there exists a constant $C>0$ such that for all $n\ge 1$,
		\begin{align}\label{eq:lem_f_fn_bd_gen}
			\sup_{\param\in \Param} \left\lVert \bar{\psi}(\param)- \E[\bar{\psi}_{n}(\param)] \right\rVert  \le Cw_{n}, \quad 	\E\left[ \sup_{\param\in \Param} \left\lVert \bar{\psi}(\param) - \bar{\psi}_{n}(\param) \right\rVert \right] \le C w_{n} \sqrt{n}.
		\end{align}
		Furthermore, if $w_{n}\sqrt{n}=O(1/(\log n)^{1+\eps})$ for some $\eps>0$, then $\sup_{\param\in \Param} \left\lVert \bar{\psi}(\param) - \bar{\psi}_{n}(\param) \right\rVert\rightarrow 0$ as $t\rightarrow \infty$ almost surely. 
	\end{lemma}
	
	\iffalse
	
	\begin{lemma}\label{lem:f_n_concentration_L1_gen}
		Suppose $w_{n}\in (0,1]$, $n\ge 1$ are non-increasing and satisfy $w_{n}^{-1} - w_{n-1}^{-1}\le 1$ for all sufficiently large $n\ge 1$. Fix a continuous and bounded function $\psi:\mathfrak{X}\times \Param \rightarrow \R^{d}$. Define functions $\bar{\psi}(\cdot):= \E_{\y\sim \pi}\left[ \psi(\varphi(\y),\cdot) \right] $ and $\bar{\psi}_{n}:\Param\rightarrow \R^{d}$ recursively as $\bar{\psi}_{0}\equiv \mathbf{0}$ and 
		\begin{align}
			\bar{\psi}_{n}(\cdot) = (1-w_{n})\bar{\psi}_{n-1}(\cdot) + w_{n} \psi(\x_{n}, \cdot).
		\end{align}
		Then under Assumptions \ref{assumption:A1-ell_smooth}-\ref{assumption:A3-cvx_constraint}, there exists a constant $C>0$ such that for all $n\ge 1$,
		\begin{align}\label{eq:lem_f_fn_bd_gen}
			\sup_{\param\in \Param} \left\lVert \bar{\psi}(\param)- \E[\bar{\psi}_{n}(\param)] \right\rVert  \le Cw_{n}, \quad 	\E\left[ \sup_{\param\in \Param} \left\lVert \bar{\psi}(\param) - \bar{\psi}_{n}(\param) \right\rVert \right] \le Cw_{n}\sqrt{n}.
		\end{align}
		Furthermore, if $w_{n}\sqrt{n}=O(1/(\log n)^{1+\eps})$ for some $\eps>0$, then $\sup_{\param\in \Param} \left\lVert \bar{\psi}(\param) - \bar{\psi}_{n}(\param) \right\rVert\rightarrow 0$ as $t\rightarrow \infty$ almost surely. 
	\end{lemma}
	
	\fi
	
	\noindent We remark that a similar result appeared in  \cite[Lem B.7]{mairal2013stochastic}, which states the second inequality in \eqref{eq:lem_f_fn_bd_gen} as well as the last almost sure convergence statement under the stronger condition of $\sum_{n=1}^{\infty} w_{n}=\infty$ and $\sum_{n=1}^{\infty} w_{n}^{2}\sqrt{n}<\infty$. We will refer to the latter condition as the `strong square-summability' condition. These conditions were necessary in order to use Lemma \ref{lem:positive_convergence_lemma} \textbf{(ii)} to deduce the almost sure convergence. We give a direct argument using Borel-Cantelli lemma without this additional condition, for which the weaker condition of $w_{n}\sqrt{n}=O(1/(\log n)^{1+\eps})$ for some $\eps>0$ is sufficient.

	Next, Lemma \ref{lem:pos_variation} states a series of finite variation statements that provide a basis for the forthcoming arguments. Most of the statements also appeared in the literature (see, e.g., \cite{mairal2010online, mairal2013stochastic, lyu2020online, lyu2020online_CP}) in some special cases, where the strong square summability condition was used. We give an improved argument only using the square summability condition except for the last item. 
	
	\begin{lemma}\label{lem:pos_variation}
		Let $(\param_{n})_{n\ge 1}$ be an output of Algorithm \ref{algorithm:SMM} under any of the three cases \ref{C1}-\ref{C3} in Theorem \ref{thm:global_convergence}.  Suppose \ref{assumption:A1-ell_smooth}-\ref{assumption:A4-w_t}. The following hold:
		\begin{description}
			\item[(i)] $\displaystyle \sum_{n=1}^{\infty} \E\left[ w_{n+1}\left( \ell(\x_{n+1},\param_{n}) - \bar{f}_{n}(\param_{n}) \right) \right]^{+}  <\infty $; 
			
			\item[(ii)] $\displaystyle 	\sum_{n=0}^{\infty}  \left( \E\left[ \bar{g}_{n+1}(\param_{n+1}) - \bar{g}_{n}(\param_{n}) \right]\right)^{+} <\infty$;
			
			\item[(iii)] $\displaystyle \E\left[  \sum_{n=1}^{\infty} w_{n}\left| \bar{g}_{n}( \param_{n}) - \bar{f}_{n}(\param_{n}) \right| \right] <\infty$;
			
			\item[(iv)] $\displaystyle \sum_{n=0}^{\infty}  w_{n+1}\left| \E\left[  f(\param_{n}) - \bar{f}_{n}(\param_{n})   \right] \right| <\infty$;
			
			\item[(v)] Suppose the optional condition in \ref{assumption:A4-w_t} holds. Then $\displaystyle \E\left[ \sum_{n=0}^{\infty}  w_{n+1}  \left| f(\param_{n}) - \bar{f}_{n}(\param_{n}) \right| \right]   <\infty$;
			
			\item[(vi)]  $\E\left[ \sum_{n=1}^{\infty}  \rVert \param_{n}-\param_{n-1} \rVert^{2} \right] < \infty$.
		\end{description}
	\end{lemma}
	
	The following lemma is one of the key innovations in the present work, which allows us to obtain convergence rate bounds stated in Theorems \ref{thm:rate_surrogate_gaps} and \ref{thm:rate_stationarity}. Roughly speaking, it gives gradient versions of the finite variation statements in Lemma \ref{lem:pos_variation}. 
	
	\begin{lemma}[Variation of gradients]\label{lem:gradient_finite_sum}
		Let $(\param_{n})_{n\ge 1}$ be the output of Algorithm \ref{algorithm:SMM} under any of the three cases \ref{C1}-\ref{C3} in Theorem \ref{thm:global_convergence}.  Suppose \ref{assumption:A1-ell_smooth}-\ref{assumption:A4-w_t}. The following hold:
		\begin{description}[itemsep=0.2cm]
			\item[(i)] $\displaystyle \E\left[	\sum_{n=1}^{\infty} w_{n} \,\left(  \lVert \nabla \bar{g}_{n}(\param_{n})-\nabla \bar{f}_{n}(\param_{n}) \rVert^{2} +  \left\lVert	\E\left[  \nabla \bar{f}_{n}(\param_{n}) \right] - \E\left[ \nabla f(\param_{n}) \right] \right\rVert^{2}  	\right) \right] <\infty.$

			\item[(ii)] If in addition the optional condition in \ref{assumption:A4-w_t} holds, then
			\begin{align}
				\E\left[	\sum_{n=1}^{\infty} w_{n+1} \,\left( \lVert \nabla \bar{g}_{n}(\param_{n})-\nabla \bar{f}_{n}(\param_{n}) \rVert^{2} + \lVert \nabla \bar{f}_{n}(\param_{n})-\nabla f(\param_{n}) \rVert^{2} + \lVert \nabla \bar{g}_{n}(\param_{n})-\nabla f(\param_{n}) \rVert^{2}\right) \right] <\infty.
			\end{align}
			
		\end{description}
	\end{lemma}
	
	Based on Lemmas \ref{lem:f_n_concentration_L1_gen},  \ref{lem:pos_variation}, and \ref{lem:gradient_finite_sum}, one can deduce Theorem \ref{thm:rate_surrogate_gaps} as well as Theorem \ref{thm:global_convergence} for case \ref{C1}. A nice property of Algorithm \ref{algorithm:SMM} with convex surrogates $g_{n}$ with identically zero regularization ($\Psi_{n}\equiv 0$ in \eqref{eq:alg1_param_update}) and optimality gaps (see \ref{assumption:A5_sufficient_surrogate_decay}) is that the iterates $\param_{n}$ are exact minimizers of the convex averaged surrogates $\bar{g}_{n}$ over $\Param$. Hence $-\nabla \bar{g}_{n}(\param_{n})$ lies in the normal cone of $\Param$ at $\param_{n}$ for each $n\ge 1$. Thus, in order for asymptotic stationarity with respect to the empirical loss $\bar{f}_{n}$ or the expected loss $f$, one only needs to show that every convergent subsequence of $\lVert \nabla \bar{g}_{n}(\param_{n}) - \nabla h_{n}(\param_{n}) \rVert$ vanish almost surely, where $h_{n}\in \{ \bar{f}_{n}, f \}$  for $n\ge 1$. 
	
	However, when $\bar{g}_{n}$ is (approximately) minimized with proximal regularization (Algorithm \ref{algorithm:BSM-PR}) or within a diminishing radius (Algorithm \ref{algorithm:BSM-DR}), $-\nabla \bar{g}_{n}(\param_{n})$ may be close to but not within the normal cone of $\Param$ at $\param_{n}$. For instance, even though each convex sub-problem in Algorithm \ref{algorithm:BSM-DR} is exactly solved, the technique of radius restriction in Algorithm \ref{algorithm:BSM-DR} introduces additional radius constraints so that $-\nabla \bar{g}_{n}(\param_{n})$ may be normal to the trust region boundary, which is the sphere of distance $c'w_{n}/m$ centered at $\param_{n-1}$; Similarly, using proximal regularization in Algorithm \ref{algorithm:BSM-PR} tilts the surrogate gradient.  In order to handle these issues, in Lemma \ref{lem:asymptotic_stationarity},  we will show that the sequence $\param_{n}$, $n\ge 1$  still verifies stationarity for $\bar{g}_{n}$ in an asymptotic sense. For this, we need to take `convergent subsequences` of the averaged surrogate functions $\bar{g}_{n}$, which can be easily done under the compact parameterization assumption in \ref{assumption:A7_param_surr}.
	
	\begin{lemma}[Asymptotic Surrogate Stationarity]\label{lem:asymptotic_stationarity}
		Assume \ref{assumption:A1-ell_smooth}-\ref{assumption:A4-w_t} and \ref{assumption:A7_param_surr} and let $(\param_{n})_{n\ge 1}$ be an output of Algorithm \ref{algorithm:SMM} under any of the three cases \ref{C1}-\ref{C3} in Theorem \ref{thm:global_convergence}. Assume that $\Delta_{n}=o(\lVert \param_{n}-\param_{n-1} \rVert)$. Let $(t_{k})_{k\ge 1}$ be any sequence such that $\param_{t_{k}}$ and $\bar{g}_{t_{k}}$ converges almost surely. Then $\param_{\infty}=\lim_{k\rightarrow \infty} \param_{t_{k}}$ is almost surely a stationary point of $\bar{g}_{\infty}:=\lim_{k\rightarrow\infty} \bar{g}_{t_{k}}$ over $\Param$. 
	\end{lemma}

	Finally, Lemma \ref{lem:finite_variation_surr} is the last ingredient for proving Theorem \ref{thm:rate_stationarity}, which states that the optimality for the averaged surrogates $\bar{g}_{n}$ is summable with weights $w_{n}$. 
	
	\begin{lemma}[Finite variation of surrogate optimality] \label{lem:finite_variation_surr}
		Let $(\param_{n})_{n\ge 1}$ be an output of Algorithm \ref{algorithm:SMM} under any of the three cases \ref{C1}-\ref{C3} in Theorem \ref{thm:global_convergence}. Assume \ref{assumption:A1-ell_smooth}-\ref{assumption:A4-w_t}. Then  for each initialization $\param_{0}$, we have
		\begin{align}
			\E\left[ 	\sum_{n=1}^{\infty} w_{n+1} \left| -  \inf_{\param\in \Param} \left\langle \nabla \bar{g}_{n}(\param_{n}),\, \frac{\param - \param_{n} }{\lVert \param - \param_{n}\rVert}\right\rangle \right|   \right]< \infty.
		\end{align}
		
	\end{lemma}

	\vspace{0.2cm}
	\subsection{Proof of Theorem \ref{thm:rate_surrogate_gaps}}
	
	In this subsection, we will assume Lemmas \ref{lem:f_n_concentration_L1_gen}, \ref{lem:pos_variation}, and \ref{lem:gradient_finite_sum} and prove Theorem \ref{thm:rate_surrogate_gaps}.
	
	\begin{proof}[\textbf{Proof of Theorem \ref{thm:rate_surrogate_gaps}}] 
		Suppose \ref{assumption:A1-ell_smooth}-\ref{assumption:A4-w_t} and any of the cases \ref{C1}-\ref{C3} stated in Theorem \ref{thm:global_convergence}. By Lemma \ref{lem:pos_variation} \textbf{(iii)} and Lemma \ref{lem:gradient_finite_sum} \textbf{(i)}, we get 
		\begin{align}\label{eq:empirical_loss_min_finite_sum_pf}
			\E\left[ \sum_{n=1}^{\infty} w_{n+1} \left( \left|  \bar{g}_{n}(\param_{n}) - \bar{f}_{n}(\param_{n})  \right|  + \lVert \nabla \bar{g}_{n}(\param_{n})-\nabla \bar{f}_{n}(\param_{n}) \rVert^{2}    \right)   \right]  <\infty. 
		\end{align}
		Recall that $\E[|X|]<\infty$ implies $|X|<\infty$ almost surely for any random variable $X$.  It follows that the summation above is almost surely finite. Then \eqref{eq:thm_nonconvex_empirical_rate} follows from  Lemma \ref{lem:positive_convergence_lemma}.

		Next, we show  \eqref{eq:thm_nonconvex_expected_variation}.  First we deduce some intermediate bounds. By Lemma \ref{lem:pos_variation} \textbf{(iii)}-\textbf{(iv)} and Jensen's inequality, we get 
		\begin{align}\label{eq:thm2_pf_variation00}
			&\E\left[ \sum_{n=1}^{\infty} w_{n} \left| \E[\bar{g}_{n}(\param_{n})]  -\E[f(\param_{n})]   \right|  \right] \\
			&\qquad \le 	\E\left[ \sum_{n=1}^{\infty} w_{n}\left(  \E\left[  \left| \bar{g}_{n}(\param_{n}) -  \bar{f}_{n}(\param_{n})   \right| \right]  +  \left| \E[\bar{f}_{n}(\param_{n})]  -\E[f(\param_{n})]   \right|   \right) \right] <\infty.
		\end{align}
		Similarly, using Lemma \ref{lem:gradient_finite_sum} \textbf{(i)}, we can deduce 
		\begin{align}\label{eq:thm2_pf_variation0}
			\sum_{n=1}^{\infty} w_{n} \,  \left\lVert	\E\left[  \nabla \bar{g}_{n}(\param_{n}) \right] - \E\left[ \nabla f(\param_{n}) \right] \right\rVert^{2}  	<\infty.
		\end{align}
		
		Now recall that by Lemma \ref{lem:f_n_concentration_L1_gen}, there exists a constant $C>0$ such that for all $n\ge 1$, 
		\begin{align}\label{eq:thm2_pf_variation1}
			\sup_{\param\in \Param} \left\lvert  f(\param) -  \E[  \bar{f}_{n}(\param)] \right\rvert  \le Cw_{n},\qquad 	\sup_{\param\in \Param} \left\lVert \nabla f(\param) -  \E[ \nabla \bar{f}_{n}(\param)] \right\rVert  \le Cw_{n}.
		\end{align}
		Using \eqref{eq:thm2_pf_variation00} and the first bound in \eqref{eq:thm2_pf_variation1}, we get 
		\begin{align}
			\left\lvert \E[ \bar{g}_{n}(\param_{n})] -  f(\param_{n}) \right\rvert^{2} &\le   \left( \lvert  \E[ \bar{g}_{n}(\param_{n})]- \E[ \bar{f}_{n}(\param_{n})] \rvert   + \sup_{\param\in \Param} \lvert \E[ \bar{f}_{n}(\param)] -  f(\param) \rvert  \right)^{2} \\
			&\le \left\lvert \E[ \bar{g}_{n}(\param_{n})]- \E[ \bar{f}_{n}(\param_{n})] \right\rvert^{2} + Cw_{n+1} \left\lvert \E[ \bar{g}_{n}(\param_{n})]- \E[ \bar{f}_{n}(\param_{n})] \right\rvert + C^{2}w_{n+1}^{2}.
		\end{align}
		Recalling that $\bar{g}_{n}$ and $\bar{f}_{n}$ are uniformly bounded over $\Param$ by \ref{assumption:A1-ell_smooth}, using \eqref{eq:thm2_pf_variation00} we get 
		\begin{align}\label{eq:thm2_pf_variation111}
			\E\left[ \sum_{n=1}^{\infty}	w_{n+1}\left\lvert \E[ \bar{g}_{n}(\param_{n})] -  f(\param_{n}) \right\rvert^{2}  \right]<\infty.
		\end{align}
		An identical argument using \eqref{eq:thm2_pf_variation0}  as wel as  the second bound in \eqref{eq:thm2_pf_variation1} shows 
		\begin{align}\label{eq:thm2_pf_variation11}
			\E\left[ \sum_{n=1}^{\infty}	w_{n+1} \left\lVert \E[ \nabla \bar{g}_{n}(\param_{n})] -  \nabla f(\param_{n}) \right\rVert^{2}\right] <\infty. 
		\end{align}
		Combining the above two bounds and using  Lemma \ref{lem:positive_convergence_lemma} give \eqref{eq:thm_nonconvex_expected_variation}.
		
		It remains to show  \eqref{eq:thm_nonconvex_empirical_variation} and \eqref{eq:thm_nonconvex_expected_rate}. Suppose the optional condition in \ref{assumption:A4-w_t} holds. Then by Lemma \ref{lem:pos_variation} \textbf{(v)} and Lemma \ref{lem:gradient_finite_sum} \textbf{(ii)}, we deduce 
		\begin{align}\label{eq:expected_loss_min_finite_sum_pf}
			\E\left[ \sum_{n=1}^{\infty} w_{n+1} \left( \left|  \bar{g}_{n}(\param_{n}) - f(\param_{n})  \right|  + \lVert \nabla \bar{g}_{n}(\param_{n})-\nabla f(\param_{n}) \rVert^{2}    \right)   \right]  <\infty. 
		\end{align}
		Then \eqref{eq:thm_nonconvex_expected_rate} follows from Lemma \ref{lem:positive_convergence_lemma}.

		Lastly, recall that by Lemma \ref{lem:f_n_concentration_L1_gen} and see \ref{assumption:A4-w_t}, there exists a constant $C>0$ such that for all $n\ge 1$, 
		\begin{align}\label{eq:thm2_pf_variation2}
			\E\left[ \sup_{\param\in \Param} \left\lvert  f(\param) -    \bar{f}_{n}(\param) \right\rvert \right]  \le Cw_{n}\sqrt{n} \le 2C,\qquad 	\E\left[ \sup_{\param\in \Param} \left\lVert \nabla f(\param) -   \nabla \bar{f}_{n}(\param) \right\rVert \right]  \le Cw_{n}\sqrt{n}\le 2C.
		\end{align}
		Then note that 
		\begin{align}
			\E\left[ 	\lvert \E[ \bar{g}_{n}(\param_{n})] -  \bar{f}_{n}(\param_{n}) \rvert^{2}  \right] &\le  \E\left[   \left( \lvert  \E[ \bar{g}_{n}(\param_{n})]- f(\param_{n})] \rvert   + \sup_{\param\in \Param} \, \lvert \bar{f}_{n}(\param) -  f(\param) \rvert  \right)^{2} \right] \\
			&\le \E\left[ 	\left\lvert \E[ \bar{g}_{n}(\param_{n})] -  f(\param_{n}) \right\rvert^{2} + 2C \left\lvert \E[ \bar{g}_{n}(\param_{n})]- \E[ f(\param_{n})] \right\rvert \right] + C^{2}(n+1) w_{n+1}^{2}.
		\end{align}
		Then multiply $w_{n}$ on both sides sum over all $n\ge 1$. The first two terms in the resulting infinite sum in the righthand side are finite by \eqref{eq:thm2_pf_variation111} and  \eqref{eq:thm2_pf_variation00}. For the last term, note that since $w_{n}\sqrt{n}=o(1)$ (see \ref{assumption:A4-w_t}) and using the optional condition in \ref{assumption:A4-w_t}, 
		\begin{align}
			\sum_{n=1}^{\infty} (n+1)w_{n+1}^{3} = \sum_{n=1}^{\infty} \left( w_{n+1}\sqrt{n+1} \right) w_{n+1}^{2} \sqrt{n+1} \le C' \sum_{n=1}^{\infty} w_{n+1}^{2} \sqrt{n+1}<\infty
		\end{align}
		for some constant $C'>0$.  It follows that 
		\begin{align}
			\E\left[ \sum_{n=1}^{\infty}	w_{n+1} \left\lvert \E[ \bar{g}_{n}(\param_{n})] -  \bar{f}_{n}(\param_{n}) \right\rvert^{2}  \right] <\infty.
		\end{align}
		An identical argument using the second bound in \eqref{eq:thm2_pf_variation2} as well as \eqref{eq:thm2_pf_variation0} shows 
		\begin{align}
			\E\left[ 	\sum_{n=1}^{\infty}	w_{n+1} \left\lVert \E[ \nabla \bar{g}_{n}(\param_{n})] -  \nabla \bar{f}_{n}(\param_{n}) \right\rVert^{2}  \right] <\infty. 
		\end{align}
		Combining the above two bounds gives \eqref{eq:thm_nonconvex_empirical_variation}. This completes the proof. 
	\end{proof}

	\subsection{Proof of Theorem \ref{thm:global_convergence}} 
	Next, we prove Theorem \ref{thm:global_convergence} assuming Lemmas \ref{lem:f_n_concentration_L1_gen}, \ref{lem:pos_variation}, \ref{lem:gradient_finite_sum}, and \ref{lem:asymptotic_stationarity}.
	
	\begin{proof}[\textbf{Proof of Theorem \ref{thm:global_convergence}}] 
		Suppose \ref{assumption:A1-ell_smooth}-\ref{assumption:A4-w_t}.  First assume cases \ref{C1} or \ref{C2} in Theorem \ref{thm:global_convergence}. We will first show the assertion except for the asymptotic stationarity statements.

		We first show the first part of \textbf{(i)}. By Lemma \ref{lem:pos_variation} \textbf{(iii)} and Lemma \ref{lem:gradient_finite_sum} \textbf{(i)}, we get 
		\begin{align}\label{eq:thm1_pf_eq1}
			\E\left[ \sum_{n=1}^{\infty} w_{n+1} \left( \left|  \bar{g}_{n}(\param_{n}) - \bar{f}_{n}(\param_{n})  \right|  + \lVert \nabla \bar{g}_{n}(\param_{n})-\nabla \bar{f}_{n}(\param_{n}) \rVert^{2}    \right)   \right] <\infty. 
		\end{align} 
		Denote $h^{(1)}_{n}:=\bar{g}_{n} - \bar{f}_{n}$ and $h^{(2)}_{n}:=\lVert \nabla \bar{g}_{n} -\nabla \bar{f}_{n} \rVert^{2}$. We claim that  $|h^{(i)}_{n+1}(\param_{n+1})-  h^{(i)}_{n}(\param_{n})  |=O(w_{n+1})$ for $i=1,2$. Then by Lemma \ref{lem:positive_convergence_lemma}, it holds that $|h^{(i)}_{n}(\param_{n})|\rightarrow 0$ almost surely as $n\rightarrow \infty$ for $i=1,2$.

		To show the claim, recall the recursive definitions of $\bar{g}_{n}$ and $\bar{f}_{n}$ and that $\param\mapsto \ell(\x, \param)$ is $R$-Lipschitz continuous by \ref{assumption:A1-ell_smooth} for each $\x\in \mathfrak{X}$. %It follows that $\bar{f}_{n}$ is also $R$-Lipschitz continuous for $n\ge 1$. 
		Also, $h^{(1)}_{n}(\param) = (1-w_{n}) h^{(1)}_{n-1}(\param) + w_{n} (g_{n}(\param)-f_{n}(\param))$, where $f_{n}(\cdot)=\ell(\x_{n}, \cdot)$  and $g_{n}\in \surr_{L,\rho}^{\mathbb{L}}(f_{n}, \param_{n-1})$ (see Definition \ref{def:block_surrogate}) and $\nabla (g_{n}-f_{n})$ is $L$-Lipschitz continuous and has norm $\le \eps_{n}$ at $\param_{n-1}$. Since the parameter space $\Param$ is compact by \ref{assumption:A3-cvx_constraint}, it follws that $\nabla (g_{n}-f_{n})$ is uniformly bounded by some constant $L'>0$ over $\Param$. Then $g_{n}-f_{n}$ is $L'$-Lipschitz, and by induction, $h^{(1)}_{n}$ is also $L$-Lipschitz for $n\ge 1$. Now note that 
		\begin{align}
			|h^{(1)}_{n+1}(\param_{n+1})  - h^{(1)}_{n}(\param_{n}) | & \le	|h^{(1)}_{n+1}(\param_{n+1})  - h^{(1)}_{n+1}(\param_{n})|  + 	|h^{(1)}_{n+1}(\param_{n})  - h^{(1)}_{n}(\param_{n}) | \\
			&\le L' \lVert \param_{n+1} - \param_{n} \rVert + \left| \left( \bar{g}_{n+1}(\param_{n}) -\bar{g}_{n}(\param_{n})  \right)  -  \left( \bar{f}_{n+1}(\param_{n}) -\bar{f}_{n}(\param_{n})  \right)  \right| \\
			&\le L' c'w_{n} + w_{n+1} \left| g_{n+1}(\param_{n}) - \bar{g}_{n}(\param_{n+1}) -  f_{n+1}(\param_{n}) + \bar{f}_{n}(\param_{n})  \right|\\
			&= L' \lVert \param_{n+1} - \param_{n} \rVert + w_{n+1} \left|   \bar{g}_{n}(\param_{n+1})  - \bar{f}_{n}(\param_{n})  \right|,
		\end{align}
		where the third inequality uses Lemma \ref{lem:stability} \textbf{(iii)} as well as the recursive definitions of $\bar{g}_{n+1}$ and $\bar{f}_{n+1}$. Then note that   $\ell:\mathcal{X}\times \Param\rightarrow \R$ is bounded by  \ref{assumption:A1-ell_smooth}, so $\bar{g}_{n}$ and $\bar{f}_{n}$ are also uniformly bounded in $n$. This shows $	|h^{(1)}_{n+1}(\param_{n+1})  - h^{(1)}_{n}(\param_{n}) | =O(w_{n+1})$.
		
		Next, we verify $|h^{(2)}_{n}(\param_{n})- h^{(2)}_{n}(\param_{n-1})|=O(w_{n+1})$. Note that by linearity of gradients, $\nabla \bar{g}_{n}$ and $\nabla \bar{f}_{n}$ satisfy the same recursion as $\bar{g}_{n}$ and $\bar{f}_{n}$. Moreover, $\nabla h^{(2)}_{n}(\param) = (1-w_{n}) \nabla h^{(2)}_{n-1}(\param) + w_{n} \nabla (g_{n}-f_{n})(\param)$ and since $\nabla (g_{n}-f_{n})$ is $L$-Lipschitz, so is $\nabla h^{(2)}_{n}$ for all $n\ge 1$. Then by using a similar argument as above, 
		\begin{align}
			&\left| \sqrt{h^{(2)}_{n+1}(\param_{n+1})} - \sqrt{h^{(2)}_{n}(\param_{n})} \right|	\\
			&\qquad = \lVert \nabla h_{n+1}^{(2)}(\param_{n+1}) -  \nabla  h_{n+1}^{(2)}(\param_{n}) \rVert  + \lVert \nabla h_{n+1}^{(2)}(\param_{n}) -  \nabla h_{n}^{(2)}(\param_{n}) \rVert  \\
			&\qquad  \le L\lVert \param_{n+1}-\param_{n} \rVert + w_{n+1} \lVert \nabla (g_{n+1}-f_{n+1}) (\param_{n})  + \nabla \bar{f}_{n}(\param_{n}) - \nabla \bar{g}_{n}(\param_{n})   \rVert\\
			&\qquad  \le L c'w_{n+1} + w_{n+1} \lVert  \nabla h^{(2)}_{n}(\param_{n})   \rVert.
		\end{align}
		But recall that $\nabla h^{(2)}_{n}$ is uniformly bounded by $L'$ over $\Param$. To conclude, write 
		\begin{align}
			|h^{(2)}_{n}(\param_{n})- h^{(2)}_{n}(\param_{n-1})| = \left| \sqrt{h^{(2)}_{n}(\param_{n})} - \sqrt{h^{(2)}_{n-1}(\param_{n-1})} \right| \cdot \left| \sqrt{h^{(2)}_{n}(\param_{n})} + \sqrt{h^{(2)}_{n-1}(\param_{n-1})} \right|.
		\end{align}
		Noting that $h^{(2)}_{n}$ is uniformly bounded for all $n\ge 1$, we conclude  that $|h^{(2)}_{n}(\param_{n})- h^{(2)}_{n}(\param_{n-1})|=O(w_{n+1})$, as desired. 
		
		Next, we show the first part of \textbf{(ii)}. Recall that under \ref{assumption:A4-w_t} (without the optional condition), by Lemma \ref{lem:f_n_concentration_L1_gen} (se also Lemma \ref{lem:uniform_convergence_asymmetric_weights}), we have $\sup_{\param\in \Param} |\bar{f}_{n}(\param) - f(\param) |\rightarrow 0$ and $\sup_{\param\in \Param} \lVert \nabla \bar{f}_{n}(\param) - \nabla f(\param) \rVert\rightarrow 0$ almost surely as $n\rightarrow \infty$. Then   by \textbf{(i)}, we have 
		\begin{align}\label{eq:thm1_pf_eq2}
			| \bar{g}_{n}(\param_{n}) - f(\param_{n}) | \le 			 |\bar{g}_{n}(\param_{n}) - \bar{f}_{n}(\param_{n}) |   + \sup_{\param\in \Param} | \bar{f}_{n}(\param) - f(\param) |  \rightarrow  0\quad \textup{a.s. as $n\rightarrow \infty$}.  
		\end{align}
		Similarly, by triangle inequality and \textbf{(i)}, 
		\begin{align}\label{eq:thm1_pf_eq3}
			\lVert \nabla \bar{g}_{n}(\param_{n}) - \nabla f(\param_{n}) \rVert \le 			 \lVert \nabla \bar{g}_{n}(\param_{n}) - \nabla \bar{f}_{n}(\param_{n}) \rVert   + \sup_{\param\in \Param} \lVert \nabla \bar{f}_{n}(\param) - \nabla f(\param) \rVert  \rightarrow  0\quad \textup{a.s. as $n\rightarrow \infty$}.  
		\end{align}
		%Finally, recall that under case \ref{C1}, $-\nabla \bar{g}_{n}(\param_{n})$ is in the normal cone of $\Param$ at $\param_{n}$ for each $n\ge 1$. Hence from the last result, we also deduce that $-\nabla f(\param_{n})$ lies in the normal cone of $\Param$ at $\param_{n}$ asymptotically as $n\rightarrow \infty$. 
		
		Now we show the second parts of \textbf{(i)} and \textbf{(ii)}, the asymptotic stationarity.  Assume case \ref{C1}. Then each $\param_{n}$ is an exact minimizer of $\bar{g}_{n}$ over $\Param$, so $-\nabla \bar{g}_{n}(\param_{n})$ is in the normal cone of $\Param$ at $\param_{n}$ for each $n\ge 1$. Since we have shown that $\lVert \nabla \bar{f}_{n}(\param_{n})- \nabla \bar{g}_{n}(\param_{n}) \rVert$ and $\lVert \nabla f(\param_{n})- \nabla \bar{g}_{n}(\param_{n}) \rVert$ both converges to zero almost surely as $n\rightarrow \infty$, it follows that  both $-\nabla \bar{f}_{n}(\param_{n})$  and $-\nabla f(\param_{n})$ belong to the normal cone of $\Param$ at $\param_{n}$ asymptotically almost surely as $n\rightarrow \infty$. 
		
		Assume case \ref{C2}. Further assume \ref{assumption:A7_param_surr}.  Let $\param_{\infty}\in \Param$ be an arbitrary limit point of the sequence $(\param_{t})_{t\ge 1}$ and let $(t_{k})_{k\ge 1}$ be a (random) sequence such that $\param_{t_{k}}\rightarrow \param_{\infty}$ almost surely as $n\rightarrow \infty$. Since $\Param \times \mathcal{K}$ is compact, we may choose a further sequence of $(t_{k})_{k\ge 1}$, which we will denote the same, so that $\kappa_{t_{k}}$ converges to some element $\kappa_{\infty}\in \mathcal{K}$. Hence  $\bar{g}_{\infty} :=\lim_{k\rightarrow \infty} \bar{g}_{t_{k}}$ is well-defined almost surely. It is important to note that $\param_{\infty}$ is a stationary point of $\bar{g}_{\infty}$ over $\Param$ by Lemma \ref{lem:asymptotic_stationarity}. Hence $-\nabla \bar{g}_{\infty}(\param_{\infty})$ is in the normal cone of $\Param$ at $\param_{\infty}$. But since we have shown that $h^{(i)}\rightarrow 0$ almost suresly as $n\rightarrow \infty$ for $i=2,4$,  we  must have
		\begin{align}
			\lim_{k\rightarrow \infty} \nabla \bar{f}_{t_{k}}(\param_{\infty})=\nabla \bar{g}_{\infty}(\param_{\infty}) = \nabla f(\param_{\infty}).
		\end{align}
		Since $\param_{\infty}$ was an arbitrary limit point of $(\param_{n})_{n\ge 1}$, this completes the proof of \textbf{(i)}-\textbf{(ii)}.

		Lastly, assume \ref{C3} in Theorem \ref{thm:global_convergence}.  In this case, we do not have iterate stability $\lVert \param_{n}-\param_{n-1} \rVert=O(w_{n})$ so we cannot use Lemma \ref{lem:positive_convergence_lemma} \textbf{(ii)} to deduce the whole sequence convergence as we did before. However, we can still deduce subsequential convergence.  Indeed, by Lemma \ref{lem:pos_variation} \textbf{(iii)} and Lemma \ref{lem:gradient_finite_sum} \textbf{(i)}, we have \eqref{eq:thm1_pf_eq1}. Then by Lemma \ref{lem:positive_convergence_lemma} \textbf{(i)} and noting that $\sum_{n=1}^{\infty} w_{n}=\infty$, we conclude that there exists a subsequence $(\param_{n_{k}})_{k\ge 1}$ of $(\param_{n})_{n\ge 1}$ such that almost surely, $\left|  \bar{g}_{n_{k}}(\param_{n_{k}}) - \bar{f}_{n_{k}}(\param_{n_{k}})  \right|  + \lVert \nabla \bar{g}_{n}(\param_{n_{k}})-\nabla \bar{f}_{n_{k}}(\param_{n_{k}}) \rVert^{2}\rightarrow 0$ as $k\rightarrow \infty$.  Also, using a similar argument as before in \label{eq:thm1_pf_eq2} and \label{eq:thm1_pf_eq3}, we can deduce that almost surely along the same subsequence, $\left|  \bar{g}_{n_{k}}(\param_{n_{k}}) - f(\param_{n_{k}})  \right|  + \lVert \nabla \bar{g}_{n}(\param_{n_{k}})-\nabla f(\param_{n_{k}}) \rVert^{2}\rightarrow 0$ as $k\rightarrow \infty$. Lastly, asymptotic stationarity for $\bar{g}_{n}$ is given by Lemma \ref{lem:asymptotic_stationarity}. 
	\end{proof}
	
	\begin{remark}
		Without appealing to Lemma \ref{lem:f_n_concentration_L1_gen} and using the convergence results for the empirical loss minimization stated in Theorem \ref{thm:global_convergence}, one can directly prove Theorem \ref{thm:global_convergence} \textbf{(ii)} using a similar argument as in the proof of \textbf{(i)}. This amounts to use the finite sum \eqref{eq:expected_loss_min_finite_sum_pf} in place of \eqref{eq:empirical_loss_min_finite_sum_pf} and showing the bounds $|h^{(i)}_{n+1}(\param_{n+1})-  h^{(i)}_{n}(\param_{n})  |=O(w_{n+1})$  for $i=3,4$, where $h^{(3)}=\bar{g}_{n+1} - f$ and $h^{(4)}_{n} = \lVert \nabla \bar{g}_{n} - \nabla f \rVert^{2}$. However, \eqref{eq:expected_loss_min_finite_sum_pf} holds under the additional optional condition in \ref{assumption:A4-w_t}, which essentially says $\sum_{n=1}^{\infty} w_{n}^{2}\sqrt{n}<\infty$. While this condition is standard in the literature (see, e.g., \cite[(E)]{mairal2013stochastic}), \cite[(M2)]{lyu2020online}, and \cite[(A3)]{lyu2020online_CP}, our proof above avoids it and only relies on the more standard square summability condition $\sum_{n=1}^{\infty} (\log n) w^{2}<\infty$ (see  \ref{assumption:A4-w_t}). However, the stronger condition $\sum_{n=1}^{\infty} w_{n}^{2}\sqrt{n}<\infty$ as well as the finite sum \eqref{eq:expected_loss_min_finite_sum_pf}  is crucial in our proof of Theorem \ref{thm:rate_stationarity}.
	\end{remark}

	\subsection{Proof of Theorem \ref{thm:rate_stationarity}}

	Next, we prove Theorem \ref{thm:rate_stationarity} assuming Lemmas  \ref{lem:pos_variation}, \ref{lem:gradient_finite_sum}, and \ref{lem:finite_variation_surr}.

	\begin{proof}[\textbf{Proof of Theorem \ref{thm:rate_stationarity}}]
		
		We first show \textbf{(i)}. Note that \eqref{eq:thm_convergence_bd_surrogate} follows immediately from Lemmas \ref{lem:finite_variation_surr}  and \ref{lem:positive_convergence_lemma}. Next, we show \eqref{eq:thm_convergence_bd_f_t}. By Cauchy-Schwarz inequality, for all $\param\in \Param$,
		\begin{align}
			\left|  \left\langle \nabla \bar{g}_{n}(\param_{n}),\, \frac{\param - \param_{n} }{\lVert \param - \param_{n}\rVert}\right\rangle  -  \left\langle \nabla \bar{f}_{n}(\param_{n}),\, \frac{\param - \param_{n} }{\lVert \param - \param_{n}\rVert}\right\rangle  \right| \le \lVert \nabla \bar{g}_{n}(\param_{n}) - \nabla \bar{f}_{n}(\param_{n})  \rVert.
		\end{align}
		It follows that for all $n\ge 1$, 
		\begin{align}\label{eq:thm3_finite_sum_pf0}
			\left| -  \inf_{\param\in \Param}  \left\langle \nabla \bar{f}_{n}(\param_{n}),\, \frac{\param - \param_{n} }{\lVert \param - \param_{n}\rVert}\right\rangle \right|   \le  \left| -  \inf_{\param\in \Param}  \left\langle \nabla \bar{g}_{n}(\param_{n}),\, \frac{\param - \param_{n} }{\lVert \param - \param_{n}\rVert}\right\rangle  \right|  + \lVert \nabla \bar{g}_{n}(\param_{n}) - \nabla \bar{f}_{n}(\param_{n})  \rVert.
		\end{align}

		On the other hand, by Lemmas \ref{lem:finite_variation_surr} and \ref{lem:gradient_finite_sum} \textbf{(i)}, we have 
		\begin{align}\label{eq:thm3_finite_sum_pf1}
			\sum_{n=1}^{\infty}w_{n+1} \left[  \left| -  \inf_{\param\in \Param}  \left\langle \nabla \bar{g}_{n}(\param_{n}),\, \frac{\param - \param_{n} }{\lVert \param - \param_{n}\rVert}\right\rangle \right|  + \lVert \nabla \bar{g}_{n}(\param_{n}) - \nabla \bar{f}_{n}(\param_{n})  \rVert^{2} \right] <\infty. 
		\end{align}
		Then by Lemma \ref{lem:positive_convergence_lemma}, we have 
		\begin{align}\label{eq:two_trem_bd}
			\min_{1\le k \le n} \left[ \left| -  \inf_{\param\in \Param}  \left\langle \nabla \bar{g}_{k}(\param_{k}),\, \frac{\param - \param_{k} }{\lVert \param - \param_{k}\rVert}\right\rangle  \right| + \lVert \nabla \bar{g}_{k}(\param_{k}) - \nabla \bar{f}_{k}(\param_{k})  \rVert^{2} \right]  = O\left( \left( \sum_{k=1}^{n} w_{k} \right)^{-1} \right).
		\end{align}
		Let $t_{n}\in \{1,\dots, n\}$ for $n\ge 1$ be such that the minimum above is achieved. Namely, denoting the term in the minimum above by $A_{k}$, we have $A_{t_{n}} = O\left( \left( \sum_{k=1}^{n} w_{k} \right)^{-1} \right)$. Since all terms in $A_{k}$ are nonnegative, it follows that there exists a constant $c_{1},c_{2}>0$ such that for all $n\ge 1$, alsmot surely,
		\begin{align}\label{eq:two_trem_bd_subseq}
			\left| -  \inf_{\param\in \Param}  \left\langle \nabla \bar{g}_{t_{n}}(\param_{t_{n}}),\, \frac{\param - \param_{t_{n}} }{\lVert \param - \param_{t_{n}}\rVert}\right\rangle  \right| \le \frac{c_{1} }{\sum_{k=1}^{n} w_{k}},\quad 	\lVert \nabla \bar{g}_{t_{n}}(\param_{t_{n}}) - \nabla f (\param_{t_{n}})   \rVert  \le \frac{c_{2} }{\sqrt{\sum_{k=1}^{n} w_{k}}}.
		\end{align}
		Hence from \eqref{eq:thm3_finite_sum_pf0}, it follows that there exists some constant $c_{3}>0$ such that for all $n\ge 1$, 
		\begin{align}
			\left| -  \inf_{\param\in \Param}  \left\langle \nabla \bar{f}_{t_{n}}(\param_{t_{n}}),\, \frac{\param - \param_{t_{n}} }{\lVert \param - \param_{t_{n}}\rVert}\right\rangle \right|    \le  \frac{c_{3} }{\sqrt{\sum_{k=1}^{n} w_{k}}}.
		\end{align}
		This completes the proof of \textbf{(i)}. 
		
		Next, we show \eqref{eq:thm_convergence_bd_f_E} in \textbf{(ii)}. Recall that by Lemma \ref{lem:f_n_concentration_L1_gen} and see \ref{assumption:A4-w_t}, there exists a constant $C>0$ such that for all $n\ge 1$, 
		\begin{align}\label{eq:thm2_pf_variation2_E}
			\E\left[ \sup_{\param\in \Param} \left\lVert \nabla f(\param) -   \nabla \bar{f}_{n}(\param) \right\rVert \right]  \le Cw_{n}\sqrt{n}.
		\end{align}
		Then first observe 	that by Cauchy-Schwarz and triangle inequalities, 
		\begin{align}
			\left|  \left\langle \nabla \bar{g}_{n}(\param_{n}),\, \frac{\param - \param_{n} }{\lVert \param - \param_{n}\rVert}\right\rangle  -  \left\langle \nabla f(\param_{n}),\, \frac{\param - \param_{n} }{\lVert \param - \param_{n}\rVert}\right\rangle  \right| &\le \lVert \nabla \bar{g}_{n}(\param_{n}) - \nabla f(\param_{n})  \rVert \\
			&\le \lVert \nabla \bar{g}_{n}(\param_{n}) - \nabla \bar{f}_{n}(\param_{n})  \rVert + \sup_{\param\in \Param} \lVert \nabla \bar{f}_{n}(\param) - \nabla f(\param)  \rVert.
		\end{align}
		It follows that for all $n\ge 1$, 
		\begin{align}\label{eq:thm3_finite_sum_pf3}
			\left| -  \inf_{\param\in \Param}  \left\langle \nabla f(\param_{n}),\, \frac{\param - \param_{n} }{\lVert \param - \param_{n}\rVert}\right\rangle \right|   &\le  \left| -  \inf_{\param\in \Param}  \left\langle \nabla \bar{g}_{n}(\param_{n}),\, \frac{\param - \param_{n} }{\lVert \param - \param_{n}\rVert}\right\rangle  \right|  + \lVert \nabla \bar{g}_{n}(\param_{n}) - \nabla \bar{f}_{n}(\param_{n})  \rVert \\
			& \qquad + \sup_{\param\in \Param} \lVert \nabla \bar{f}_{n}(\param) - \nabla f(\param)  \rVert.
		\end{align}	
		Note that by Lemmas \ref{lem:finite_variation_surr}, \ref{lem:gradient_finite_sum} \textbf{(i)}, we have 
		\begin{align}
			\sum_{n=1}^{\infty}  w_{n} \left( 	\E\left[ \left| -  \inf_{\param\in \Param}  \left\langle \nabla \bar{g}_{n}(\param_{n}),\, \frac{\param - \param_{n} }{\lVert \param - \param_{n}\rVert}\right\rangle  \right|   \right]+ \E\left[  \lVert \nabla \bar{g}_{n}(\param_{n}) - \nabla \bar{f}_{n}(\param_{n})  \rVert^{2} \right] \right)  <\infty. 
		\end{align}
		By Lemma  \ref{lem:positive_convergence_lemma}, similarly as before, we can take a sequence $t_{n}\in \{1,\dots,n\}$ for $n\ge 1$ such that for some constants $c_{3},c_{4}>0$ and for all $n\ge 1$, 
		\begin{align}\label{eq:two_trem_bd_subseq2}
			\E\left[ \left| -  \inf_{\param\in \Param}  \left\langle \nabla \bar{g}_{t_{n}}(\param_{t_{n}}),\, \frac{\param - \param_{t_{n}} }{\lVert \param - \param_{t_{n}}\rVert}\right\rangle  \right| \right] \le \frac{c_{3} }{\sum_{k=1}^{n} w_{k}},\quad  \E\left[ 	\lVert \nabla \bar{g}_{t_{n}}(\param_{t_{n}}) - \nabla f (\param_{t_{n}})   \rVert  \right] \le \frac{c_{4} }{\sqrt{\sum_{k=1}^{n} w_{k}}}.
		\end{align}
		Hence taking expectation on \eqref{eq:thm3_finite_sum_pf3} and using this sequence $t_{n}$ with \eqref{eq:thm2_pf_variation2_E}, we get 
		\begin{align}
			\E\left[ \left| -  \inf_{\param\in \Param}  \left\langle \nabla f(\param_{t_{n}}),\, \frac{\param - \param_{t_{n}} }{\lVert \param - \param_{t_{n}}\rVert}\right\rangle  \right| \right] \le   \frac{c_{5} }{\sqrt{\sum_{k=1}^{n} w_{k}}} + w_{n} \sqrt{n}.
		\end{align}
		for all $n\ge 1$ for some constant $c_{5}>0$. This shows  \eqref{eq:thm_convergence_bd_f_E}. 
		
		It remains to show \eqref{eq:thm_convergence_bd_f} in  \textbf{(ii)}. Assume the optional condition in \ref{assumption:A4-w_t} holds. Using a similar argument as before,  for all $n\ge 1$, 
		\begin{align}\label{eq:thm3_finite_sum_pf4}
			\left| -  \inf_{\param\in \Param}  \left\langle \nabla f(\param_{n}),\, \frac{\param - \param_{n} }{\lVert \param - \param_{n}\rVert}\right\rangle \right|   \le  \left| -  \inf_{\param\in \Param}  \left\langle \nabla \bar{g}_{n}(\param_{n}),\, \frac{\param - \param_{n} }{\lVert \param - \param_{n}\rVert}\right\rangle  \right|  + \lVert \nabla \bar{g}_{n}(\param_{n}) - \nabla f(\param_{n})  \rVert.
		\end{align}
		Also, by Lemmas \ref{lem:finite_variation_surr} and \ref{lem:gradient_finite_sum} \textbf{(ii)}, almost surely,
		\begin{align}\label{eq:thm3_finite_sum_pf4}
			\sum_{n=1}^{\infty}w_{n+1} \left[  \left| -  \inf_{\param\in \Param}  \left\langle \nabla \bar{g}_{n}(\param_{n}),\, \frac{\param - \param_{n} }{\lVert \param - \param_{n}\rVert}\right\rangle \right|  + \lVert \nabla \bar{g}_{n}(\param_{n}) - \nabla f(\param_{n})  \rVert^{2} \right] <\infty. 
		\end{align}
		The rest of the argument is identical to the proof of \textbf{(i)}. 
	\end{proof}
	
	\subsection{Proofs of Corollaries \ref{cor:unconstrained} and \ref{cor:iteration_complexity}}
	
	Lastly in this section, we prove Corollaries \ref{cor:unconstrained} and \ref{cor:iteration_complexity}. 
	
	\begin{proof}[\textbf{Proof of Corollary \ref{cor:unconstrained}}]
		Suppose the iterates $\param_{n}$ are all in the interior of $\Param$. Then in Theorem \ref{thm:rate_stationarity} \eqref{eq:thm_convergence_bd_surrogate}, we may choose $\param\in \Param$ in a way that $\param-\param_{n}$ is a positive scalar multiple of $\nabla \bar{g}_{n}(\param_{n})$. Then \eqref{eq:thm_convergence_bd_surrogate} reduces to the first bound in \eqref{eq:cor_rates_bounds}. The same argument shows the second bound in \eqref{eq:cor_rates_bounds}. For the last bound in \eqref{eq:cor_rates_bounds}, we use \eqref{eq:thm2_pf_variation11} we derived in the proof of Theorem \ref{thm:rate_surrogate_gaps}. Combining it with Lemma \ref{lem:finite_variation_surr} and using Fubini's theorem, we get 
		\begin{align}
			\sum_{n=1}^{\infty}	w_{n+1} \left( \E\left[ \lVert \nabla \bar{g}_{n}(\param_{n}) \rVert  \right] +  \left\lVert \E[ \nabla \bar{g}_{n}(\param_{n})] -  \nabla \bar{f}_{n}(\param_{n}) \right\rVert^{2} \right) <\infty. 
		\end{align}
		Then by Lemma \ref{lem:positive_convergence_lemma}, we have 
		\begin{align}\label{eq:two_trem_bd_cor}
			\min_{1\le k \le n} \left[  \E\left[ \lVert \nabla \bar{g}_{n}(\param_{n}) \rVert  \right]+ \left\lVert \nabla \E[\bar{g}_{k}(\param_{k})] - \nabla f(\param_{k})  \right\rVert^{2} \right] = O\left( \left( \sum_{k=1}^{n} w_{k} \right)^{-1} \right)
		\end{align}
		Choose a subsequence $t_{n}$ along which the above minimum is achieved for all $n\ge 1$. Then we have 
		\begin{align}
			\E\left[ \lVert \nabla \bar{g}_{t_{n}}(\param_{t_{n}}) \rVert  \right]= O\left( \left( \sum_{k=1}^{t_{n}} w_{k} \right)^{-1} \right),\quad  \left\lVert \nabla \E[\bar{g}_{t_{n}}(\param_{t_{n}})] - \nabla f(\param_{t_{n}})  \right\rVert  = O\left( \left( \sum_{k=1}^{t_{n}} w_{k} \right)^{-1/2} \right).
		\end{align}
		Therefore by using Jensen's inequality and the above bound, we deduce 
		\begin{align}
			\left\lVert   \nabla f(\param_{t_{n}})  \right\rVert  \le  \E\left[	\left\lVert   \nabla \bar{g}_{t_{n}}(\param_{t_{n}})  \right\rVert \right] +  \left\lVert \nabla \E[\bar{g}_{t_{n}}(\param_{t_{n}})] - \nabla f(\param_{t_{n}})  \right\rVert  = O\left( \left( \sum_{k=1}^{t_{n}} w_{k} \right)^{-1/2} \right).
		\end{align}	
		This completes the proof. 
	\end{proof}

	\begin{proof}[\textbf{Proof of Corollary \ref{cor:iteration_complexity} }]
		First, suppose $w_{n}=n^{-1/2}(\log n)^{\delta}$ for some $\delta>1$. Then the upper bound on the rate of convergence in Theorem \ref{thm:rate_stationarity} is of order $O(n^{-1/2}(\log n)^{\delta} )$. Then one can conclude by using the fact that $n\ge \eps^{-2} (3\log \eps^{-1})^{2\delta}$ implies $n^{-1/2}(\log n)^{\delta} \le \eps$ for all sufficiently small $\eps>0$. Indeed, assuming $n\ge \eps^{-2} (3\log \eps^{-1})^{2\delta}$, 
		\begin{align}
			n^{-1/2}(\log n)^{\delta} \le \frac{\eps }{(3\log \eps^{-1})^{\delta} }  (2\log \eps^{-1} + 2\delta \log (3\log \eps^{-1})  )^{\delta} 
		\end{align}
		and the last expression is at most $\eps$ for all sufficiently small $\eps>0$. This shows \textbf{(i)}. A similar argument using Theorem \ref{thm:rate_stationarity} \textbf(ii) as well as Corollary \ref{cor:unconstrained} shows \textbf{(ii)}.
	\end{proof}

	\section{Proof of Lemmas \ref{lem:f_n_concentration_L1_gen}, \ref{lem:pos_variation}, and  \ref{lem:gradient_finite_sum} }
	\label{section:key_lemma_pf1}
	
	Recall that for each $n\ge 0$, $\mathcal{F}_{n}$ denotes the $\sigma$-algebra generated by the history of underlying Markov chain $Y_{0},Y_{1},\dots,Y_{n}$ as well as the possible randomness in Algorithm \ref{algorithm:SMM} up to iteration $n$. Note that  $\x_{n}=\varphi(Y_{n})$ for $n\ge 1$ by \ref{assumption:A2-MC}. 
	
	We first prove Lemma \ref{lem:f_n_concentration_L1_gen} below. Our proof uses an auxiliary lemma, Lemma \ref{lem:uniform_convergence_asymmetric_weights} in the appendix.

	\begin{proof}[\textbf{Proof of Lemma \ref{lem:f_n_concentration_L1_gen}}]
		Recall that $\x_{k}=\varphi(Y_{k})$ under \ref{assumption:A2-MC}. Let $\pi_{k}$ denote the distribution of $Y_{k}$. Let $M:=\sup_{\x,\param} \lVert \psi(\x,\param)\rVert<\infty$. Note that, by a change of measure, 
		\begin{align}
			\E\left[ \psi(\x_{n},\param) \right] &= \sum_{\mathbf{y}\in\Omega} \psi(\varphi(\mathbf{y}),\param)\,P^{n}(Y_{0},\mathbf{y}) \\
			&= \sum_{\mathbf{y} \in\Omega} \psi(\varphi(\mathbf{y}),\param)\, \pi(\mathbf{y}) + \sum_{\mathbf{y}\in \Omega} \psi(\varphi(Y_{0}),\param)(P^{n}(Y_{0},\mathbf{y}) - \pi(\mathbf{y}))  \\
			&= \bar{\psi}(\param) + \sum_{\mathbf{y}\in \Omega} \psi(\varphi(Y_{0}),\param) \, (P^{n}(Y_{0},\mathbf{y}) - \pi(\mathbf{y})).
		\end{align}	
		By the triangle inequality, it follows that  
		\begin{align}
			\left\lVert \E\left[ \psi(\x_{n},\param) \right] -  \bar{\psi}(\param)  \right\rVert &\le \sum_{\mathbf{y}\in \Omega} \left\lVert \psi(\varphi(Y_{0}),\param) \right\rVert \,  \, \left| P^{n}(Y_{0},\mathbf{y}) - \pi(\mathbf{y})\right| \\
			&\le M  \sum_{y\in \Omega}  \left| P^{n}(Y_{0},\mathbf{y}) - \pi(\mathbf{y}) \right| \\
			&\le 2M \lVert P^{n}(Y_{0}, \cdot) - \pi \rVert_{TV},
		\end{align}
		where the last inequality follows from the relation $2\lVert \mu-\nu \rVert_{TV} = \sum_{x} |\mu(x)-\nu(x)|$ where $\mu,\nu$ are probability distributions on the same sample space (see \citep[Prop. 4.2]{levin2017markov}). 
		
		Now recall that under the hypothesis in \ref{assumption:A4-w_t}, $w^{n}_{k}=w_{k}\prod_{i=k+1}^{n} (1-w_{i})$ is non-decreasing in $k\in \{1,\dots,n\}$. So, $w^{n}_{1}\le \dots \le w^{n}_{n}=w_{n}$. 
		Then using \ref{assumption:A2-MC}, we have
		\begin{align}
			\left\lVert \bar{\psi}(\param)- \E\left[ \bar{\psi}_{n}(\x,\param)\right]  \right\rVert  
			&\le \left\lVert \sum_{k=1}^{n}  \left(   \bar{\psi}(\param) - \E [\psi(\x_{k},\param )]   \right) w^{n}_{k} \right\rVert  \\
			&\le \sum_{k=1}^{n} \left\lVert \bar{\psi}(\param) - \E [\psi(\x_{k},\param )]  \right\rVert w^{n}_{k}   \\
			&\le  2M\sum_{k=1}^{n}  \lVert P^{k}( Y_{0},\cdot ) - \pi \rVert_{TV}   w^{n}_{k} \\
			&\le  2M w_{n}  \sum_{k=1}^{n}   \sup_{\y\in \mathfrak{X}}  \, \lVert P^{k}( \y,\cdot ) - \pi \rVert_{TV}  \le \frac{2L\lambda w_{n}}{1-\lambda}.
		\end{align}
		This shows the first bound in \eqref{eq:lem_f_fn_bd_gen}. 
		
		Finally, for $d=1$, the second inequality in \eqref{eq:lem_f_fn_bd_gen}, as well as the last part of the statement, are direct consequences of Lemma \ref{lem:uniform_convergence_asymmetric_weights}, nothing that $w^{n}_{1}\le \dots \le w^{n}_{n}=w_{n}$.  Applying this to each of the $d$ coordinates of $\psi$ and applying triangle inequality will imply the assertion. Namely, write $\psi = (\psi^{(1)},\dots,\psi^{(d)})^{T}$. Then noting that $\psi,\bar{\psi}_{n}$ are uniformly bounded by some constant $M>0$, we have 
		\begin{align}
			\sup_{\param\in \Param}	\left\lVert \bar{\psi}(\param) - \bar{\psi}_{n}(\param) \right\rVert^{2} & \le \sum_{i=1}^{d} \sup_{\param\in \Param} \left\lvert \bar{\psi}^{(i)}(\param) - \bar{\psi}^{(i)}_{n}(\param) \right\rvert^{2}\le 2M \sum_{i=1}^{d} \sup_{\param\in \Param} \left\lvert \bar{\psi}^{(i)}(\param) - \bar{\psi}^{(i)}_{n}(\param) \right\rvert.
		\end{align}
		Note that by Lemma \ref{lem:uniform_convergence_asymmetric_weights}, the expectation of each summand in the last expression is of order $O(w_{n}\sqrt{n})$. Hence the second inequality in \eqref{eq:lem_f_fn_bd_gen} follows. Lastly, the summands in the last expression above converge to zero almost surely if $w^{n}_{n}\sqrt{n}=O(1/(\log n)^{1+\eps})$ for some $\eps>0$ by Lemma \ref{lem:uniform_convergence_asymmetric_weights}. Note that $w^{n}_{n}=w_{n}$. Hence the left-hand side above converges to zero almost surely $w_{n}\sqrt{n}=O(1/(\log n)^{1+\eps})$. This completes the proof. 
	\end{proof}

	To handle the issue of dependence in signals, we adopt the strategy developed in \cite{lyu2020online} in order to handle a similar issue for vector-valued signals (or matrix factorization). The key insight in \cite{lyu2020online} is that, while the 1-step conditional distribution $P(Y_{t-1}, \cdot)$ may be far from the stationary distribution $\pi$, the $N$-step conditional distribution $P^{N}(Y_{t-N}, \cdot)$ is exponentially close to $\pi$ under mild conditions. Hence we can condition much early on -- at time $t-N$ for some suitable $N=N(t)$. Then the Markov chain runs $N+1$ steps up to time $t+1$, so if $N$ is large enough for the chain to mix to its stationary distribution $\pi$, then the distribution of $Y_{t+1}$ conditional on $\mathcal{F}_{t-N}$ is close to $\pi$. The error of approximating the stationary distribution by the $N+1$ step distribution can be controlled using total variation distance and Markov chain mixing bound. We refine the analysis for \cite[Lem. 12]{lyu2020online} and obtain a more improved version below.

	\begin{prop}\label{prop:increment_bd}
		Assumptions \ref{assumption:A1-ell_smooth}-\ref{assumption:A4-w_t} hold. Let $(\param_{n})_{n\ge 1}$ be an output of Algorithm and denote $L=\lVert \ell(\cdot, \cdot) \rVert_{\infty}$. Then there exists a constant $C>0$ such that for all $0\le N \le n$, 	it holds that 
		\begin{align}\label{eq:pos_variation_bd}
			\E\left[ \E\left[ \ell(\x_{n+1},\param_{n-N}) - \bar{f}_{n}(\param_{n-N})  \,\bigg|\, \mathcal{F}_{n-N} \right]^{+} \right]  \le Cw_{n-N}    + 2 L \lambda^{N+1},
		\end{align}		
		where $\lambda\in [0,1)$ is the exponential mixing rate of the underlying Markov chain in \ref{assumption:A2-MC}. Furthermore, if $\x_{n}$'s are i.i.d. from $\pi$, then the assertion holds with $N=0$ and $\lambda =0$. 
	\end{prop}

	\begin{proof}
		Fix $\mathbf{y}\in \Omega$ and suppose $Y_{n-N} = \mathbf{y}$. By the Markov property, the distribution of $Y_{n+1}$ conditional on $\mathcal{F}_{n-N}$ equals $P^{N+1}(Y_{n-N},\cdot)$, where $P$ denotes the transition kernel of the  chain $(Y_{n})_{n\in \mathbb{N}}$. 
		Then by the Markov property, the distribution of $Y_{n+1}$ conditional on $\mathcal{F}_{n-N}$ equals $P^{N+1}(\mathbf{y},\cdot)$, where $P$ denotes the transition kernel of the  chain $(Y_{n})_{ t\in \mathbb{N}}$. Denote $L:=\lVert \ell(\cdot, \cdot) \rVert_{\infty}$. Denote  $\param:=\param_{n-N}$, which is deterministic with respect to $\mathcal{F}_{n-N}$. 
		
		We first claim that 
		\begin{align}\label{eq:stationary_mixing_claim}
			\left|  	\E\left[ \ell(\x_{n+1},\param)  \,\bigg|\, \mathcal{F}_{n-N} \right]  - f(\param) \right| \le 2 L \lVert P^{N+1}(\mathbf{y},\cdot) - \pi \rVert_{TV} \le 2L\lambda^{N+1}.
		\end{align}
		Then this would yield 
		\begin{align}\label{eq:stationary_mixing_claim2}
			\E\left[ \ell(\x_{n+1},\param) \,\bigg|\, \mathcal{F}_{n-N} \right] \le  f(\param)  + \left | \E\left[ \ell(\x_{n+1},\param) \,\bigg|\, \mathcal{F}_{n-N} \right] -f(\param) \right| \le f(\param) + 2L\lambda^{N+1}. 
		\end{align}
		To justify \eqref{eq:stationary_mixing_claim}, note that by a change of measure, 
		\begin{align}
			\E\left[ \ell(\x_{n+1},\param)  \,\bigg|\, \mathcal{F}_{n-N} \right] &= \sum_{\mathbf{y}'\in\Omega} \ell(\varphi(\mathbf{y}'),\param)\,P^{N+1}(\mathbf{y},\mathbf{y}') \\
			&= \sum_{\mathbf{y}'\in\Omega} \ell(\varphi(\mathbf{y}'),\param)\, \pi(\mathbf{y}') + \sum_{\mathbf{y}'\in \Omega} \ell(\varphi(\mathbf{y}'),\param)(P^{N+1}(\mathbf{y},\mathbf{y}') - \pi(\mathbf{y}'))  \\
			&= f(\param) + \sum_{\mathbf{y}'\in \Omega} \ell(\varphi(\mathbf{y}'),\param)(P^{N+1}(\mathbf{y},\mathbf{y}') - \pi(\mathbf{y}')).
		\end{align}	
		Then the claim follows by bouding $\ell(\varphi(\y'), \param) \le \lVert \cdot, \param \rVert_{\infty}$ for all $\y'\in \Omega$ and  $	\sum_{\mathbf{y}'\in \Omega} \left| P^{N+1}(\mathbf{y},\mathbf{y}') - \pi(\mathbf{y}') \right| = 2 \lVert P^{N+1}(\y,\cdot) - \pi \rVert_{TV}$.

		Next, we analyze $\E\left[ -\bar{f}_{n}(\param)  \,|\, \mathcal{F}_{n-N} \right]$ using a different approach. Namely, we decompose the times  interval  $[0,n]$ into two intervals $[1,n-N]$ and $[n-N+1, t]$ and consider the trajectory $(\x_{k})_{1\le k \le n}$ restricted onto each of these intervals. Since we are conditioning on $\mathcal{F}_{n-N}$, the trajectory $(\x_{k})_{1\le k\le n-N}$ is fully observed and may be far from the average behavior. The distribution of $\x_{k}$ then becomes close to the stationary distribution $\pi$ during the second interval of length at an exponential mixing rate of $\lambda$ (see \ref{assumption:A2-MC}). We first compare the trajectory on the second interval using the Markov chain mixing, and then compare the mean of the trajectory of the first interval with the stationary trajectory by integrating out the conditioning on $\mathcal{F}_{n-N}$. See Figure \ref{fig:mixing_control} for illustration. 
		\begin{figure*}[h]
			\centering
			\includegraphics[width=0.8 \linewidth]{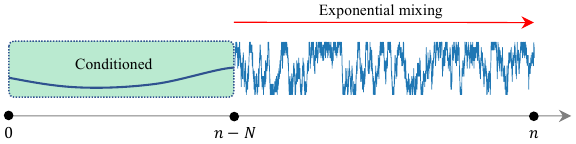}
			%\vspace{-0.5cm}
			\caption{ Illustration of the decomposition of the process used to bound the positive variation.
			}
			\label{fig:mixing_control}
		\end{figure*}

		We proceed with the sketch given above. Denote  $\hat{f}_{n:N}:= \sum_{k=t-N+1}^{n} \ell(\x_{k}, \cdot) w^{n}_{k}$. Write 
		\begin{align}
			\E\left[ -\bar{f}_{n}(\param)  \,\bigg|\, \mathcal{F}_{n-N} \right] &=-  \E\left[ \sum_{k=1}^{n} \ell(\x_{k}, \param) w^{n}_{k} \,\bigg|\, \mathcal{F}_{n-N} \right]  = - \left( \sum_{k=1}^{n-N} \ell(\x_{k}, \param) w^{n}_{k}  \right) -  \E\left[ \hat{f}_{n:N} \,\bigg|\, \mathcal{F}_{n-N} \right] 
		\end{align}

		\iffalse
		Let $\hat{w}_{k}^{n}=w^{n}_{k}/C_{M}$ for $t-N+1\le k \le t-N+M$, where $C_{M}=\sum_{i=t-N+1}^{n-M} w^{n}_{i} $. Then $[\hat{w}^{n}_{n-N+1},\dots,\hat{w}^{n}_{k}]$ defines a non-increasing probability distribution on $\{ t-N+1, \dots, t-M-1, t-N+M\}$. Denote $\hat{f}_{n}=\sum_{k=t-N+1}^{n-M} \ell(\x_{k}, \param) \hat{w}^{n}_{k}$. 	Then  by Lemma \ref{lem:uniform_convergence_asymmetric_weights} and noting that $w^{n}_{k}\le w_{k}$, we have 
		\begin{align}
			\E_{\x\sim \Param}\left[ \sup_{\param\in \Param} |f(\param) - \hat{f}_{n}(\param)  | \right] \le \frac{c_{1}w_{n-N+M} \sqrt{M}  }{C_{M}}.
		\end{align}
		for some constant $c_{1}>0$ independent of $t,N$ and $M$. 
		\fi

		\noindent Next,  denote $C_{N}=\sum_{k=n-N+1}^{n}w^{n}_{k}$. Using Markov property, we have 
		\begin{align}
			\left| C_{N} f(\param) - \E\left[  \hat{f}_{n:N}(\param) \,\bigg|\, \mathcal{F}_{n-N} \right] \right| &\le \sum_{k=t-N+1}^{n} \left| f(\param) - \E_{\y\sim \pi_{k-N}} [\ell( \varphi(\y) ,\param )]  \right| w^{n}_{k}   \\
			&\le  2L\sum_{k=1}^{n-N}  \lVert P^{k}( Y_{n-N},\cdot ) - \pi \rVert_{TV}   w^{n}_{k} \\
			&\le  2Lw_{n}  \sum_{k=1}^{n}   \sup_{\y\in \mathfrak{X}}  \, \lVert P^{k}( \y,\cdot ) - \pi \rVert_{TV}  \le \frac{2Lw_{n}  }{1-\lambda},
		\end{align}
		by using $w^{n}_{1}\le \dots \le w^{n}_{n}  = w_{n}\le 1$ under \ref{assumption:A4-w_t} and also  \ref{assumption:A2-MC} to bound the total variation distance terms. Hence by triangle inequality, 
		\begin{align}\label{eq:pos_variation_pf1}
			\E\left[ -\bar{f}_{n}(\param)  \,\bigg|\, \mathcal{F}_{n-N} \right] 
			&\le   -\sum_{k=1}^{n-N} \ell(\x_{k}, \param) w^{n}_{k}  -C_{N} f(\param) + \E\left[ \left| C_{N}f(\param) -  \hat{f}_{n:N}(\param)
			\right|   \,\bigg|\, \mathcal{F}_{n-N}  \right] \\
			&\le - \sum_{k=1}^{n-N} \ell(\x_{k}, \param) w^{n}_{k} - \sum_{k=n-N+1}^{n} f(\param) w^{n}_{k} +  \frac{2L  }{1-\lambda}.
		\end{align}
		Then combining the above bounds with \eqref{eq:stationary_mixing_claim} and a triangle inequality gives
		\begin{align}
			&\left(\E\left[ \ell(\x_{n+1},\param) - \bar{f}_{n}(\param)  \,\bigg|\, \mathcal{F}_{n-N} \right] \right)^{+} \\
			&\qquad \le f(\param)  + \left | \E\left[ \ell(\x_{n+1},\param) \,\bigg|\, \mathcal{F}_{n-N} \right] -f(\param) \right|  + \E\left[ -\bar{f}_{n}(\param)  \,\bigg|\, \mathcal{F}_{n-N} \right] \\
			& \qquad \le  \left( \sum_{k=1}^{n-N} \left( f(\param)-\ell(\x_{k}, \param) \right) w^{n}_{k}  \right)  +\frac{ 2L w_{n} }{1-\lambda}   + 2 L \lambda^{N+1} \\
			& \qquad \le \left( \sum_{k=1}^{n-N} \left( f(\param)-\ell(\x_{k}, \param) \right) w^{n-N+M}_{k}  \right)  +\frac{ 2L w_{n-N}  }{1-\lambda}   + 2 L \lambda^{N+1},
		\end{align}	
		where we have used  that $w_{k}$ is non-increasing in $k$ and  $w^{n}_{k}=w_{k}\prod_{i=k+1}^{n}(1-w_{i})\le w_{k}\prod_{i=k+1}^{m}(1-w_{i})=w^{m}_{k}$ for $k\le m \le n$. Then by  Lemma \ref{lem:f_n_concentration_L1_gen}, we have 
		\begin{align}\label{eq:pos_variation_pf2}
			\E\left[  \sum_{k=1}^{n-N+M} \left( f(\param)-\ell(\x_{k}, \param) \right) w^{n-N+M}_{k} \right] & \le Cw_{n-N} 
		\end{align}	
		Therefore, the assertion follows from \eqref{eq:pos_variation_pf1} by integrating it with respect to $\mathcal{F}_{n-N}$ and using \eqref{eq:pos_variation_pf2}. 
		
		Lastly, suppose the data sequence $\x_{n}$ is i.i.d. from the stationary distribution $\pi$. We can then take $Y_{n}=\x_{n}$, so $Y_{n}$ is a Markov chain with mixing rate $\lambda=1$, since for any $\y\in \Omega$, the one-step conditional distribution $P(\y,\cdot)$ exactly equals the stationary distribution $\pi$. This shows the assertion. 
	\end{proof}

	Next, we give proof of Lemma \ref{lem:pos_variation}.

	\begin{proof}[\textbf{Proof of Lemma \ref{lem:pos_variation}}]
		We will first show \textbf{(i)}-\textbf{(v)} under cases \ref{C1}-\ref{C2} in Theorem \ref{thm:global_convergence}. We will then show these statements for case \ref{C3} in Theorem \ref{thm:global_convergence}. After that, we will finally prove \textbf{(vi)}. 
		
		Denote $V(\x,\param):=\ell(\x,\param) - \bar{f}_{n}(\param)$. By \ref{assumption:A1-ell_smooth}, $V$ is Lipschitz in $\param$ for some $R>0$. Under cases \ref{C1}-\ref{C2}, by Lemma \ref{lem:stability}, we have $\lVert \param_{k}-\param_{k-1} \rVert \le c_{0}w_{k}$ for all $k\ge 1$ for some constant $c_{0}>0$. Since $w_{k}$ is non-increasing in $k$, by triangle inequality we get $\lVert \param_{n-N} - \param_{n} \rVert\le c_{0}\sum_{k=n-N}^{N}w_{k}\le c_{0}Nw_{n-N}$.  This yields  
		\begin{align}
			\E\left[ V(\x_{n+1}, \param_{n}) \,|\, \mathcal{F}_{n-N} \right] &=  \E\left[ V(\x_{n+1}, \param_{n-N}) \,|\, \mathcal{F}_{n-N} \right]  + \E\left[ V(\x_{n+1}, \param_{n}) - V(\x_{n+1}, \param_{n-N}) \,|\, \mathcal{F}_{n-N} \right] \\
			&\le \E\left[ V(\x_{n+1}, \param_{n-N}) \,|\, \mathcal{F}_{n-N} \right]  + R \,  \E\left[ \lVert \param_{n-N} - \param_{n} \rVert \,|\, \mathcal{F}_{n-N} \right] \\
			&\le  \E\left[ V(\x_{n+1}, \param_{n-N}) \,|\, \mathcal{F}_{n-N} \right]  + c_{0}R N w_{n-N}.
		\end{align}
		Note that $\param_{n-N}$ is deterministic with respect to $\mathcal{F}_{n-N}$. Hence multiplying by $w_{n+1}\le w_{n}$, taking positive parts, and using Proposition \ref{prop:increment_bd}, we get 
		\begin{align}\label{eq:pf_main_mixing_bd_1}
			\sum_{n=1}^{\infty} \E\left[ \E\left[ w_{n+1}V(\x_{n+1}, \param_{n}) \,|\, \mathcal{F}_{n-a_{n}} \right]^{+} \right] 
			&\le    \sum_{n=1}^{\infty} C w_{n} w_{n-a_{n}}    + 2 L w_{n}\lambda^{a_{n}+1}  + c_{0}R  a_{n} w_{n}w_{n-a_{n}}. 
		\end{align}

		\iffalse
		{\color{blue}

			\begin{align}
				&w_{n} = n^{-\delta},\quad n-a_{n}=n^{\alpha},\quad b_{n}=n^{\beta} 
			\end{align}
			
			\begin{align}
				&w_{n}w_{n-a_{n}+b_{n}}\sqrt{b_{n}} = t^{-\delta}  (n^{\alpha}+n^{\beta})^{-\delta}  t^{\beta/2} \overset{\alpha>\beta}{\approx} t^{-\delta-\delta\alpha +\beta/2} \qquad \textup{summable} \\
				& \Rightarrow \delta(1+\alpha) - \beta/2 >1 \\
				& \Leftrightarrow \delta>\frac{1 + \beta/2 }{1+\alpha} = \frac{\beta+2}{2+2\alpha} 
			\end{align}
			In order to make the lower bound above close to $1/2$, choose $\beta\approx 0$ and $\alpha\approx 1$.
		}
		\fi
		
		Now we show \textbf{(i)}. Denote $Z_{n}=w_{n+1} V(\x_{n+1}, \param_{n})$. Then by iterated expectation and Jensen's inequality, it follows that 
		\begin{align}\label{eq:pf_key_jensen}
			\sum_{n=1}^{\infty}\E[Z_{n}]^{+} &= \sum_{n=1}^{\infty}\left(\E\left[\E\left[Z_{n}\,\bigg|\, \mathcal{F}_{n-a_{t}}\right]\right]\right)^{+} \le \sum_{t=1}^{\infty}\E\left[\left( \E\left[Z_{t}\,\bigg|\, \mathcal{F}_{t-a_{t}}\right]\right)^{+}\right].
		\end{align}
		Notice that the first expression above is the summation in \textbf{(i)} that we want to bound, and the last expression above equals the left-hand side of \eqref{eq:pf_main_mixing_bd_1}. This shows \textbf{(i)}.

		Next, we show \textbf{(ii)}. Note that the left hand side in \textbf{(ii)} is bounded by the left hand side of \textbf{(i)} plus $\sum_{n=1}^{\infty} w_{n}^{2}\left(\sum_{k=1}^{n+1} \E[\eps_{k}] \right)<\infty$ by Proposition \ref{prop:Wt_bd} and \ref{assumption:A4-w_t}. So \textbf{(ii)} follows from \textbf{(i)}. 
		
		For \textbf{(iii)}, observe from Proposition \ref{prop:Wt_bd} and \eqref{eq:g_f_bar_eps_bound},  
		\begin{align}
			\sum_{n=1}^{\infty}  w_{n+1}\left( \bar{g}_{n}(\param_{n}) - \bar{f}_{n}(\param_{n}) \right)  \le \bar{g}_{1}(\param_{1}) +  \sum_{n=1}^{\infty}   w_{n+1}\left( \ell(\x_{n+1},\param_{n}) - \bar{f}_{n}(\param_{n}) \right) +\sum_{n=1}^{\infty} w_{n}^{2} \left(\sum_{k=1}^{n+1} \eps_{k} \right).
		\end{align}
		Then \textbf{(iii)} follows from taking expectation and using \textbf{(i)} with \ref{assumption:A4-w_t}.
		
		Next, we show \textbf{(iv)}. The argument is similar. Since the loss function $\ell$ is $R$-Lipstchitz by \ref{assumption:A1-ell_smooth}, so is $f-\bar{f}_{n}$. Then, as before, we can $\lVert \param_{n}-\param_{n-1} \rVert \le c_{0}w_{n}$ and $w_{n}$ non-increasing  to show 
		\begin{align}
			\left| \E\left[  f(\param_{n}) - \bar{f}_{n}(\param_{n})   \,\bigg|\, \mathcal{F}_{n-N} \right] \right| &\le \left|  \E\left[ f(\param_{n-N}) - \bar{f}_{n-N}(\param_{n})  \,\bigg|\, \mathcal{F}_{n-N} \right]  \right| + c_{0}R N w_{n-N}. 
		\end{align}
		Note that $\param_{n-N}$ is deterministic with respect to $\mathcal{F}_{n-N}$. Hence using Proposition \ref{prop:increment_bd} with $N=a_{n}$ where $a_{n}$ as in \ref{assumption:A4-w_t}, it follows that 
		\begin{align}
			\left|	\E\left[  f(\param_{n}) - \bar{f}_{n}(\param_{n})   \,\bigg|\, \mathcal{F}_{n-a_{n}} \right] \right| \le \frac{2L w_{n-a_{n}}}{1-\lambda} + c_{0}R a_{n} w_{n-a_{n}}. 
		\end{align}
		Multiplying by $w_{n+1}$ and integrating with respect to $\mathcal{F}_{n-a_{n}}$, we get 
		\begin{align}\label{eq:pf_key_4}
			w_{n+1}  \E\left[ \left|	\E\left[   f(\param_{n}) - \bar{f}_{n}(\param_{n})   \,\bigg|\, \mathcal{F}_{n-a_{n}}  \right] \right|  \right]   \le   \frac{2L w_{n}w_{n-a_{n}}}{1-\lambda} + c_{0}R a_{n} w_{n}w_{n-a_{n}}. 
		\end{align}
		By \ref{assumption:A4-w_t}, the right-hand side is summable. Then the assertion follows from Jensen's inequality similarly as in the proof of \textbf{(i)}. 
		
		We turn to prove \textbf{(v)}. Suppose the optional condition in \ref{assumption:A4-w_t} holds, that is,  $\sum_{n=1}^{\infty} w_{n} w_{n-a_{n}}\sqrt{n}<\infty$, where $a_{n}$ is the sequence in \ref{assumption:A4-w_t}. As in the proof of \textbf{(iv)}, and using the second inequality in Lemma \ref{lem:f_n_concentration_L1_gen}, 
		\begin{align}
			\E\left[ \left|  f(\param_{n}) - \bar{f}_{n}(\param_{n})  \right|  \,\bigg|\, \mathcal{F}_{n-N}  \right]  &\le   \E\left[\left| f(\param_{n-N}) - \bar{f}_{n-N}(\param_{n}) \right|  \,\bigg|\, \mathcal{F}_{n-N} \right]   + c_{0}R N w_{n-N} \\
			&\le  C w_{n-N} \sqrt{n-N} + c_{0}R N w_{n-N},
		\end{align}
		where $C>0$ is a constant in Lemma \ref{lem:f_n_concentration_L1_gen}. Then letting $N=a_{n}$ in \ref{assumption:A4-w_t} integrating out $\mathcal{F}_{n-N}$, multiplying by $w_{n+1}\le w_{n}$ and summing over all $n\ge 1$ gives 
		\begin{align}
			\sum_{n=0}^{\infty}  w_{n+1} \E\left[  \left| f(\param_{n}) - \bar{f}_{n}(\param_{n}) \right|  \right] \le	 \sum_{n=1}^{\infty} Cw_{n} w_{n-a_{n}}\sqrt{n-a_{n}} + c_{0}R a_{n} w_{n}w_{n-a_{n}}<\infty.
		\end{align}

		Now, we establish \textbf{(i)}-\textbf{(v)} under case \ref{C3} in Theorem \ref{thm:global_convergence}. This case must be dealt with separately since the stability of estimates $\lVert \param_{n}-\param_{n-1} \rVert=O(w_{n})$ is not readily available (see Lemma \ref{lem:stability}). In this case, we rely on `instant mixing' ($\lambda=0$) of the i.i.d. data sequence $\x_{n}$ stated in the second part of Proposition \ref{prop:increment_bd}. Namely, we condition on $\mathcal{F}_{n}$ and use $\lambda=0$ while avoiding using $\lVert \param_{n}-\param_{n-1} \rVert=O(w_{n})$. For \textbf{(i)}, we first use \eqref{eq:pf_main_mixing_bd_1} with $a_{n}=0=N$ to get 
		\begin{align}
			\sum_{n=1}^{\infty} \E\left[ \E\left[ w_{n+1}V(\x_{n+1}, \param_{n}) \,|\, \mathcal{F}_{n} \right]^{+} \right] 
			&\le    \sum_{n=1}^{\infty} C w_{n}^{2}  < \infty. 
		\end{align}
		The rest follows from \eqref{eq:pf_key_jensen}. Then \textbf{(ii)}-\textbf{(iii)} can be deduced similarly as in the previous case using \textbf{(i)}. Also, \textbf{(iv)} follows from \eqref{eq:pf_key_4} using $\lambda=0$ and $N=0$. \textbf{(v)} can be deduced similarly.

		Lastly, we show \textbf{(vi)}. For cases \ref{C1}-\ref{C2} in Theorem \ref{thm:global_convergence}, this is trivial since we have $\lVert \param_{n}-\param_{n-1} \rVert= O(w_{n})$ and $\sum_{n=1}^{\infty} w_{n}^{2}<\infty$ by \ref{assumption:A4-w_t}. Suppose \ref{C3}. By Lemma \ref{lem:stability} \textbf{(i)}, we have 
		\begin{align}
			\bar{g}_{n}(\param_{n-1}) - \bar{g}_{n}(\param_{n}) \ge  \frac{\hat{\rho}}{2}\lVert \param_{n}-\param_{n-1} \rVert^{2} -  \Delta_{n}.
		\end{align}
		Moreover, using  \eqref{eq:surrogate_loss_recursion} and \eqref{eq:g_f_bar_eps_bound}, we get
		\begin{align*}
			-  \Delta_{n+1} \le \bar{g}_{n+1}(\param_{n})-\bar{g}_{n+1}(\param_{n+1})\le 
			\bar{g}_{n}(\param_{n}) -\bar{g}_{n+1}(\param_{n+1}) + w_{n+1}(\ell(\x_{n+1},\param_{n})-\bar{f}_{n}(\param_{n})) + w_{n}^{2}\sum_{k=1}^{n} \eps_{k}. 
		\end{align*} 
		Taking expectation and summing over $n\ge 1$ and using \textbf{(i)} as well as \ref{assumption:A4-w_t}, 
		\begin{align}
			\left| \sum_{n=1}^{\infty}  \E\left[ \bar{g}_{n+1}(\param_{n})-\bar{g}_{n+1}(\param_{n+1}) \right]  \right| <\infty. 
		\end{align}
		Using \ref{assumption:A5_sufficient_surrogate_decay}, this implies 
		\begin{align}
			\E\left[ 	\sum_{n=1}^{N}  \frac{\hat{\rho}}{2} \rVert \param_{n}-\param_{n-1} \rVert^{2} \right] \le  \sum_{n=1}^{\infty} \E\left[ \bar{g}_{n}(\param_{n-1})-\bar{g}_{n}(\param_{n}) \right]  + \E\left[ \sum_{n=1}^{\infty} \Delta_{n} \right] < \infty,
		\end{align} 
		as desired. 
	\end{proof}

	Lastly in this section, we prove Lemma \ref{lem:gradient_finite_sum}.

	\begin{proof}[\textbf{Proof of Lemma \ref{lem:gradient_finite_sum}}]
		Denote $\alpha_{n} = \nabla\bar{g}_{n}(\param_{n})-\nabla\bar{f}_{n}(\param_{n})$ and $\beta_{n} = \nabla\bar{f}_{n}(\param_{n})-\nabla f(\param_{n})$.  Fix $\param\in \R^{p}$ and $\eps>0$. With \ref{assumption:A1-ell_smooth} and a recursion, $\bar{f}_{n}$ and $g_{n}$ are also $R$-Lipschitz continuous. By Lemma \ref{lem:surrogate_L_gradient}, we can write
		\begin{align}
			\left| \bar{g}_{n}(\param_{n}+\eps \param) - \bar{g}_{n}(\param_{n}) - \langle  \nabla \bar{g}_{n}(\param_{n}), \eps\param \rangle \right| &\le \frac{L\eps^{2}}{2} \lVert \param \rVert^{2} ,\\
			\left| \bar{f}_{n}(\param_{n}+\eps \param) - \bar{f}_{n}(\param_{n}) - \langle \nabla \bar{f}_{n}(\param_{n}), \eps \param \rangle \right| &\le  \frac{L\eps^{2}}{2} \lVert \param \rVert^{2},
		\end{align}
		for some uniform constant $L>0$ that does not depend on $\param,\eps$ and $n$.  Recall that $\bar{g}_{n}\ge \bar{f}_{n}-\bar{\eps}_{n}$ for all $n\ge 1$.  Hence 
		\begin{align}
			-	\frac{L\eps^{2}}{2} \lVert \param \rVert^{2} +  \bar{f}_{n}(\param_{n}) + \langle \nabla \bar{f}_{n}(\param_{n}),  \eps \param \rangle &\le   \bar{f}_{n}(\param_{n}+\eps \param)  \\
			&\le   \bar{g}_{n}(\param_{n}+\eps \param) + \bar{\eps}_{n}  \\
			&\le \bar{g}_{n}(\param_{n}) + \langle \nabla \bar{g}_{n}(\param_{n}, \eps\param \rangle+ \frac{L\eps^{2}}{2} \lVert \param \rVert^{2} + \bar{\eps}_{n},
		\end{align}
		so, we obtain the following key inequality
		\begin{align}\label{eq:gradient_key_inequality}
			\langle  \alpha_{n}, \eps\param \rangle = \tr\left( \left( \nabla \bar{g}_{n}(\param_{n})-\nabla \bar{f}_{n}(\param_{n})\right)^{T}(\eps\param) \right) \ge (\bar{f}_{n}(\param_{n}) - \bar{g}_{n}(\param_{n}) - \bar{\eps}_{n}) - L\eps^{2} \lVert \param\rVert^{2}.
		\end{align}
		Choosing $\param=- \alpha_{n}$, we get 
		\begin{align}
			-\eps \lVert  \alpha_{n} \rVert^{2} \ge  (\bar{f}_{n}(\param_{n}) - \bar{g}_{n}(\param_{n})- \bar{\eps}_{n}) - c\eps^{2} \lVert \alpha_{n} \rVert^{2}.
		\end{align}
		Rearranging, we get 
		\begin{align}
			(\eps-c\eps^{2}) \lVert \alpha_{n} \rVert^{2} \le  (\bar{g}_{n}(\param_{n})+\bar{\eps}_{n} - \bar{f}_{n}(\param_{n})).
		\end{align}
		Recall that this holds for all $\eps>0$ and $n\ge 1$. Taking expectation, multiplying by $w_{n}$, and summing up in $n$, 
		\begin{align}
			(\eps-c\eps^{2}) \sum_{n=1}^{\infty} w_{n} \E\left[ \lVert \alpha_{n} \rVert^{2} \right]\le \sum_{n=1}^{\infty} w_{n} \E\left[\bar{g}_{n}(\param_{n}) - \bar{f}_{n}(\param_{n}) \right] + \sum_{n=1}^{\infty} w_{n} \E[ \bar{\eps}_{n}] <\infty,
		\end{align}
		where the finiteness above uses Lemma \ref{lem:pos_variation} and \ref{assumption:A4-w_t} with \eqref{eq:bar_eps_bound}. Since $\eps>0$ was arbitrary, this shows 
		\begin{align}\label{eq:alpha_grad_bd1}
			\sum_{n=1}^{\infty} w_{n} \E\left[ \lVert \alpha_{n} \rVert^{2} \right] <\infty.
		\end{align}

		Next, we will analyze the sums $\sum_{n=1}^{\infty} w_{n} \left\lVert \E\left[ \beta_{n} \right] \right\rVert^{2}$ and $\sum_{n=1}^{\infty} w_{n} \E\left[ \lVert  \beta_{n}  \rVert^{2}\right]$.  To this end, note that  
		\begin{align}
			\nabla f(\param) = \nabla \E_{\x\sim \pi}\left[ \ell(\x,\param) \right] = \E_{\x\sim \pi}\left[\nabla_{2}  \ell(\x,\param) \right]. 
		\end{align}
		Define $\psi:(\x,\param)\mapsto \nabla_{2} \ell(\x,\param)$. Then by the linearity of gradients and induction, we can show 
		\begin{align}
			\bar{\psi}_{n}(\param) = \nabla \bar{f}_{n}(\param).
		\end{align} 
		Thus by Lemma \ref{lem:f_n_concentration_L1_gen}, there exists a constant $C>0$ such that for all $n\ge 1$,
		\begin{align}\label{eq:grad_concentration_pf}
			\sup_{\param\in \Param} \left\lVert \E\left[  \nabla f(\param) \right]- \E\left[  \nabla \bar{f}_{n}(\param)\right] \right\rVert^{2}  \le Cw_{n},\qquad 	\E\left[  \sup_{\param\in \Param} \left\lVert  \nabla f(\param)  - \nabla  \bar{f}_{n}(\param) \right\rVert^{2}\right]  \le Cw_{n}\sqrt{n}.
		\end{align}
		By \ref{assumption:A4-w_t}, this implies 
		\begin{align}\label{eq:grad_concentration_pf2}
			\sum_{n=1}^{\infty} w_{n} \, \left\lVert \E\left[  \beta_{n}\right] \right\rVert^{2} \le 	\sum_{n=1}^{\infty} w_{n}\,	\sup_{\param\in \Param} \left\lVert \E\left[  \nabla f(\param) \right]- \E\left[  \nabla \bar{f}_{n}(\param)\right] \right\rVert^{2} \le  C\sum_{n=1}^{\infty} w_{n}^{2} < \infty.	
		\end{align}
		Similarly, we also get 
		\begin{align}\label{eq:grad_concentration_pf3}
			\sum_{n=1}^{\infty} w_{n} \, \E\left[  \left\lVert  \beta_{n} \right\rVert^{2} \right] \le 	\sum_{n=1}^{\infty} w_{n}\,	\sup_{\param\in \Param} \E\left[  \left\lVert  \nabla f(\param) \right]- \E\left[  \nabla \bar{f}_{n}(\param) \right\rVert^{2} \right]   \le  C\sum_{n=1}^{\infty} w_{n}^{2}\sqrt{n}.
		\end{align}
		
		Now we conclude the assertions. For \textbf{(i)}, by Jensen's inequality, we have
		\begin{align}
			\sum_{n=1}^{\infty} w_{n} \E\left[  \lVert \alpha_{n}  \rVert^{2} \right]\le \sum_{n=1}^{\infty} w_{n} \lVert \E\left[ \alpha_{n} \right] \rVert^{2} <\infty.
		\end{align}
		Hence by Fubini's theorem, we get 
		\begin{align}
			\E\left[ \sum_{n=1}^{\infty} w_{n} \left( \lVert \alpha_{n} \rVert^{2}  + \left\lVert \E\left[  \beta_{n}\right] \right\rVert^{2} \right)\right] =	\sum_{n=1}^{\infty} w_{n}\E\left[   \lVert \alpha_{n} \rVert^{2} \right]   +	\sum_{n=1}^{\infty} w_{n} \, \left\lVert \E\left[  \beta_{n}\right] \right\rVert^{2} <\infty.
		\end{align}
		This shows \textbf{(i)}. 
		
		Next, assume the optional condition in \ref{assumption:A4-w_t} holds. since $w_{n}$ is non-increasing in $n$, it follows that $\sum_{n=1}^{\infty} w_{n}^{2}\sqrt{n} <\infty$. Then \eqref{eq:grad_concentration_pf3} shows $\sum_{n=1}^{\infty} w_{n}\E\left[ \lVert \beta_{n} \rVert^{2} \right]<\infty$. We have shown that $\sum_{n=1}^{\infty} w_{n}\E\left[ \lVert \alpha_{n} \rVert^{2} \right]<\infty$ in \eqref{eq:alpha_grad_bd1}. Then noting that $\lVert \alpha_{n} +\beta_{n}\rVert^{2} \le 2 \lVert \alpha_{n}\rVert^{2} + 2\lVert \beta_{n}\rVert^{2}$, 
		using Fubini's theorem, we get 
		\begin{align}
			\E\left[	\sum_{n=1}^{\infty} w_{n}\left( \lVert \alpha_{n} \rVert^{2} +  \lVert \beta_{n} \rVert^{2} +  \lVert \alpha_{n} + \beta_{n} \rVert^{2} \right) \right]<\infty.
		\end{align}
		Since $\bar{g}_{n}(\param_{n})-\nabla f(\param_{n})=\alpha_{n} + \beta_{n}$, this is the desired conclusion.
	\end{proof}

	\vspace{0.2cm}
	\section{Proof of Lemma \ref{lem:finite_variation_surr}}
	\label{section:key_lemma_pf2}
	
	The goal of this section is to prove Lemma  \ref{lem:finite_variation_surr}. Throughout this section, let $(\param_{n})_{n\ge 1}$ be an output of Algorithm \ref{algorithm:SMM}, where each update $\param_{n-1}\rightarrow \param_{n}$ uses either Algorithm \ref{algorithm:BSM-PR} (full coordinate descent with proximal regularization) or Algorithm \ref{algorithm:BSM-DR} (block coordinate descent with diminishing radius). For the latter case, decompose the update $\param_{n-1}\rightarrow \param_{n}$ into $\param_{n-1}=\param_{n}^{(0)}\rightarrow \param_{n}^{(1)} \rightarrow \dots \rightarrow \param_{n}^{(m)}=\param_{n}$, where each $\param_{n}^{(i)}$ is the output of \eqref{eq:BSM_factor_update_DR}  at sub-iteration $i$ of Algorithm \ref{algorithm:BSM-DR}. Namely, let  $\mathbb{J}$ denote the set of coordinate blocks used in Algorithm \ref{algorithm:SMM} and let $J_{1}(n),\dots,J_{m}(n)\in \mathbb{J}$ denote the coordinate blocks used in the $i$th iteration of Algorithm \ref{algorithm:BSM-DR} for the update $\param_{n-1}\rightarrow \param_{n}$.  Then $\param_{n}^{(i)}$ is an approximate minimizer of $\bar{g}_{n}(\param) + \Psi_{n}(\lVert \param-\param_{n-1} \rVert)$ over the convex set $\Param^{J_{i}(n)}_{n}$ (see \eqref{eq:param_J_n_def}) with optimality gap $\Delta_{n}^{(i)}$ (see \ref{assumption:A5_sufficient_surrogate_decay}), where $\Psi_{n}(x)$ takes value $0$ if $x\le c'w_{n}/m$ and $\infty$ otherwise (radius restriction). For a unified treatment, we also use the same notation when Algorithm \ref{algorithm:BSM-PR} is used instead of Algorithm \ref{algorithm:BSM-DR}. In this case, we take $m=1$ (single update), $\mathbb{J}=\{ \{1,\dots,p\} \}$,  $J_{n}^{(1)}=\{1,\dots,p\}$ , and $\Param_{n}^{J_{1}(n)}=\Param$ (no frozen coordinates), and $\Psi_{n}(x) \equiv \frac{\hat{\rho}}{2} x^{2}$ (proximal regularization). In this case, the optimality gap is denoted as $\Delta_{n}^{(1)}=\Delta_{n}$ (see \ref{assumption:A5_sufficient_surrogate_decay}).  
	
	%which consists of vectors $\param\in \Param$ that agree with $\param_{n}$ on coordinates not in $J_{i}$ and within distance $c'w_{n+1}/m$ from $\param_{n}$.

	The starting point of the analysis in this section is the following proposition, which states that the linear variation of $\bar{g}_{n}$ over the sequence $(\param_{n;i})_{n\ge 1, 1\le i \le m}$ is finite.

	\begin{prop}\label{prop:linear_surr_variation}
		Assume \ref{assumption:A1-ell_smooth}-\ref{assumption:A4-w_t}. Under any of the three cases \ref{C1}-\ref{C3} in Theorem \ref{thm:global_convergence}, we have 
		\begin{align}\label{eq:finite_surr_change_variation}
			\E\left[ 	\sum_{n=1}^{\infty}   \left| \left\langle \nabla \bar{g}_{n}(\param_{n-1}),\,  \param_{n} - \param_{n-1}  \right\rangle \right| \right] <\infty. 
		\end{align}
		Furthermore, almost surely, 
		\begin{align*}
			\sum_{n=1}^{\infty} \E_{J_{1}(n),\dots,J_{m}(n)} \left[     \left| \left\langle \nabla \bar{g}_{n}(\param_{n-1}),\,  \param_{n} - \param_{n-1}  \right\rangle \right| \,\bigg|\, \mathcal{F}_{n-1} \right] <\infty. 
		\end{align*}
	\end{prop}

	\begin{proof}
		Fix $i=1,\dots,m$. By Proposition \ref{prop:g_grad_Lipschitz},  $\nabla \bar{g}_{n}$ over $\Param$ is $L$-Lipschitz for all $n\ge 1$. Hence by Lemma \ref{lem:surrogate_L_gradient}, for all $t\ge 1$, 
		\begin{align*}
			\left|\bar{g}_{n}(\param_{n}^{(i)})-\bar{g}_{n}(\param_{n-1}) - \left\langle \nabla \bar{g}_{n}(\param_{n-1}),\,  \param_{n}^{(i)} - \param_{n-1}  \right\rangle \right| \le \frac{L}{2}\lVert \param_{n}^{(i)} - \param_{n-1} \Vert_{F}^{2}.
		\end{align*}
		Note that for $1\le i \le m$, 
		\begin{align}
			\bar{g}_{n}(\param_{n}^{(i)}) + \Psi_{n}(\lVert \param_{n}^{(i)}-\param_{n-1} \rVert) \le \bar{g}_{n}(\param_{n}^{(i-1)}) + \Psi_{n}(\lVert \param_{n}^{(i-1)}-\param_{n-1} \rVert)  + \Delta_{n}^{(i)}. 
		\end{align}
		By induction, we get 
		\begin{align}
			\bar{g}_{n}(\param_{n}^{(i)})  \le \bar{g}_{n}(\param_{n}^{(i)}) + \Psi_{n}(\lVert \param_{n}^{(i)}-\param_{n-1} \rVert) \le \bar{g}_{n}(\param_{n-1}) + \sum_{j=1}^{i} \Delta_{n}^{(j)}. 
		\end{align}
		Hence 
		\begin{align}
			|	\bar{g}_{n}(\param_{n}^{(i)})  - \bar{g}_{n}(\param_{n-1}) | \le \bar{g}_{n}(\param_{n-1})- \bar{g}_{n}(\param_{n}^{(i)}) +  2 \sum_{J=1}^{i} \Delta_{n}^{(j)},
		\end{align}
		so it follows that 
		\begin{align}\label{eq:1storder_growth_bd_pf}
			\left| \left\langle \nabla \bar{g}_{n}(\param_{n-1}),\,  \param_{n}^{(i)} - \param_{n-1}  \right\rangle \right| & \le   \frac{L}{2}\lVert \param_{n}^{(i)} - \param_{n-1} \Vert_{F}^{2} + \bar{g}_{n}(\param_{n-1})- \bar{g}_{n}(\param_{n}^{(i)})  +   2\sum_{j=1}^{i} \Delta_{n}^{(j)} \\
			&\le  \frac{L}{2}\lVert \param_{n}^{(i)} - \param_{n-1} \Vert_{F}^{2} + \bar{g}_{n}(\param_{n-1})- \bar{g}_{n}(\param_{n}) +  3\sum_{j=i}^{m} \Delta_{n}^{(j)}. 
		\end{align}
		On the other hand, using \eqref{eq:surrogate_loss_recursion} and $\bar{g}_{n}\ge \bar{f}_{n}$, note that 
		\begin{align*}
			0\le \bar{g}_{n+1}(\param_{n})-\bar{g}_{n+1}(\param_{n+1})\le 
			\bar{g}_{n}(\param_{n}) -\bar{g}_{n+1}(\param_{n+1}) + w_{n+1}(\ell(\x_{n+1},\param_{n})-\bar{f}_{n}(\param_{n})).
		\end{align*} 
		Hence using Lemma \ref{lem:pos_variation}, we have  
		\begin{align*}
			\sum_{n=1}^{\infty} \E\left[ \bar{g}_{n+1}(\param_{n})-\bar{g}_{n+1}(\param_{n+1}) \right] <\infty.
		\end{align*}
		
		Next, by Lemma \ref{lem:pos_variation}, we have 
		\begin{align}
			\sum_{n=1}^{\infty} \E\left[  \lVert \param_{n} - \param_{n+1} \Vert_{F}^{2} \right] <\infty. 
		\end{align}
		Moreover, we have $\E\left[ \sum_{n=1}^{\infty} \Delta_{n} \right]<\infty$ by \ref{assumption:A5_sufficient_surrogate_decay} for cases \ref{C2} and \ref{C3} and by Lemma \ref{lem:strongly_convex_surrogate_A6'} for case \ref{C1}. Thus the above inequalities yield 
		\begin{align*}
			\sum_{n=1}^{\infty} \E\left[ \left| \left\langle \nabla \bar{g}_{n}(\param_{n-1}),\,  \param_{n}^{(i)} - \param_{n-1}  \right\rangle \right|\right] <\infty.
		\end{align*}
		Taking $i=m$ and using Fubini's theorem, we obtain the first assertion \eqref{eq:finite_surr_change_variation}. Furthermore, one can rewrite \eqref{eq:finite_surr_change_variation} as 
		\begin{align}
			\E\left[ \sum_{n=1}^{\infty} \E_{J_{1}(n),\dots,J_{m}(n)} \left[    \left| \left\langle \nabla \bar{g}_{n}(\param_{n-1}),\,  \param_{n} - \param_{n-1}  \right\rangle \right| \,\bigg|\, \mathcal{F}_{n-1}  \right] \right] <\infty. 
		\end{align}
		Since the random variable inside the expectation above is nonnegative, it has to be finite almost surely. This shows the second part of the assertion. 
	\end{proof}

	Next, we show that the block coordinate descent we use to obtain $\param_{n}$ should always give the optimal first-order descent up to a small additive error. 
	
	\begin{lemma}[Approximate first-order optimality]\label{lem:first_order_optimality}
		Assume \ref{assumption:A1-ell_smooth}-\ref{assumption:A4-w_t}. Assume cases \ref{C1}-\ref{C3} in Theorem \ref{thm:global_convergence}. Then there exists constants $c_{1},c_{2},c_{3} >0$ such that almost surely, for any sequence $(b_{n})_{n\ge 1}$, $0<b_{n}\le w_{n}$,  
		\begin{align}\label{eq:first_order_optimality}
			\E\left[ \left\langle \nabla \bar{g}_{n}(\param_{n-1}),\,    b_{n}^{-1}(\param_{n}-\param_{n-1} )  \right\rangle  \,\bigg| \, \mathcal{F}_{n-1} \right]
			&\le  \inf_{\param\in \Param } \left\langle \nabla \bar{g}_{n-1}(\param_{n-1}),\, \frac{\param - \param_{n-1} }{\lVert \param - \param_{n-1}\rVert}\right\rangle  +  \E\left[ b_{n}^{-1}\Delta_{n} \,\bigg| \, \mathcal{F}_{n-1} \right] \\
			&\hspace{0.5cm}+ c_{1}w_{n} + \E\left[ c_{2}(1+\hat{\rho}) b_{n}  +  c_{3} b_{n}^{-1} \lVert \param_{n} - \param_{n-1} \rVert^{2} \,\bigg| \, \mathcal{F}_{n-1} \right],
		\end{align}
		where we take $\hat{\rho}= 0$ for cases \ref{C1}-\ref{C2}.	Furthermore, the above holds for all $b_{n}>0$ for cases \ref{C1} and \ref{C3}. 
	\end{lemma}
	
	%\noindent For the proof of Lemma \ref{lem:first_order_optimality}, we need the following two simple statements in establishing a lemma that is crucial in proving Lemma \ref{lem:finite_variation_surr}. 

	\begin{proof}
		We first show the assertion for cases \ref{C1} and \ref{C3}. Fix arbitrary $\param= \Param$ such that $\lVert \param - \param_{n-1} \rVert \le   c'b_{n}$.  Recall that $\param_{n}$ is an inexact minimizer of $\bar{g}_{n}(\param) + \Psi_{n}(\lVert\param-\param_{n-1} \rVert)$ over $\Param$. Denoting $\param_{n}^{\star}$ to be an exact minimizer of the same problem and by using convexity of $\Param$, we have for each $a\in [0,1]$, 
		\begin{align}\label{eq:first_order_opt_pf1_PR}
			\bar{g}_{n}(\param_{n})  - \Delta_{n} &\le 
			\bar{g}_{n}(\param_{n}) + \Psi_{n}(\lVert \param-\param_{n-1} \rVert)  - \Delta_{n} \\
			&\le  \bar{g}_{n}(\param_{n}^{\star})  +  \Psi_{n}(\lVert \param_{n}^{\star}-\param_{n-1} \rVert)  \\
			&  \le  \bar{g}_{n}\left( \param_{n-1} + a  \left( \param-\param_{n-1}  \right) \right) +  \Psi_{n}( a \lVert \param-\param_{n-1} \rVert).
		\end{align}
		Note that by Proposition \ref{prop:g_grad_Lipschitz},  $\nabla \bar{g}_{n}$ is $L$-Lipschitz and $\lVert \nabla g_{n+1} - \nabla g_{n} \rVert\le L'w_{n+1}$ for some constant $L'>0$. Hence subtracting $\bar{g}_{n}(\param_{n-1})$ from both side, we get 
		\begin{align}\label{eq:first_order_opt_pf2_PR}
			&\left\langle \nabla \bar{g}_{n}(\param_{n-1}),\,   \param_{n}-\param_{n-1} \right\rangle - \frac{L}{2} \lVert \param_{n}-\param_{n-1} \rVert^{2} - \Delta_{n}\\
			&\qquad \le  a \left\langle \nabla \bar{g}_{n}(\param_{n-1}),\,  \param-\param_{n-1} \right\rangle + \frac{a^{2}(L+\hat{\rho}) }{2}  \lVert \param - \param_{n-1} \rVert^{2}  \\
			&\qquad \le  a \left\langle \nabla \bar{g}_{n-1}(\param_{n-1}),\,  \param-\param_{n-1} \right\rangle + \frac{a^{2}(L+\hat{\rho}) }{2}  \lVert \param - \param_{n-1} \rVert^{2}  + aL'w_{n} \lVert \param - \param_{n-1} \rVert,
		\end{align}
		where we have used $\Psi_{n}(x) \equiv \frac{\hat{\rho}}{2} x^{2}$ for the first inequality. Rearranging, we get 
		\begin{align}\label{eq:first_order_opt_pf3_PR}
			\left\langle \nabla \bar{g}_{n}(\param_{n-1}),\,   \param_{n}-\param_{n-1} \right\rangle &\le 
			a \left\langle \nabla \bar{g}_{n-1}(\param_{n-1}),\,  \param-\param_{n-1} \right\rangle + c_{1} a^{2}(L+\hat{\rho})  \lVert \param - \param_{n-1} \rVert^{2}  + aL'w_{n} \lVert \param - \param_{n-1} \rVert \\
			&\hspace{3cm} + \frac{L}{2} \lVert \param_{n}-\param_{n-1} \rVert^{2} + \Delta_{n}
		\end{align}
		for any constant $c_{1}\ge 1/2$. This holds for all $a\in [0,1]$. Viewing the right-hand side as a quadratic function in $a$, we can choose $c_{1}\ge 1$ large enough so that it is non-decreasing in $a$. Indeed, the only possibly negative term is $ a \left\langle \nabla \bar{g}_{n-1}(\param_{n-1}),\,  \param-\param_{n-1} \right\rangle$, but its absolute value is at most $a c_{2} \lVert \param-\param_{n-1} \rVert$ for some constant $c_{2}>0$ by Cauchy-Schwarz inequality and Proposition  \ref{prop:g_grad_Lipschitz}. We will make a such choice for $c_{1}$. Then by choosing $a = b_{n}/\lVert \param - \param_{n-1} \rVert$, we obtain 
		\begin{align}\label{eq:first_order_opt_pf4_PR}
			\left\langle \nabla \bar{g}_{n}(\param_{n-1}),\,   \param_{n}-\param_{n-1} \right\rangle &\le 
			b_{n} \left\langle \nabla \bar{g}_{n-1}(\param_{n-1}),\,  \frac{\param-\param_{n-1}}{\lVert \param-\param_{n-1} \rVert}  \right\rangle \\
			&\qquad + c_{1} b_{n}^{2}(L+\hat{\rho})  + L'w_{n} b_{n} + \frac{L}{2} \lVert \param_{n}-\param_{n-1} \rVert^{2} + \Delta_{n}.
		\end{align}
		
		We have shown that the above holds  for all $\param\in \Param$ such that $\lVert \param - \param_{n} \rVert \le   c'b_{n+1}$. It remains to argue  that \eqref{eq:first_order_optimality} also holds for all $\param\in\Param$ with $\lVert \param-\param_{n} \rVert\ge c'b_{n+1}$. Indeed, for such $\param$, let $\param'$ be the point in the secant line between $\param$ and $\param_{n}$ such that $\lVert \param'-\param_{n} \rVert \le c'b_{n+1}$. Then $\param'\in \Param$ and \eqref{eq:first_order_optimality} holds for $\param$ replaced with $\param'$. However, the right-hand side is unchanged when replacing $\param$ with any point on the line passing through $\param$ and $\param_{n}$. Then the assertion follows by dividing both sides by $b_{n}$ and taking the conditional expectation with respect to $\mathcal{F}_{n-1}$. 
		
		Next, we show the assertion for case \ref{C2}. The argument is similar but a bit more complicated due to the use of randomized block coordinate descent in Algorithm \ref{algorithm:BSM-DR}. As before, fix arbitrary $\param= \Param$ such that $\lVert \param - \param_{n-1} \rVert \le   c'b_{n}$.  We will suppress the dependency of random block coordinate $J_{i}(n)$ on $n$ below and denote it $J_{i}$. Let $\Param_{n}^{J_{i}}$ be as in \eqref{eq:param_J_n_def}, which is a convex subset of $\Param$ that contains both $\param_{n}^{(i-1)}$ and $\param_{n}^{(i-1)} + \sum_{j\in J_{i}}(\param-\param_{n}^{(i-1)})^{j}e_{j}$, where $e_{j}$ denote the $j$th standard basis vector of $\R^{p}$. It follows that 
		\begin{align}
			\param_{n}^{(i-1)} + a \sum_{j\in J_{i}}(\param-\param_{n}^{(i-1)})^{j}e_{j} \in \Param^{J_{i}}_{n} \qquad \textup{for $a\in [0,1]$}. 
		\end{align}
		Recalling that $\param_{n}^{(i)}$ is an approximate minimizer of $\bar{g}_{n}(\param)$ over $\Param_{n}^{J_{i}}$ with gap $\Delta_{n}^{(i)}$ (see \ref{assumption:A5_sufficient_surrogate_decay}),  
		\begin{align}\label{eq:first_order_opt_pf1}
			\bar{g}_{n}(\param_{n}^{(i)}) - \Delta_{n}^{(i)} 
			&\le  \bar{g}_{n}(\param_{n}^{(i\star)})  \le  \bar{g}_{n}\left( \param_{n}^{(i-1)} + a  \left( \sum_{j\in J_{i}}(\param-\param_{n}^{(i-1)})^{j}e_{j} \right)  \right), 
		\end{align}
		where $\param_{n}^{(i\star)}$ denotes an exact minimizer of $\bar{g}_{n}$ over $\Param_{n}^{J_{i}}$. 
		Then subtract $\bar{g}_{n}(\param_{n}^{(i-1)})$ from both sides and use Lemma \ref{lem:surrogate_L_gradient}. Note that by Proposition \ref{prop:g_grad_Lipschitz},  $\nabla \bar{g}_{n}$ is $L$-Lipschitz and $\lVert \nabla g_{n+1} - \nabla g_{n} \rVert\le L'w_{n+1}$ for some constant $L'>0$. Also, note that the coordinate blocks $J_{1},\dots,J_{m}$ are disjoint so that each coordinate of $\param_{n-1}$ is updated at most once, see \ref{assumption:A6-faithful_sampling}. This implies 
		\begin{align}
			\sum_{j\in J_{i}}(\param-\param_{n}^{(i-1)})^{j}e_{j} = \sum_{j\in J_{i}}(\param-\param_{n-1})^{j}e_{j},\qquad \lVert \param_{n}^{(i-1)}-\param_{n-1} \rVert,\, \lVert \param_{n}^{(i)}-\param_{n}^{(i-1)} \rVert\le \lVert \param_{n}-\param_{n-1} \rVert.
		\end{align}
		So it follows that 
		\begin{align}\label{eq:first_order_opt_pf2}
			&\left\langle \nabla \bar{g}_{n}(\param_{n}^{(i-1)}),\,   \param_{n}^{(i)}-\param_{n}^{(i-1)} \right\rangle - \frac{L}{2} \lVert \param_{n}^{(i)}-\param_{n}^{(i-1)} \rVert^{2} - \Delta_{n}^{(i)} \\
			&\qquad \le  a \left\langle \nabla \bar{g}_{n}(\param_{n}^{(i-1)}),\,  \sum_{j\in J_{i}}(\param-\param_{n-1})^{j}e_{j}\right\rangle + \frac{a^{2}L }{2}  \lVert \param - \param_{n-1} \rVert^{2}  \\
			&\qquad \le  a \left\langle \nabla \bar{g}_{n-1}(\param_{n}^{(i-1)}),\,  \sum_{j\in J_{i}}(\param-\param_{n-1})^{j}e_{j}\right\rangle + \frac{a^{2}L }{2}  \lVert \param - \param_{n-1} \rVert^{2} + aL'w_{n} \lVert \param-\param_{n-1} \rVert \\
			&\qquad \le  a \left\langle \nabla \bar{g}_{n-1}(\param_{n-1}),\,  \sum_{j\in J_{i}}(\param-\param_{n-1})^{j}e_{j}\right\rangle + \frac{a^{2}L }{2}  \lVert \param - \param_{n-1} \rVert^{2} + aL'w_{n} \lVert \param-\param_{n-1} \rVert \\
			&\hspace{3cm} + aL\lVert \param_{n}-\param_{n-1} \rVert \, \lVert \param-\param_{n-1} \rVert.
		\end{align}
		Also, by Cauchy-Schwarz inequality and $L$-Lipschitz continuity of $\nabla \bar{g}_{n}$,  we have 
		\begin{align}
			\left\langle \nabla \bar{g}_{n}(\param_{n}^{(i-1)}),\,   \param_{n}^{(i)}-\param_{n}^{(i-1)} \right\rangle &\ge \left\langle \nabla \bar{g}_{n}(\param_{n-1}),\,   \param_{n}^{(i)}-\param_{n}^{(i-1)} \right\rangle - L \lVert \param_{n}^{(i-1)} - \param_{n-1} \rVert \, \lVert \param_{n}^{(i)} - \param_{n}^{(i-1)} \rVert \\
			&\ge \left\langle \nabla \bar{g}_{n-1}(\param_{n-1}),\,   \param_{n}^{(i)}-\param_{n}^{(i-1)} \right\rangle - L \lVert \param_{n} - \param_{n-1} \rVert^{2}.
		\end{align}
		Thus combining the above  inequalities, we obtain 
		\begin{align}
			\left\langle \nabla \bar{g}_{n}(\param_{n-1}),\,   \param_{n}^{(i)}-\param_{n}^{(i-1)} \right\rangle &\le  	a \left\langle \nabla \bar{g}_{n-1}(\param_{n-1}),\,  \sum_{j\in J_{i}}(\param-\param_{n-1})^{j}e_{j}\right\rangle  + c_{1} \lVert \param_{n}-\param_{n-1} \rVert^{2} + a^{2} c_{2} \lVert \param-\param_{n-1} \rVert^{2}  \\
			&\qquad + a c_{3} w_{n} \lVert \param-\param_{n-1} \rVert + ac_{4} \lVert \param_{n}-\param_{n-1} \rVert \, \lVert \param-\param_{n-1} \rVert + \Delta_{n}^{(i)},
		\end{align}
		where $c_{1},\dots,c_{4}>0$ are constants. 
		Hence summing over $i=1,\dots,m$,
		\begin{align}\label{eq:first_order_opt_pf3}
			&	\left\langle \nabla \bar{g}_{n}(\param_{n-1}),\,   \param_{n} -\param_{n-1}\right\rangle   \\
			&\hspace{2cm} \le  a  \left\langle \nabla \bar{g}_{n-1}(\param_{n-1}),\, \sum_{i=1}^{m} \sum_{j\in J_{i}}(\param-\param_{n-1})^{j}e_{j} \right\rangle    + c_{1}'  \lVert \param_{n} - \param_{n-1} \rVert^{2}  +  a^{2} c_{2} ' L \lVert\param-\param_{n-1} \rVert^{2}  \\
			&\hspace{3cm} +  a c_{3}'  w_{n}  \lVert\param-\param_{n-1} \rVert + a c_{4}'  \lVert \param-\param_{n-1} \rVert \,  \lVert \param_{n} - \param_{n-1} \rVert + \Delta_{n},
		\end{align}
		where we denoted $c_{i}'=mc_{i}$ for $i=1,\dots,4$.  Note that \eqref{eq:first_order_opt_pf3} holds for all $a\in [0,1]$. By Cauchy-Schwarz inequality, Proposition \ref{prop:g_grad_Lipschitz}, and $\Param$ is compact, 
		\begin{align}
			\left|  \left\langle \nabla \bar{g}_{n-1}(\param_{n-1}),\, \sum_{i=1}^{m} \sum_{j\in J_{i}}(\param-\param_{n-1})^{j}e_{j} \right\rangle   \right| ,\, \lVert \param_{n}-\param_{n-1} \rVert \, \lVert \param-\param_{n-1} \rVert   \le c_{5} \lVert \param-\param_{n-1} \rVert 
		\end{align}
		for some constant $c_{5}>0$ independent of $\param$ (and hence $b_{n}$). Hence by choosing $c_{2}'$ in \eqref{eq:first_order_opt_pf3} large enough, we can make the right-hand side in \eqref{eq:first_order_opt_pf3} (viewed as a quadratic function in $a$) non-decreasing in $a$. We will make a such choice for $c_{2}'$. Thus \eqref{eq:first_order_opt_pf3} holds for all $a\ge 0$. So we can choose $a=b_{n}/\lVert \param - \param_{n-1} \rVert$ in \eqref{eq:first_order_opt_pf3}, divide both sides by $b_{n}$, and take the conditional expectation with respect to $\mathcal{F}_{n-1}$. Note that  \ref{assumption:A6-faithful_sampling}, for any $\param\in \Param$, 
		\begin{align}
			\E_{J_{1}(n),\dots,J_{m}(n)}\left[ \sum_{i=1}^{m} \sum_{j\in J_{i}(n)} (\param-\param_{n-1})^{j}e_{j} \,\bigg|\, \mathcal{F}_{n-1} \right] = \bar{c} (\param-\param_{n-1}), 
		\end{align}
		where $\bar{c}$ denotes the common  expected number of each coordinate $j\in \{1,\dots,m\}$ appearing in the coordinate blocks $J_{1}(n),\dots,J_{m}(n)$. This will yield \ref{eq:first_order_optimality} where the right hand side is for all $\param\in \Param$  such that $\lVert \param-\param_{n-1} \rVert\le c'b_{n}$. By using a similar argument as before, it can be easily extended for all $\param\in \Param$. 
	\end{proof}

	Now we show Lemma \ref{lem:finite_variation_surr}. 
	
	\begin{proof}[\textbf{Proof of Lemma \ref{lem:finite_variation_surr}}] 	
		By using Lemma  \ref{lem:first_order_optimality} with deterministic $b_{n}=w_{n}$ and taking the conditional expectation with respect to $\mathcal{F}_{n}$,  we have 
		\begin{align}\label{eq:pf_lem_finite_var}
			w _{n+1} \left[ - \inf_{\param\in \Param  } \left\langle \nabla \bar{g}_{n}(\param_{n}),\, \frac{\param - \param_{n} }{\lVert \param - \param_{n}\rVert}\right\rangle \right]  &\le    \E\left[   \left\langle \nabla \bar{g}_{n+1}(\param_{n}),\,  \param_{n+1} - \param_{n}  \right\rangle   \,\bigg|\, \mathcal{F}_{n}  \right]  \\
			&\hspace{2cm} + cw_{n+1}^{2}  + \E\left[   \Delta_{n+1}   \right] +  \E\left[ \lVert \param_{n+1}-\param_{n} \rVert^{2}  \,\bigg|\, \mathcal{F}_{n}  \right] 
		\end{align}
		for some constant $c>0$ for all $n\ge 1$. 	By Proposition \ref{prop:linear_surr_variation} and \ref{assumption:A4-w_t}, the right-hand side above except for the last term is summable. For the last term, note that by Lemma \ref{lem:pos_variation} \textbf{(vi)} and Fubini's theorem,  we have 
		\begin{align}
			\E\left[  \sum_{n=1}^{\infty}  \E\left[ \lVert \param_{n+1}-\param_{n} \rVert^{2}  \,\bigg|\, \mathcal{F}_{n}  \right] \right] =	\E\left[ \sum_{n=1}^{\infty} \lVert \param_{n+1}-\param_{n} \rVert^{2}   \right]		<\infty. 
		\end{align}
		Hence $ \sum_{n=1}^{\infty}  \E\left[ \lVert \param_{n+1}-\param_{n} \rVert^{2}  \,\bigg|\, \mathcal{F}_{n}  \right] <\infty$ almost surely. Thus the right-hand side of \eqref{eq:pf_lem_finite_var} is summable almost surely. 
	\end{proof}
	
	\vspace{0.3cm}
	\section{Proof of Lemma \ref{lem:asymptotic_stationarity} } 
	\label{sec:lem_asymptotic_stationarity}

	In this section, we prove the only remaining key lemma, Lemma \ref{lem:asymptotic_stationarity}, on asymptotic stationarity of the iterates with respect to the averaged surrogate loss functions. We proceed by first proving the assertion for cases \ref{C1} and \ref{C3} and then for the case \ref{C2} in Theorem \ref{thm:global_convergence}.

	Throughout this section, we will make use of the parameterized surrogate assumption \ref{assumption:A7_param_surr}. From this, there exists a compact set $\mathcal{K}$ and a function $\hat{g}:\mathcal{K}\times \Param\rightarrow [0,\infty)$ such that for all $n\ge 1$, $\bar{g}_{n}(\param)= \hat{g}(\kappa_{n}, \param)$ for some $\kappa_{n}\in \mathcal{K}$.  Furthermore, $\hat{g}$ is Lipschitz in the first coordinate. Hence by taking a subsequence of $(t_{k})_{k\ge 1}$, we may assume that $\kappa_{\infty}:=\lim_{k\rightarrow \infty} \kappa_{t_{k}}$. Hence the function $\bar{g}_{\infty}:=\lim_{k\rightarrow \infty}\bar{g}_{t_{k}} = \hat{g}(\kappa_{\infty}, \cdot)$ is well-defined. By continuity of $\nabla \bar{g}_{n}$ (see Lemma \ref{lem:surrogate_L_gradient}), we also have $\nabla g_{t_{k}}\rightarrow \nabla g_{\infty}$ almost surely. 
	
	\subsection{  Proof of Lemma \ref{lem:asymptotic_stationarity}  for cases  \ref{C1} and \ref{C3} in Theroem \ref{thm:global_convergence}.  }
	
	In this subsection, we prove Lemma \ref{lem:asymptotic_stationarity} assuming cases \ref{C1} and \ref{C3} in Theroem \ref{thm:global_convergence}. The argument for case \ref{C3} may require a bit of care since we cannot use per-iteration stability $\lVert \param_{n}-\param_{n-1} \rVert = O(w_{n})$ (see Lemma \ref{lem:stability}) and rely on the total stability $\E\left[ \sum_{n=1}^{\infty} \lVert \param_{n}-\param_{n-1} \rVert  \right]<\infty$ (see Lemma \ref{lem:pos_variation} \textbf{(vi)}). Nontheless, Lemma \ref{lem:asymptotic_stationarity} for both cases \ref{C1} and \ref{C3} can be shown using a simple argument that only uses $\Vert \param_{n}-\param_{n-1} \rVert=o(1)$. 
	
	We begin with a simple observation. 
	
	\begin{prop}\label{prop:iterate_opt_gap}
		Assume the cases \textbf{\ref{C1}} or \ref{C3} in Theorem \ref{thm:global_convergence}. Take $\hat{\rho}=0$ for case \ref{C1} and $\hat{\rho}>-\rho$ for case \ref{C2}. For each $n\ge 1$, let $\param_{n}^{\star}$ be the minimizer of the strongly convex function $\param \mapsto  \bar{G}_{n}(\param):=\bar{g}_{n}(\param)+\frac{\hat{\rho}}{2}\lVert \param-\param_{n-1} \rVert^{2}$ over the convex set $\Param$. Then there exists a constant $\alpha>0$ such that for all $n\ge 1$, 
		\begin{align}\label{eq:iterate_optimality_gap_prox}
			\lVert \param_{n} - \param_{n}^{\star} \rVert^{2} \le  \alpha \Delta_{n}.
		\end{align}
	\end{prop}
	
	\begin{proof}
		Note that $\bar{G}_{n}$ is $\gamma$-strongly convex with modulus of convexity $\gamma=\rho>0$ for case \ref{C1} and $\gamma=(\hat{\rho}-|\rho|)$ for case \ref{C2}. Then the assertion follows from 
		\begin{align}\label{eq:2nd_order_growth_prox}
			\frac{\gamma}{2} \lVert \param_{n} - \param_{n}^{\star} \rVert^{2}  \le \bar{G}_{n}(\param_{n}) - \bar{G}_{n}(\param_{n}^{\star} ) \le \Delta_{n}
		\end{align}
		for $n\ge 1$. Indeed, the first inequality follows from the second-order growth property (see Lemma \ref{lem:second_order_growh_univariate}) since  $\bar{G}_{n}^{(n)}$ is a $\gamma$-weakly convex function minimized at $\param_{n}$ over $\Param$, and the second inequality follows from the definition of optimality gap $\Delta_{n}$ in \ref{assumption:A5_sufficient_surrogate_decay}. 
	\end{proof}

	We are now ready to give proof of Lemma \ref{lem:asymptotic_stationarity} for no regularization (case \ref{C1}) and proximal regularization (case \ref{C3}). 
	
	\vspace{0.1cm}
	\begin{proof}[\textbf{Proof of Lemma \ref{lem:asymptotic_stationarity} for cases \ref{C1} and \ref{C3} in Theorem \ref{thm:global_convergence}}] 
		Assume \ref{assumption:A1-ell_smooth}-\ref{assumption:A4-w_t}, \ref{assumption:A7_param_surr}, and one of the cases \textbf{\ref{C1}} or \ref{C3} in Theorem \ref{thm:global_convergence}.  Fix an arbitrary convergent subsequence $(\param_{t_{k}})_{k\ge 1}$ of $(\param_{n})_{n\ge 1}$ denote $\param_{\infty}:=\lim_{k\rightarrow \infty} \param_{t_{k}}$. By taking a further subsequence, we may assume that the following limits $\bar{g}_{\infty}:=\lim_{k\rightarrow \infty} \bar{g}_{t_{k}}$ and $\nabla \bar{g}_{\infty}:=\lim_{k\rightarrow \infty} \nabla \bar{g}_{t_{k}}$ exist almost surely. It suffies to show that $\param_{\infty}$ is a stationary point of $\bar{g}_{\infty}$ over $\Param$.  
		
		We keep the same notation for $\hat{\rho}$ as in Proposition \ref{prop:iterate_opt_gap}. By using the first-order optimality of $\param_{n}^{\star}\in \argmin_{\param\in \Param}  \left( \bar{g}_{n}(\param) + \frac{\hat{\rho}}{2}\lVert \param-\param_{n-1} \rVert^{2} \right)$, we have  
		\begin{align}
			\left\langle \nabla \bar{g}_{n} (\param_{n}^{\star}) + \hat{\rho} (\param_{n}^{\star}- \param_{n-1}) ,\, \param - \param_{n}  \right\rangle \ge 0 \qquad \forall \param\in \Param.
		\end{align}
		On the other hand, using Cauchy-Schwarz inequality and Proposition \ref{prop:iterate_opt_gap}, 
		\begin{align}
			&\left| \left\langle \nabla \bar{g}_{n} ( \param_{n}) + \hat{\rho} (\param_{n}- \param_{n-1}) ,\, \param - \param_{n}  \right\rangle  - 	\left\langle \nabla \bar{g}_{n} (\param_{n}^{\star}) + \hat{\rho} (\param_{n}^{\star}- \param_{n-1}) ,\, \param - \param_{n}  \right\rangle \right| \\
			&\qquad \le \left| \left\langle \nabla \bar{g}_{n} (\param_{n})  - \nabla \bar{g}_{n} (\param_{n}^{\star})  + \hat{\rho} (\param_{n}- \param_{n}^{\star}) ,\, \param - \param_{n}  \right\rangle  \right|  \\
			&\qquad \le  \left( \lVert  \nabla \bar{g}_{n} (\param_{n})  -  \nabla \bar{g}_{n} (\param_{n}^{\star})   \rVert   + \hat{\rho} \lVert \param_{n} - \param_{n}^{\star} \rVert  \right) \lVert \param - \param_{n} \rVert  \\
			&\qquad \le  (L+\hat{\rho})   \lVert \param - \param_{n} \rVert \,   \lVert \param_{n}^{\star} - \param_{n} \rVert   \\
			&\qquad \le  (L+\hat{\rho})   \lVert \param - \param_{n} \rVert \, \sqrt{\alpha \Delta_{n}}.
		\end{align}
		Similarly, also note that 
		\begin{align}
			|\left\langle  \hat{\rho} (\param_{n}^{\star}- \param_{n-1}) ,\, \param - \param_{n}  \right\rangle| \le \hat{\rho} \lVert \param - \param_{n} \rVert \, \lVert \param_{n}^{\star} - \param_{n} \rVert   \le \hat{\rho} \lVert \param-\param_{n} \rVert  \lVert \param - \param_{n} \rVert \, \sqrt{\alpha \Delta_{n}}.
		\end{align}
		Note that from \ref{assumption:A5_sufficient_surrogate_decay}, $\E[\sum_{n=1}^{\infty} \Delta_{n}]=0$, and by Fubini's theorem, this implies that $\sum_{n=1}^{\infty} \Delta_{n}<\infty$ almost surely. Hence $\Delta_{n}=o(1)$. Furthermore, since $\param_{t_{k}}$ converges, the sequence $(\param_{t_{k}})_{k\ge 1}$ is bounded. So the above inequalities  imply
		\begin{align}
			\liminf_{k\rightarrow \infty} \, \left\langle \nabla \bar{g}_{t_{k}} (\param_{t_{k}}) ,\, \param - \param_{t_{k}}  \right\rangle \ge 0 \qquad \forall \param\in \Param.
		\end{align}
		Since $\nabla \bar{g}_{t_{k}}\rightarrow \nabla \bar{g}_{\infty}$ as $k\rightarrow\infty$, we conclude that $\left\langle \nabla \bar{g}_{\infty} (\param_{\infty} ) ,\, \param - \param_{\infty}  \right\rangle   \ge 0 $ for all $\param\in \Param$,  which means that $\param_{\infty}$ is a stationary point of $\bar{g}_{\infty}$ over $\Param$. 
	\end{proof}

	\subsection{  Proof of Lemma \ref{lem:asymptotic_stationarity}  for cases \ref{C2} in Theroem \ref{thm:global_convergence}.  }
	
	A key property of the iterates $(\param_{n})_{n\ge 1}$ for cases \ref{C1}-\ref{C2} in Theroem \ref{thm:global_convergence} is their stability, $\lVert \param_{n}-\param_{n-1} \rVert=O(w_{n})$, which is given in Lemma \ref{lem:stability}. For \ref{C2}, which we consider in this subsection, we have an extra complication due to the use of an additional diminishing radius condition. Roughly speaking, for each limit point $\param_{\infty}$ of the iterates  $(\param_{n})_{n\ge 1}$, we will need to show that $\param_{\infty}$ is not touching any trust region boundary we used along the way and it is indeed a stationary point of some limiting function of the averaged surrogates $\bar{g}_{n}$. The argument we provide here is adapted from \cite{lyu2020convergence}, which was first developed for (deterministic) block coordinate descent with diminishing radius.

	We first show a sufficient condition for a subsequence $(\param_{t_{k}})_{k\ge 1}$ to converge to a stationary point of a limiting averaged surrogate function. 
	
	\begin{prop}\label{prop:stationary_conditions}
		Assume \ref{assumption:A1-ell_smooth}-\ref{assumption:A4-w_t}, \ref{assumption:A7_param_surr}, and the case \ref{C2} in Theorem \ref{thm:global_convergence}. Further assume $ \Delta_{n}=o(\lVert \param_{n}-\param_{n-1} \rVert)$. 
		Suppose there exists a sequence $(t_{k})_{k\ge 1}$ in $\mathbb{N}$ such that almost surely, either 
		\begin{align}\label{eq:stationary_conditions}
			\sum_{k=1}^{\infty}  \lVert \param_{t_{k}+1} - \param_{t_{k}}  \rVert = \infty  \qquad \text{or} \qquad 	\liminf_{k\rightarrow \infty}\,\, \,  \E\left[  \left| \left\langle \nabla \bar{g}_{t_{k}+1}(\param_{t_{k}}),\,  \frac{  \param_{t_{k}+1} - \param_{t_{k}} }{ \lVert \param_{t_{k}+1} - \param_{t_{k}}  \rVert }  \right\rangle \right|   \, \bigg| \, \mathcal{F}_{t_{k}}\right] = 0.
		\end{align}
		Then there exists a further subsequence $(s_{k})_{k\ge 1}$ of $(t_{k})_{k\ge 1}$ such that $\bar{g}_{\infty}:=\lim_{k\rightarrow \infty} \bar{g}_{s_{k}}$ and $\param_{\infty}:=\lim_{k\rightarrow \infty} \param_{s_{k}}$ exist almost surely and $\param_{\infty}$ and is a stationary point of $\bar{g}_{\infty}$ over $\Param$. 
	\end{prop}
	
	\begin{proof}
		Proposition \ref{prop:linear_surr_variation}, we have that for all $i=1,\dots,m$, almost surely, 
		\begin{align}
			\sum_{k=1}^{\infty}  \lVert \param_{t_{k}+1} - \param_{t_{k}}  \rVert  \, \E\left[    \sum_{i=1}^{m} \left| \left\langle \nabla \bar{g}_{t_{k}+1}(\param_{t_{k}}),\,  \frac{ \param_{t_{k}+1} - \param_{t_{k}}}{ \lVert \param_{t_{k}+1} - \param_{t_{k}}  \rVert }   \right\rangle \right| \,\mathcal{F}_{t_{k}}\right] <\infty.
		\end{align}
		Thus the former condition in the assertion implies the latter almost surely. Thus it suffices to show that this latter condition implies the assertion. Assume this condition, and
		let $(s_{k})_{k\ge 1}$ be a subsequence of $(t_{k})_{k\ge 1}$ for which the  liminf in \eqref{eq:stationary_conditions} is achieved.  Using \ref{assumption:A7_param_surr}, we may  take a subsequence along which the following limits  $\param_{\infty}=\lim_{k\rightarrow \infty} \param_{s_{k}}$, $\bar{g}_{\infty}:=\lim_{k\rightarrow\infty} \bar{g}_{s_{k}}$, and $\nabla \bar{g}_{\infty}:=\lim_{k\rightarrow\infty} \nabla \bar{g}_{s_{k}}$ exist. 
		
		Now suppose for contradiction that $\param_{\infty}$ is not a stationary point of $\bar{g}_{\infty}$ over $\Param$. Then there exists $\param^{\star}\in \Param$ and $\delta>0$ such that 
		\begin{align}
			\left\langle \nabla \bar{g}_{\infty}(\param_{\infty}),\, \param^{\star}-\param_{\infty} \right\rangle<-\delta<0. 
		\end{align}
		Necessarily $\param^{\star}\ne \param_{\infty}$. By the triangle inequality, write 
		\begin{align}
			& \lVert \left\langle \nabla \bar{g}_{s_{k}+1}(\param_{s_{k}}),\, \param^{\star}-\param_{s_{k}} \right\rangle - \left\langle \nabla \bar{g}_{\infty}(\param_{\infty}),\, \param^{\star}-\param_{\infty} \right\rangle \rVert \\
			&\qquad \le \lVert \nabla \bar{g}_{s_{k}+1}(\param_{s_{k}}) - \nabla \bar{g}_{\infty}(\param_{\infty}) \rVert\cdot \lVert \param^{\star}-\param_{s_{k}} \rVert +  \lVert \nabla \bar{g}_{\infty}(\param_{\infty}) \rVert \cdot \lVert \param_{\infty} - \param_{s_{k}} \rVert.
		\end{align}
		By the choice of the subsequence $s_{k}$, we see that the right-hand side goes to zero as $k\rightarrow \infty$. Hence for all sufficiently large $k\ge 1$, we have 
		\begin{align}
			\left\langle \nabla \bar{g}_{s_{k}+1}(\param_{s_{k}}) ,\, \param^{\star}-\param_{s_{k}} \right\rangle< -\delta/2.
		\end{align}
		Note that by  Lemma \ref{lem:stability},   $\lVert \param_{n}-\param_{n-1} \rVert \le cw_{n}$ for all $n\ge 1$ for some constant $c>0$. Applying Proposition \ref{lem:first_order_optimality} with $b_{n}=(c'/c)  \lVert \param_{n}-\param_{n-1} \rVert  \le w_{n}$, we get 
		\begin{align}
			\E\left[ \left\langle \nabla \bar{g}_{n}(\param_{n}),\,   \frac{\param_{n}-\param_{n-1} }{ \lVert  \param_{n}-\param_{n-1} \rVert} \right\rangle \,\bigg|\, \mathcal{F}_{n-1}  \right] 
			&\le c_{1}  \inf_{\param\in \Param  } \left\langle \nabla \bar{g}_{n}(\param_{n}),\, \frac{\param - \param_{n} }{\lVert \param - \param_{n}\rVert}\right\rangle   + c_{3}w_{n}  \\
			&\qquad +   c_{3} \E\left[  \frac{ \Delta_{n} }{\lVert \param_{n}-\param_{n-1} \rVert} \,\bigg| \, \mathcal{F}_{n-1} \right]
		\end{align}
		for some constants $c_{1},c_{2},c_{3}>0$ for all $n\ge 1$.  Hence by the hypothesis on $\Delta_{n}$, and that $w_{n}=o(1)$, denoting $\lVert \Param \rVert:=\sup_{\param,\param'\in \Param} \lVert \param-\param' \rVert<\infty$, we conclude
		\begin{align}
			\liminf_{k\rightarrow \infty}\,\, \,  \E\left[	 \left\langle \nabla \bar{g}_{s_{k}+1}(\param_{s_{k}}),\,     \frac{\param_{s_{k}}-\param_{s_{k}} }{\lVert  \param_{s_{k}}-\param_{s_{k}}  \rVert} \right\rangle  \right]  \le -\frac{c_{1} \delta}{2 \lVert \Param \rVert} < 0. 
		\end{align}
		However, this contradicts the choice of the subsequence $(\param_{s_{k}})_{k\ge 1}$ and the second condition in \eqref{eq:stationary_conditions}.  This shows the assertion. 
	\end{proof}

	%In the rest of this subsection, we prove Lemma \ref{lem:main_analytic_lemma} \textbf{\textup{(iv)}}. We first introduce some notation. 

	Recall that during the update $\param_{n-1}\rightarrow \param_{n-1}^{(1)}\rightarrow \dots \rightarrow \param_{n-1}^{(m)}=\param_{n}$ in Algorithm \ref{algorithm:BSM-DR},  each block coordinate 
	changes by at most $c'w_{n}$ in Frobenius norm. For each $n\ge 1$, we say $\param_{n}$ is a \textit{long point} if none of the block coordinates of $\param_{n-1}$ change by $c'w_{n}$ in Frobenius norm and \textit{short point} otherwise. Observe that if $\param_{n}$ is a long point, then imposing the search radius restriction in \eqref{eq:BSM_factor_update_DR} has no effect, and $\param_{n}$ is obtained from $\param_{n-1}$ by block-minimizing  $\bar{g}_{n}$ over $\Param$ without any radius restriction. 
	
	\begin{prop}\label{prop:long_points_stationary}
		Assume \ref{assumption:A1-ell_smooth}-\ref{assumption:A4-w_t}, \ref{assumption:A7_param_surr}, and the case \ref{C2} in Theorem \ref{thm:global_convergence}. If $(t_{k})_{k\ge 1}$ is such that $\bar{g}_{\infty}:=\lim_{k\rightarrow \infty} \bar{g}_{t_{k}}$ and $\param_{\infty}:=\lim_{k\rightarrow \infty} \param_{t_{k}}$ exist almost surely, then $\param_{\infty}$ is a stationary point of $\bar{g}_{\infty}$ over $\Param$. 
	\end{prop}
	
	\begin{proof}
		%According to \ref{assumption:A7_param_surr},  there exists a compact set $\mathcal{K}$ and a function $\hat{g}:\mathcal{K}\times \Param\rightarrow [0,\infty)$ such that for all $n\ge 1$, $\bar{g}_{n}(\param)= \hat{g}(\kappa_{n}, \param)$ for some $\kappa_{n}\in \mathcal{K}$.  Furthermore, $\hat{g}$ is Lipschitz in the first coordinate. Hence by taking a subsequence of $(t_{k})_{k\ge 1}$, we may assume that $\kappa_{\infty}:=\lim_{k\rightarrow \infty} \kappa_{t_{k}}$. Hence the function $\bar{g}_{\infty}:=\lim_{k\rightarrow \infty}\bar{g}_{t_{k}} = \hat{g}(\kappa_{\infty}, \cdot)$ is well-defined. 
		
		The argument is similar to that of \cite[Prop. 2.7.1]{bertsekas1997nonlinear}. However, here we do not need to assume the uniqueness of solutions to minimization problems of $\bar{g}_{t}$ in each block coordinate due to the added search radius restriction (recall that we assume $c'<\infty$ in the current section), and also note that we are considering a random choice of block coordinates in each application of Algorithm \ref{algorithm:BSM-DR}. 
		
		Recall that $J_{1}(n),\dots,J_{m}(n)$ denotes the coordinate blocks chosen by Algorithm \ref{algorithm:BSM-DR} at iteration $n$ of Algorithm \ref{algorithm:SMM}, and by \ref{assumption:A6-faithful_sampling}, their joint distribution does not depend on $n$ and also they are independent of everything else. Fix a coordinate block $J\in \mathbb{J}$ that has a positive probability of being chosen by $J_{1}(n)$. Then $J_{1}(t_{k})=J$ for infinitely many $k$'s almost surely by a Borel-Cantelli lemma. We may refine the subsequence  $(t_{k})_{k\ge 1}$ so that $J_{1}(t_{k})=J$ for all $k\ge 1$ almost surely.  For each $\param\in \Param$,  we write $\param = [\theta^{J}, \theta^{J^{c}}]$, where $\param^{J}$ is the projection of $\param$ onto the coordinate block $J$ and similarly for the complementary coordinate block $J^{c}$. In this fashion, we denote $\param_{n} = [\theta_{n}^{J}, \, \theta_{n}^{J^{c}}]$ and $\param_{\infty} = [\theta_{\infty}^{J}, \, \theta_{\infty}^{J^{c}}]$. Recall the block update $\param_{t_{k}-1}\rightarrow \param_{t_{k}-1}^{(1)}\rightarrow \dots \rightarrow \param_{t_{k}-1}^{(m)}=\param_{t_{k}}$ and denote $\param_{t_{k}-1;1} := \param_{t_{k}-1}^{(1)}$. Then since $\param_{t_{k}}$ is a long point, 
		\begin{align}\label{eq:pf_stationarity_long_point1}
			\bar{g}_{t_{k}}\left( \theta_{t_{k}-1;1}^{J}, \,\theta_{t_{k}-1;1}^{J^{c}}\right) \le \bar{g}_{t_{k}}\left(\theta, \, \theta_{t_{k}-1;1}^{J^{c}}\right)
		\end{align}
		for every $\theta \in  \Param^{J}= \textup{Proj}_{\R^{J}}(\Param)$. Indeed, in this case, his is because $\theta_{t_{k}-1;1}^{J}$ is a minimizer of $\theta\mapsto \bar{g}_{t_{k}+1}(\theta, \theta^{c}_{t_{k};1})$ over $\Param^{J}$. Noting that $\lVert \param_{n}-\param_{n-1} \rVert = O(w_{n})=o(1)$, it follows that 
		\begin{align}\label{eq:pf_stationarity_long_point2}
			\bar{g}_{t_{k}}\left( \theta_{\infty}^{J}, \,\theta_{\infty}^{J^{c}}\right) \le \bar{g}_{t_{k}}\left(\theta, \, \theta_{\infty}^{J^{c}}\right)
		\end{align}
		for all $\theta \in \Param^{J}$.  Since $\param^{J}$ is convex, it follows that 
		\begin{align}
			\left\langle \nabla_{J} \bar{g}_{\infty}(\param_{\infty}),\, \theta -\theta_{\infty}^{J}\right\rangle \ge 0 \quad \quad \text{for all $\theta \in \Param^{J}$}. 
		\end{align}
		By using a similar argument for the subsequent coordinate blocks $J_{2}(n),\dots,J_{m}(n)$, one can show that the above holds for every coordinate block $J$ that has a positive chance of being chosen by some of the blocks among $J_{1}(n),\dots,J_{m}(n)$. The union of all such $J$ covers the entire coordinates $\{1,\dots,p\}$ by \ref{assumption:A6-faithful_sampling}. Hence we deduce that $\langle \nabla \bar{g}_{\infty}(\param_{\infty}), \param-\param_{\infty} \rangle \ge 0$ for all $\param\in \Param$. This shows the assertion. 
	\end{proof}
	
	The following proposition gives a key property of a non-stationary limit point of the iterates $(\param_{n})_{n\ge 1}$, if exists.

	\begin{prop}\label{prop:non-stationary_nbh}
		Assume \ref{assumption:A1-ell_smooth}-\ref{assumption:A4-w_t}, \ref{assumption:A7_param_surr}, and the case \ref{C2} in Theorem \ref{thm:global_convergence}. Suppose there exists a sequence $(n_{k})_{k\ge 1}$ such that $\bar{g}_{\infty}:=\lim_{k\rightarrow \infty} \bar{g}_{n_{k}}$ and $\param_{\infty}:=\lim_{k\rightarrow \infty} \param_{n_{k}}$ exist almost surely and $\param_{\infty}$ is not a stationary point of $\bar{g}_{\infty}$ over $\Param$.  Then there exists $\eps>0$ such that the $\eps$-neighborhood $B_{\eps}(\param_{\infty}):=\{ \param\in \Param\,|\, \lVert \param-\param_{\infty} \rVert<\eps\}$ with the following properties: 
		\begin{description}
			\item[(a)] $B_{\eps}(\param_{\infty})$ does not contain any stationary points of $\bar{g}_{\infty}$ over $\Param$;
			
			\vspace{0.1cm} 
			\item[(b)] There exists infinitely many $\param_{t}$'s outside of $B_{\eps}(\param_{\infty})$.
		\end{description}
	\end{prop}
	
	\begin{proof}
		We will first show that there exists an $\eps$-neighborhood $B_{\eps}(\param_{\infty})$ of $\param_{\infty}$ that does not contain any long points of $\Lambda$. Suppose for contradiction that for each $\eps>0$, there exists a long point $\Lambda$ in $B_{\eps}(\param_{\infty})$. Then one can construct a sequence of long points converging to $\param_{\infty}$. But then by Proposition \ref{prop:long_points_stationary}, $\param_{\infty}$ is a stationary point, a contradiction.

		Next, we show that there exists $\eps>0$ such that $B_{\eps}(\param_{\infty})$ satisfies  \textbf{(a)}. Suppose for contradiction that there exist no such $\eps>0$. Then we have a sequence $(\param_{\infty;k})_{k\ge 1}$ of stationary points of $\Lambda$ that converges to $\param_{\infty}$. Denote the limiting surrogate loss function associated with $\param_{\infty;k}$ by $\bar{g}_{\infty;k}$. Recall that each $\bar{g}_{\infty;k}$ is parameterized by elements in a compact set according to \ref{assumption:A7_param_surr}. Hence by choosing a subsequence, we may assume that $\bar{g}_{\infty}:=\lim_{k\rightarrow \infty} \bar{g}_{\infty;k}$ is well-defined. Fix $\param\in \Param$ note that by Cauchy-Schwarz inequality,
		\begin{align}
			\left\langle	\nabla \bar{g}_{\infty}(\param_{\infty}),\, \param-\param_{\infty} \right\rangle &\ge -\lVert \nabla\bar{g}_{\infty}(\param_{\infty}) -\nabla \bar{g}_{\infty;k}(\param_{\infty;k})\rVert  \cdot \lVert \param-\param_{\infty} \rVert \\
			&\qquad -  \lVert \nabla \bar{g}_{\infty;k}(\param_{\infty;k})\rVert  \cdot \lVert \param_{\infty}-\param_{\infty;k}  \rVert + \left\langle \nabla \bar{g}_{\infty;k}(\param_{\infty;k}), \param-\param_{\infty;k} \right\rangle.
		\end{align}
		Note that $\nabla \bar{g}_{\infty;k}(\param_{\infty;k})^{T}(\param-\param_{\infty;k})\ge 0$ since $\param_{\infty;k}$ is a stationary point of $\bar{g}_{\infty;k}$ over $\Param$. Hence by taking $k\rightarrow \infty$, this shows $\nabla \bar{g}_{\infty}(\param_{\infty})^{T}(\param-\param_{\infty})\ge 0$. Since $\param\in \param^{\textup{dict}}$ was arbitrary, this shows that $\param_{\infty}$ is a stationary point of $\bar{g}_{\infty}$ over $\Param$, a contradiction.

		Lastly, from the earlier results, we can choose $\eps>0$ such that $B_{\eps}(\param_{\infty})$ has no long points of $\Lambda$ and also satisfies  \textbf{(b)}.  We will show that $B_{\eps/2}(\param_{\infty})$ satisfies \textbf{(c)}. Then $B_{\eps/2}(\param_{\infty})$ satisfies  \textbf{(a)}-\textbf{(b)}, as desired. Suppose for contradiction there are only finitely many $\param_{t}$'s outside of $B_{\eps/2}(\param_{\infty})$. Then there exists an integer $M\ge 1$ such that $\param_{t}\in B_{\eps/2}(\param_{\infty})$ for all $t\ge M$. Then each $\param_{n}$ for $n\ge M$ is a short point of $\Lambda$. By definition, it follows that $\lVert \param_{n}-\param_{n-1}\lVert_{F} \ge  w_{n}$ for all $t\ge M$. This implies $\alpha_{n}=\sqrt{\sum_{i=1}^{n} \lVert \param_{n}^{(i)}  -\param_{n-1} \rVert^{2} } \ge \lVert \param_{n+1} - \param_{n} \rVert \ge w_{n+1}$ for all $n\ge M$. Then by Proposition \ref{prop:stationary_conditions}, since $\sum_{n=1}^{\infty} w_{n}=\infty$, there exists a subsequence $(s_{k})_{k\ge 1}$ such that $\param_{\infty}':=\lim_{k\rightarrow \infty} \param_{t_{k}} $ exists and is stationary.  But since $\param'_{\infty}\in B_{\eps}(\param)$, this  contradicts \textbf{(a)} for $B_{\eps}(\param)$. This completes the proof. 
	\end{proof}

	We are now ready to give proof of Lemma \ref{lem:asymptotic_stationarity} for the diminishing radius case. 
	
	\vspace{0.1cm}
	\begin{proof}[\textbf{Proof of Lemma \ref{lem:asymptotic_stationarity} for cases \ref{C1}-\ref{C2} in Theorem \ref{thm:global_convergence}}] 
		Assume \ref{assumption:A1-ell_smooth}-\ref{assumption:A4-w_t}, \ref{assumption:A7_param_surr}, and the case \ref{C2} in Theorem \ref{thm:global_convergence}. (As mentioned at the beginning of this subsection, \ref{C1} is a special case of \ref{C2}.) Suppose there is a non-stationary limit point $\param_{\infty}$ of $\Lambda$. By Proposition \ref{prop:non-stationary_nbh}, we may choose $\eps>0$ such that $B_{\eps}(\param_{\infty})$ satisfies the conditions \textbf{(a)}-\textbf{(b)} of Proposition \ref{prop:non-stationary_nbh}. Choose $M\ge 1$ large enough so that $w_{t}<\eps/4$ whenever $t\ge M$. We call an integer interval $I:=[\ell,\ell']$ a \textit{crossing} if $\param_{\ell}\in B_{\eps/3}(\param_{\infty})$, $\param_{\ell'}\notin B_{2\eps/3}(\param_{\infty})$, and no proper subset of $I$ satisfies both of these conditions. By definition, two distinct crossings have empty intersections. Fix a crossing $I=[\ell,\ell']$, it follows that by triangle inequality,
		\begin{align}\label{eq:upcrossing_ineq}
			\sum_{t=\ell}^{\ell'-1} \lVert \param_{t+1}-\param_{t} \rVert \ge \lVert \param_{\ell'}-\param_{\ell} \rVert \ge \eps/3. 
		\end{align}
		Note that since $\param_{\infty}$ is a limit point of $\Lambda$, $\param_{t}$ visits $B_{\eps/3}(\param_{\infty})$ infinitely often. Moreover, by condition \textbf{(a)} of Proposition \ref{prop:non-stationary_nbh}, $\param_{t}$ also exits $B_{\eps}(\param_{\infty})$ infinitely often. It follows that there are infinitely many crossings.  Let $t_{k}$ denote the $k^{\textup{th}}$ smallest integer that appears in some crossing. Then $t_{k}\rightarrow \infty$ as $k\rightarrow \infty$, and by \eqref{eq:upcrossing_ineq}, 
		\begin{align}
			\sum_{k=1}^{\infty} \lVert \param_{t_{k}+1}-\param_{t_{k}} \rVert \ge (\text{$\#$ of crossings}) \, \frac{\eps}{3} = \infty. 
		\end{align}
		Then by Proposition \ref{prop:stationary_conditions}, there is a further subsequence $(s_{k})_{k\ge 1}$ of $(t_{k})_{k\ge 1}$ such that $\param_{\infty}':=\lim_{k\rightarrow \infty} \param_{s_{k}}$ exists and is stationary. However, since $\param_{t_{k}}\in B_{2\eps/3}(\param_{\infty})$, we have $\param_{\infty}'\in B_{\eps}(\param_{\infty})$. This contradicts the condition \textbf{(b)} of Proposition \ref{prop:non-stationary_nbh} for $B_{\eps}(\param_{\infty})$ that it cannot contain any stationary point of $\Lambda$. This shows the assertion. 
	\end{proof}

	\section*{Acknowledgements}

	The author appreciates helpful discussions with Stephen Wright, Rayan Saab,  Meisam Razaviyayn, and Wotao Yin. This work was partially supported by the National Science Foundation through grants DMS-2206296 and DMS-2010035.
	
	\iffalse
	\vspace{0.3cm}
	\small{
		\bibliographystyle{amsalpha}
		\bibliography{mybib}
	}
	\fi

	\newpage 
	\appendix 
	\section{Background on Markov chains and MCMC}
	
	\label{sec:MC_intro}
	
	\subsection{Markov chains}\label{subsection:MC}
	Here we give a brief account of Markov chains on countable state space (see, e.g., \cite{levin2017markov}). Fix a countable set $\Omega$. A function $P:\Omega^{2} \rightarrow [0,\infty)$ is called a \textit{Markov transition matrix} if every row of $P$ sums to 1. A sequence of $\Omega$-valued random variables $(X_{t})_{t\ge 0}$ is called a \textit{Markov chain} with transition matrix $P$ if for all $x_{0},x_{1},\dots,x_{n}\in \Omega$, 
	\begin{align}\label{eq:def_MC}
		\P(X_{n}=x_{n}\,|\, X_{n-1}=x_{n-1}, \dots, X_{0}=x_{0}) = \P(X_{n}=x_{n}\,|\, X_{n-1}=x_{n-1}) = P(x_{n-1},x_{n}). 
	\end{align}
	We say a probability distribution $\pi$ on $\Omega$ a \textit{stationary distribution} for the chain $(X_{t})_{t\ge 0}$ if $\pi = \pi P$, that is, 
	\begin{align}
		\pi(x) = \sum_{y\in \Omega} \pi(y) P(y,x).
	\end{align}
	We say the chain $(X_{t})_{t\ge 0}$ is \textit{irreducible} if for any two states $x,y\in \Omega$ there exists an integer $t\ge 0$ such that $P^{t}(x,y)>0$. For each state $x\in \Omega$, let $\mathcal{T}(x) = \{ t\ge 1\,|\, P^{t}(x,x)>0 \}$ be the set of times when it is possible for the chain to return to starting state $x$. We define the \textit{period} of $x$ by the greatest common divisor of $\mathcal{T}(x)$. We say the chain $X_{t}$ is \textit{aperiodic} if all states have period 1. Furthermore, the chain is said to be \textit{positive recurrent} if there exists a state $x\in \Omega$
	such that the expected return time of the chain to $x$ started from $x$ is finite. Then an irreducible and aperiodic Markov chain has a unique stationary distribution if and only if it is positive recurrent 
	\cite[Thm 21.21]{levin2017markov}.
	
	Given two probability distributions $\mu$ and $\nu$ on $\Omega$, we define their \textit{total variation distance} by 
	\begin{align}\label{eq:def_TV}
		\lVert \mu - \nu \rVert_{TV} = \sup_{A\subseteq \Omega} |\mu(A)-\nu(A)|.
	\end{align}
	If a Markov chain $(X_{t})_{t\ge 0}$ with transition matrix $P$ starts at $x_{0}\in \Omega$, then by \eqref{eq:def_MC}, the distribution of $X_{t}$ is given by $P^{t}(x_{0},\cdot)$. If the chain is irreducible and aperiodic with stationary distribution $\pi$, then the convergence theorem (see, e.g., \cite[Thm 21.14]{levin2017markov}) asserts that the distribution of $X_{t}$ converges to $\pi$ in total variation distance: As $t\rightarrow \infty$,
	\begin{align}\label{eq:finite_MC_convergence_thm}
		\sup_{x_{0}\in \Omega} \,\lVert P^{t}(x_{0},\cdot) - \pi \rVert_{TV} \rightarrow 0.
	\end{align}
	See \cite[Thm 13.3.3]{meyn2012markov} for a similar convergence result for the general state space chains. When $\Omega$ is finite, then the above convergence is exponential in $t$ (see., e.g., 
	\cite[Thm 4.9]{levin2017markov})). Namely, there exists constants $\lambda\in (0,1)$ and $C>0$ such that for all $t\ge 0$, 
	\begin{align}\label{eq:finite_MC_convergence_thm}
		\max_{x_{0}\in \Omega} \,\lVert P^{t}(x_{0},\cdot) - \pi \rVert_{TV} \le C \lambda^{t}.
	\end{align}
	
	An important notion in MCMC sampling is ``exponential mixing'' of the Markov chain. For a simplified discussion, suppose in \ref{assumption:A2-MC} that our data points $\x_{t}$ themselves form a Markov chain with unique stationary distribution $\pi$. Under the assumption of finite state space, irreducibility, and aperiodicity in \ref{assumption:A2-MC}  the Markov chain $\x_{t}$ ``mixes'' to the stationary distribution $\pi$ at an exponential rate. Namely, for any $\eps>0$, one can find a constant $\tau=\tau(\eps)=O(\log \eps^{-1})$, called the ``mixing time'' of $\x_{t}$,  such that the conditional distribution of
	$\x_{t+\tau}$ given $\x_{t}$ is within total variation distance $\eps$ from $\pi$ regardless of the distribution of $\x_{t}$ (see \eqref{eq:def_TV} for the definition of total variation distance).  This mixing property of Markov chains is crucial both for practical applications of MCMC sampling as well as our theoretical analysis. For instance, a common practice of using MCMC sampling to obtain approximate i.i.d. samples is to first obtain a long Markov chain trajectory $(\x_{t})_{t\ge 1}$ and then thinning it to the subsequence $(\x_{k\tau})_{k\ge 1}$ \cite[Sec. 1.11]{brooks2011handbook}. Due to the choice of mixing time $\tau$, this forms an $\eps$-approximate i.i.d. samples from $\pi$.

	\subsection{Markov chain Monte Carlo Sampling}
	\label{subsection:MCMC}
	
	Suppose we have a finite sample space $\Omega$ and probability distribution $\pi$ on it. We would like to sample a random element $\omega\in \Omega$ according to the distribution $\pi$. \textit{Markov chain Monte Carlo (MCMC)} is a sampling algorithm that leverages the properties of Markov chains we mentioned in Subsection \ref{subsection:MC}. Namely, suppose that we have found a Markov chain $(X_{t})_{t\ge 0}$ on state space $\Omega$ that is irreducible, aperiodic\footnote{Aperiodicity can be easily obtained by making a given Markov chain lazy, that is, adding a small probability $\eps$ of staying at the current state. Note that this is the same as replacing the transition matrix $P$ by $P_{\eps}:=(1-\eps)P+\eps I$ for some $\eps>0$. This `lazyfication' does not change stationary distributions, as $\pi P =\pi$ implies $\pi P_{\eps}=\pi$. }, and has $\pi$ as its unique stationary distribution. Denote its transition matrix as $P$. Then by \eqref{eq:finite_MC_convergence_thm}, for any $\eps>0$, one can find a constant $\tau=\tau(\eps)=O(\log \eps^{-1})$ such that the conditional distribution of
	$X_{t+\tau}$ given $X_{t}$ is within total variation distance $\eps$ from $\pi$ regardless of the distribution of $X_{t}$. Recall such $\tau=\tau(\eps)$ is called the \textit{mixing time} of the Markov chain  $(X_{t})_{t\ge 1}$. Then if one samples a long Markov chain trajectory $(X_{t})_{t\ge 1}$, the subsequence $(X_{k\tau})_{k\ge 1}$ gives approximate i.i.d. samples from $\pi$. 
	
	We can further compute how far the thinned sequence $(X_{k\tau })_{k\ge 1}$ is away from being independent. Namely, observe that  for any two nonempty subsets $A,B\subseteq \Omega$, 
	\begin{align}
		&\left|   \P(X_{k\tau}\in A,\, X_{\tau}\in B) - \P(X_{k\tau}\in A) \P(X_{\tau}\in B) \right| \\
		&\qquad =   \left|  \P(X_{k\tau}\in A) - \P(X_{k\tau}\in A \,|\,  X_{\tau}\in B)  \right| \, \left| \P(X_{\tau}\in B) \right| \\
		&\qquad \le \left| \P(X_{k\tau}\in A) - \P(X_{k\tau}\in A\,|\, X_{\tau}\in B) \right|  \\
		&\qquad \le \left| \P(X_{k\tau}\in A) - \pi(A) \right| + \left| \pi(A) -  \P(X_{k\tau}\in A\,|\, X_{\tau}\in B)  \right| \le \lambda^{k\tau} + \lambda^{(k-1)\tau}.
	\end{align}
	Hence the correlation between $X_{k\tau}$ and $X_{\tau}$ is $O(\lambda^{(k-1)\tau})$.
	
	For the lower bound, let us assume that $X_{t}$ is \textit{reversible} with respect to $\pi$, that is, $\pi(x)P(x,y)=\pi(y)P(y,x)$ for $x,y\in \Omega$ (e.g., random walk on graphs). Then $\tau(\eps) = \Theta(\log \eps^{-1})$ (see \cite[Thm. 12.5]{levin2017markov}), which yields $\sup_{x\in \Omega}\lVert P^{t}(x,\cdot) - \pi  \rVert_{TV} = \Theta(\lambda^{t})$. Also, $\P(X_{\tau}\in B)>\delta>0$ for some $\delta>0$ whenever $\tau$ is large enough under the hypothesis. Hence 
	\begin{align}
		&\left|   \P(X_{k\tau}\in A,\, X_{\tau}\in B) - \P(X_{k\tau}\in A) \P(X_{\tau}\in B) \right| \\
		&\qquad \ge \delta^{-1}\left| \P(X_{k\tau}\in A) - \P(X_{k\tau}\in A\,|\, X_{\tau}\in B) \right|  \\
		&\qquad \ge \big| \left| \P(X_{k\tau}\in A) - \pi(A) \right| - \left| \pi(A) -  \P(X_{k\tau}\in A\,|\, X_{\tau}\in B)  \right| \big| \ge c\lambda^{(k-1)\tau}
	\end{align}
	for some constant $c>0$. Hence the correlation between $X_{k\tau}$ and $X_{\tau}$ is $\Theta(\lambda^{(k-1)\tau})$. In particular, the correlation between two consecutive terms in $(X_{k\tau})_{k\ge 1}$ is of $\Theta(\lambda^{\tau})=\Theta(\eps)$. Thus, we can make the thinned sequence $(X_{k\tau})_{k\ge 1}$ arbitrarily close to being i.i.d. for $\pi$, but if $X_{t}$ is reversible with respect to $\pi$, the correlation within the thinned sequence is always nonzero. 
	
	In practice, one may not know how to estimate the mixing time $\tau=\tau(\eps)$. In order to empirically assess that the Markov chain has mixed to the stationary distribution, multiple chains are run for diverse mode exploration, and their empirical distribution is compared to the stationary distribution (a.k.a. multi-start heuristic). See \cite{brooks2011handbook} for more details on MCMC sampling.

	\section{Example of  Surrogate functions}
	\label{sec:ex_surrogates}
	
	In this section, we list some examples of block-convex surrogate functions, which include the usual convex surrogate functions in the literature (see, \cite{mairal2013optimization,mairal2013stochastic}).

	\begin{example}[Proximal surrogates for $L$-smooth functions]\label{ex:prox_L_smooth}
		Suppose $f$ is $L$-smooth, that is, $\nabla f$ is $L$-Lipscthiz continuous. Then $f$ is $L$-weakly convex, that is, $\param\mapsto f(\param)+\frac{L}{2} \lVert \param \rVert^{2}$ is convex (see Lemma \ref{lem:weak_convexity}). Hence for each $\rho\ge L$, the following function $g$ belongs to $\surr_{L+\rho,\rho-L}(f,\param^{\star})$:
		\begin{align}
			g:\param\mapsto f(\param) + \frac{\rho}{2} \lVert \param-\param^{\star} \rVert^{2}.
		\end{align}
		Indeed, $g\ge f$, $g(\param^{\star})=f(\param^{\star})$, $\nabla g(\param^{\star})=\nabla f(\param^{\star})$, and $\nabla g$ is $(L+\rho)$-Lipschtiz. Moreover, $g$ is convex being some of two convex functions: 
		\begin{align}
			f(\param) + \frac{L}{2}\lVert \param-\param^{\star} \rVert^{2} = \left( f(\param) + \frac{L}{2}\lVert \param\rVert^{2} \right) + \left( -L\langle \param,\param^{\star} \rangle  + \frac{L}{2}\lVert\param^{\star} \rVert^{2} \right).
		\end{align}
		Minimizing the above function over $\Param$ is equivalent to applying a proximal mapping of $f$, where the resulting estimate is denoted as  $\textup{prox}_{f/\rho}(\param^{\star})$ (see \cite{parikh2014proximal,davis2019stochastic}).\hfill $\blacktriangle$
	\end{example}
	
	\begin{example}[Proximal modification of weakly convex surrogates]
		\label{ex:prox_modification_g}
		One can also use proximal mapping to convert block weakly convex surrogates to multi-convex surrogates. Namely,  suppose $g\in \surr_{L,-\rho}^{\mathbb{J}}(f,\param^{\star})$ is a surrogate of $f$ at $\param^{\star}$ that is $\rho$-weakly convex in each block $J\in \mathbb{J}$. That is, $\param\mapsto g(\param)+\frac{\rho}{2}\lVert \param\rVert$ is convex in $\Param^{J}$ for each $J\in \mathbb{J}$. In this case, $g$ is nonconvex in each block coordinate so it cannot be directly block-minimized. In this case, we can add a quadratic term to make it a multi-convex surrogate. That is, for each $\rho'\ge \rho$, the following function $\tilde{g}$ belongs to $ \surr_{L+\rho',\rho'-\rho}^{\mathbb{J}}(f,\param^{\star})$: 
		\begin{align}
			\tilde{g}:\param\mapsto g(\param) + \frac{\rho'}{2} \lVert \param-\param^{\star} \rVert^{2},
		\end{align}
		which can be easily verified by expanding out the quadratic term. In the slightly more general case when the modulus of weak convexity of $g$ depends on the coordinate block $J$, we may add block-dependent quadratic terms to reduce the amount of proximal modification. \hfill $\blacktriangle$
		
		%$\nabla \tilde{g}(\param) = \nabla g (\param) + \rho'(\param-\param^{\star})$ 
		
	\end{example}

	\begin{example}[Prox-linear surrogates]\label{ex:prox_linear}
		\normalfont
		If  $f$ is $L$-smooth, then the following quadratic function $g$ belongs to $\surr_{2L,L}(f,\param^{\star})$: 
		\begin{align}
			g:\param \mapsto f(\param^{\star}) + \left\langle \nabla f (\param^{\star}),\, \param-\param^{\star} \right\rangle+ \frac{L}{2} \lVert \param-\param^{\star} \rVert^{2}.
		\end{align}
		(See Lemma \ref{lem:surrogate_L_gradient}  in Appendix \ref{sec:auxiliary_lemmas}.) 
		If in addition $f$ is convex, then $g\in \surr_{L,L}(f,\param^{\star})$; If $f$ is $\mu$-strongly convex, then $g\in \surr_{L-\mu,L}(f,\param^{\star})$.  \hfill $\blacktriangle$
		
		\iffalse
		The first-order optimality condition for $g$ is 
		\begin{align}
			0 \in \nabla f(\param^{\star})  + L (\param - \param^{\star})
		\end{align}
		so minimizing $g$ amounts to the following single gradient descent step 
		\begin{align}
			\theta \leftarrow \param^{\star} - \frac{1}{L}\nabla f(\param^{\star}).
		\end{align}
		The averaged surrogate in this case satisfies the following recursion 
		\begin{align}
			\bar{g}_{n}(\param) &= (1-w_{n}) \bar{g}_{n-1}(\param) +   w_{n}\left[ f(\param_{n-1}) + \nabla f (\param_{n-1}) ^{T} (\param-\param_{n-1}) + \frac{L}{2} \lVert \param-\param_{n-1} \rVert^{2} \right] \\
			&= (1-w_{n}) \bar{g}_{n-1}(\param)  + w_{n} \left[ \langle \nabla f (\param_{n-1}),\,  \param \rangle + \frac{L}{2} \langle \param, \param \rangle + L \langle \param, \param_{n-1} \rangle \right] + Const.\\
			&=\langle \overline{\nabla_{n-1} f} ,\, \param    \rangle  + \frac{L}{2} \langle \param,\, \param \rangle - L\langle \param,\, \overline{\param}_{n-1} \rangle + Const. 
		\end{align}
		First-order optimality: 
		\begin{align}
			0 \in \overline{\nabla_{n-1} f} + L(\param - \overline{\param}_{n-1}) \quad \Longleftrightarrow \quad \param_{n}\leftarrow  \overline{\param}_{n-1} - \frac{1}{L}  \overline{\nabla_{n-1} f} 
		\end{align}
		\fi
	\end{example}

	\begin{example}[Prox-linear surrogates]
		Suppose $f=f_{1}+f_{2}$ where $f_{1}$ is differentiable with $L$-Lipschitz gradient and $f_{2}$ is convex over $\Param$. Then the following function $g$ belongs to $\surr_{2L,L}(f,\param^{\star})$: 
		\begin{align}
			g:\param \mapsto f_{1}(\param^{\star}) + \left\langle \nabla f_{1} (\param^{\star}),\, \param-\param^{\star} \right\rangle + \frac{L}{2} \lVert \param-\param^{\star} \rVert^{2} + f_{2}(\param).
		\end{align}
		Minimizing $g$ over $\Param$ amounts to performing a proximal gradient step in \cite{beck2009fast, nesterov2013gradient}.  \hfill $\blacktriangle$
	\end{example}

	\begin{example}[DC programming surrogates]
		Suppose $f=f_{1}+f_{2}$ where $f_{1}$ is convex and $f_{2}$ is concave and differentiable with $L_{2}$-Lipschitz gradient over $\Param$. One can also write $f=f_{1}-(-f_{2})$ which is the difference of convex (DC) functions $f_{1}$ and $-f_{2}$. Then the following function $g$ belongs to $\surr_{2L,L}(f,\param^{\star})$: 
		\begin{align}
			g:\param \mapsto f_{1}(\param)+ f_{2}(\param^{\star}) + \left\langle \nabla f_{2} (\param^{\star}),\, \param-\param^{\star} \right\rangle.
		\end{align}
		Such surrogates are important in DC programming (see, e.g., \cite{horst1999dc}). 
		
		In fact, our method allows $f_{1}$ to be only multi-convex with respect to the coordinate blocks in $\mathbb{J}$, that is, $f_{1}$ is convex on $\textup{Proj}_{\R^{J}}(\Param)$ for each coordinate block $J\in \mathbb{J}$. In this case, $g$ above is a multi-convex surrogate of $f$ at $\param^{\star}$ and belongs to $\surr_{2L,L}^{\mathbb{J}}(f,\param^{\star})$.  \hfill $\blacktriangle$
	\end{example}
	
	\begin{example}[Convex Variational Surrogates]\label{ex:cvx_variational_surrogate}
		\normalfont
		Let $f:\R^{p}\times \R^{q}\rightarrow \R$ be a two-block multi-convex function and let $\Theta_{1}\subseteq \R^{p}$ and  $\Theta_{2}\subseteq \R^{q}$ be two convex sets. Define a function $f_{*}:\param\mapsto \inf_{H \in \Theta_{2}} f(\param,H)$. Then for each $\param^{\star}\in \Theta$, the following function
		\begin{align}
			g:\param\mapsto f(\param, H^{\star}),\quad H^{\star}\in \argmin_{H\in \Param_{2}} f(\param^{\star},H)  
		\end{align} 	
		is convex over $\Theta_{1}$ and	satisfies $g\ge f$ and $g(\param^{\star})=f(\param^{\star})$. Further, assume that 
		\begin{description}
			\item[(i)] $\param\mapsto f(\param,H)$ is differentiable for all $H\in \Theta_{2}$  and $\param\mapsto \nabla_{\param} f(\param,H)$ is $L'$-Lipschitz for all $H\in \Theta_{2}$ ;
			\item[(ii)] $H\mapsto \nabla_{\param} f(\param,H)$ is $L$-Lipschitz for all $\param\in \Theta_{1}$;   
			\item[(iii)] $H\mapsto f(\param,H)$ is $\rho$-strongly convex for all $\param\in \R^{p}$. 
		\end{description}
		Then $g$ belongs to $\surr_{L'',\rho}(f_{*},\param^{\star})$ for some $L''>0$. When $f$ is jointly convex, then $f_{*}$ is also convex and we can choose $L''=L$. 
		\hfill $\blacktriangle$
	\end{example}
	
	\begin{example}[Multi-convex Variational Surrogates]\label{ex:multiconvex_variational}
		Let $f:\R^{p_{1}}\times \dots \times \R^{p_{m}}\times \R^{p_{m+1}}\rightarrow \R$ be a $(m+1)$-block multi-convex function and let $\Theta_{1}\subseteq \R^{p}$ and  $\Theta_{i}\subseteq \R^{p_{i}}$ for $i=1,\dots,m+1$ be  convex sets. Denote $\Param:=\Theta_{1}\times \dots \times \Theta_{m}$. Define a function $f_{*}: \param  \mapsto \inf_{H \in \Theta_{p_{m+1}}} f(\param, H)$. Then for each $\param^{\star}\in \Param$, the following function
		\begin{align}
			g:\param\mapsto f(\param, H^{\star}),\quad H^{\star}\in \argmin_{H\in \Param_{p_{m+1}}} f(\param^{\star},H)  
		\end{align} 	
		is $m$-block multi-convex over $\Param$ and 
		satisfies $g\ge f$ and $g(\param^{\star})=f(\param^{\star})$. Further, assume that 
		\begin{description}
			\item[(i)] $\param\mapsto f(\param,H)$ is differentiable for all $H\in \Theta_{p_{m+1}}$  and $\param\mapsto \nabla_{\param} f(\param,H)$ is $L'$-Lipschitz for all $H\in \Theta_{p_{m+1}}$ ;
			\item[(ii)] $H\mapsto \nabla_{\param} f(\param,H)$ is $L$-Lipschitz for all $\param\in \Param$;   
			\item[(iii)] $H\mapsto f(\param,H)$ is $\mu$-strongly convex for all $\param\in \R^{p_{m+1}}$. 
		\end{description}
		Let $\mathbb{J} = \{ J_{1},\dots,J_{m} \}$, where $J_{i}$ denotes the $i$th coordinate block corresponding to $\R^{p_{i}}$. Then $g$ belongs to $\surr_{L'',\mu}^{\mathbb{J}}(f_{*},\param^{\star})$ for some $L''>0$. When $f$ is jointly convex, then $f_{*}$ is also convex and we can choose $L''=L$.  \hfill $\blacktriangle$
	\end{example}

	\begin{example}[Refining block structure]\label{ex:refine_block}
		Suppose $f:\Param\subseteq \R^{p}\rightarrow \R$ with $\Param$ convex, and we have a convex surrogate $g\in \surr_{L,\rho}(f,\param^{\star})$. Let $\mathbb{J}$ denote an arbitrary set of coordinate bloks for $\R^{p}$ (e.g., $\mathbb{J}=\{ \{1\},\dots,\{p\} \}$). Then one can also view the fuction $g$ as a multi-convex surrogate in $\surr_{L,\rho}^{\mathbb{J}}(f_{*},\param^{\star})$. Instead of minimizing $g$ over the convex set $\Param$ with respect to all $p$ coordinates,  one can block-minimize $g$ with respect to the block structure $\mathbb{J}$. For instance, if $\mathbb{J}=\{ \{1\},\dots,\{p\} \}$ and if $J_{i}(n)$ in Algorithm \ref{algorithm:BSM-DR} is chosen uniformly at random from $\mathbb{J}$, then Algorithm \ref{algorithm:BSM-DR} becomes a random coordinate descent for $m$ iterations, which choses a random coordinate $i\in \{1,\dots,p\}$ and minimize $g$ over $\Param$ with respect to the $i$th coordinate. (See, e.g., \cite{wright2015coordinate}).
		
		A similar remark holds for multi-convex surrogates. Namely, suppose $g\in \surr_{L,\rho}^{\mathbb{J}}(f,\param^{\star})$, where $\mathbb{J}=\{J_{1},\dots,J_{m}\}$ denotes a set of coordinate blocks for $\R^{p}$. Let $\mathbb{J}'$ denote an arbitrary set of coordinate bloks for $\R^{p}$ that refines $\mathbb{J}$. Namely, each coordinate block $J\in \mathbb{J}'$ is a subset of some $J_{i}$, $i=1,\dots,m$. Then one can also view the fuction $g$ as a multi-convex surrogate in $\surr_{L'',\rho}^{\mathbb{J}'}(f_{*},\param^{\star})$ and accordingly, one can block-minimize $g$ using the finer block structure given by $\mathbb{J}'$ instead of the original block structure  $\mathbb{J}$.  \hfill $\blacktriangle$
	\end{example}
	
	See \cite{mairal2013optimization} for other types of surrogate functions such as quadratic, Jensen, and saddle point surrogates.

	\section{Auxiliary lemmas}
	\label{sec:auxiliary_lemmas}
	
	\begin{lemma}[Convex Surrogate for Functions with Lipschitz Gradient]
		\label{lem:surrogate_L_gradient}
		Let $f:\R^{p}\rightarrow \R$ be differentiable and $\nabla f$ be $L$-Lipschitz continuous. Then for each $\theta,\theta'\in \R^{p}$, 
		\begin{align}
			\left| f(\theta') - f(\theta) - \langle \nabla f(\theta),\, \theta'-\theta\rangle  \right|\le \frac{L}{2} \lVert \theta-\theta'\rVert^{2}.
		\end{align}
	\end{lemma}
	
	\begin{proof}
		This is a classical Lemma. See \cite[Lem 1.2.3]{nesterov1998introductory}.
	\end{proof}

	\begin{lemma}[Characterization of weak convexity]\label{lem:weak_convexity}
		Let $f:\R^{p}\rightarrow \R$ be a smooth function. Fix a convex set $\Param\subseteq \R^{p}$ and $\rho>0$. The following conditions are equivalent. 
		\begin{description}[itemsep=0.1cm]
			\item[(i)] (Weak convexity) $\param\mapsto f(\param) + \frac{\rho}{2}\lVert \param \rVert^{2}$ is convex on $\Param$;
			\item[(ii)] (Hypermonotonicity) $ \langle \nabla f(\param) - \nabla f(\param'),\, \param-\param'  \rangle \ge - \rho\lVert \param-\param' \rVert^{2}$ for all $\param,\param'\in \Param$; 
			
			\item[(iii)] (Quadratic lower bound) $f(\param) - f(\param') \ge \langle \nabla f(\param'),\, \param-\param' \rangle - \frac{\rho}{2}\lVert \param-\param' \rVert^{2}$ for all $\param,\param'\in \Param$. 
		\end{description}
	\end{lemma}
	
	\begin{proof}
		See \cite[Thm. 7]{daniilidis2005filling} for an equivalent statement for a more general case of local Lipschitz functions. 
		\begin{description}[itemsep=0.1cm]
			\item{\textbf{(i)}$\Rightarrow$\textbf{(iii)}:} For $s\in [0,1]$, we have 
			\begin{align}\label{eq:convexity_lem}
				s \left( f(\param)  + \frac{\rho}{2} \lVert \param \rVert^{2} \right)+ (1-s) \left( f(\param')   + \frac{\rho}{2} \lVert \param' \rVert^{2} \right) \ge 	f(s\param + (1-s)\param') + \frac{\rho}{2} \lVert s\param + (1-s)\param' \rVert^{2}
			\end{align}
			so we get 
			\begin{align}
				f(\param) - f(\param') 	\ge	\frac{f(\param' + s(\param-\param') ) - f(\param')}{s}  	-   \frac{\rho}{2} \left( \lVert \param \rVert^{2} - \lVert \param' \rVert^{2} -  \frac{\lVert \param' + s(\param-\param') \rVert^{2} -	  \lVert \param' \rVert^{2}}{s}  \right) 
			\end{align}
			Taking the limit $s\rightarrow 0$, we get 
			\begin{align}
				f(\param) - f(\param')  \ge 	\langle \nabla f(\param'),\, \param-\param' \rangle 	 - \frac{\rho}{2} \left( \lVert \param \rVert^{2} -  \lVert \param' \rVert^{2} - 2\langle \param,\param-\param' \rangle \right). 
			\end{align}
			This implies \textbf{(iii)}. 
			
			\item{\textbf{(iii)}$\Rightarrow$\textbf{(ii)}:} Adding the following two inequalities
			\begin{align}
				f(\param) - f(\param') &\ge \langle \nabla f(\param'),\, \param-\param' \rangle - \frac{\rho}{2}\lVert \param-\param' \rVert^{2} \\
				f(\param') - f(\param) &\ge \langle \nabla f(\param),\, \param'-\param \rangle - \frac{\rho}{2}\lVert \param-\param' \rVert^{2}, 
			\end{align}
			we obtain $0 \ge \langle \nabla f(\param') - \nabla f(\param),\, \param-\param'  \rangle - \rho\lVert \param-\param' \rVert^{2}$. This implies \textbf{(ii)}. 
			
			\item{\textbf{(ii)}$\Rightarrow$\textbf{(i)}:} Fix $s\in [0,1]$ and $\param,\param'\in \Param$. Denote $\param_{s}:=s\param + (1-s)\param'$. Note that $\param_{s}-\param'=s(\param-\param')$ and $\param_{s}-\param=(1-s)(\param'-\param)$. By the mean value theorem, there exists $s_{*}\in [s,1]$ and $s'_{*}\in [0,s]$ such that 
			\begin{align}
				f(\param_{s}) - f(\param) = (1-s)\langle \nabla f(\param_{s_{*}}),\, \param'-\param  \rangle, \qquad 	f(\param_{s}) - f(\param') = s\langle \nabla f(\param_{s_{*}'}),\, \param-\param'  \rangle. 
			\end{align}
			Multiplying by $s$ and $1-s$ respectively and adding the resulting equations, we get 
			\begin{align}
				f(\param_{s}) - sf(\param) - (1-s)f(\param') &= - s(1-s) \langle \nabla f(\param_{s_{*}}) - \nabla f(\param_{s'_{*}}),\, \param-\param' \rangle.
			\end{align}
			Note that since $s_{*}' \le s \le s_{*}$ and $\param_{s_{*}},\param_{s_{*}'}$ are in the secant line between $\param$ and $\param'$, by \textbf{(ii)}, 
			\begin{align}
				\left\langle \nabla f(\param_{s_{*}}) - \nabla f(\param_{s'_{*}}),\, \param-\param'  \right\rangle  &= \frac{  \lVert \param - \param' \rVert }{ \lVert \param_{s^{*}} - \param_{s'_{*}} \rVert} 	\left\langle \nabla f(\param_{s_{*}}) - \nabla f(\param_{s'_{*}}),\, \param_{s_{*}}-\param_{s'_{*}} \right\rangle   \\
				& \ge 	-\rho \lVert  \param_{s^{*}}-\param_{s'_{*}} \rVert	\cdot  \lVert \param - \param' \rVert \\
				&\ge 	-\rho \lVert \param - \param' \rVert^{2}.
			\end{align}
			It follows that 
			\begin{align}
				f(\param_{s}) - sf(\param) - (1-s)f(\param') &\le  s(1-s) \rho \lVert \param-\param' \rVert^{2},
			\end{align}
			which is equivalent to \eqref{eq:convexity_lem}. 
		\end{description}
	\end{proof}

	\iffalse
	\begin{lemma}[Surrogate Functions and Subdifferential]
		\label{lem:convex_subdifferential}
		Suppose that $f,g:\R^{p}\rightarrow \R$ are convex, and that $h:=g-f$ is differentiable at $\theta\in \R^{p}$ with $\nabla h(\theta)=0$. Then $\partial f (\theta)= \partial g(\theta)$. 
	\end{lemma}
	
	\begin{proof}
		See \cite[Lem. A.2]{mairal2013stochastic}.
	\end{proof}

	\begin{lemma}[Lower Bound on Strongly Convex functions]\label{lem:strongly_convex_lower_bd}
		Let $f:\R^{p}\rightarrow \R$ be a $\mu$-strongly convex function and fix $\kappa\in \R^{p}$. Then for all $z\in \partial f(\kappa)$ and $\theta \in \R^{p}$, 
		\begin{align}
			f(\theta) \ge f(\kappa) + z^{T} (\theta-\kappa) +  \frac{\mu}{2} \lVert \theta-\kappa \rVert^{2}.
		\end{align}
	\end{lemma}
	
	\begin{proof}
		See \cite[Lem. A.3]{mairal2013stochastic}.
	\end{proof}
	\fi

	\begin{lemma}\label{lem:L_smooth_weak_convex}
		Let $f:\R^{p}\rightarrow \R$ be a function such that $\nabla f$ is $L$-Lipscthiz for some $L>0$. Then $f$ is $L$-weakly convex, that is, $\param\mapsto f(\param)+\frac{L}{2}\lVert \param\rVert^{2}$  is convex. 
	\end{lemma} 
	
	\begin{proof}
		Follows immediately by Lemmas \ref{lem:surrogate_L_gradient} and  \ref{lem:weak_convexity}.
	\end{proof}

	\begin{lemma}[Second-Order Growth Property]
		\label{lem:second_order_growh_univariate}
		Let $f:\R^{p} \rightarrow [0,\infty)$ be  $\mu$-strongly convex and let $\Theta$ is a convex subset of $\R^{p}$. Let $\theta^{*}$ denote the minimizer of $f$ over $\Theta$. Then for all $\theta\in \Theta$,
		\begin{align}
			f(\theta) \ge  f(\theta^{*}) + \frac{\mu}{2} \lVert \theta-\theta^{*} \rVert^{2}.	
		\end{align}
	\end{lemma}
	
	\begin{proof}
		See \cite[Lem. A.4]{mairal2013stochastic}.
	\end{proof}

	Next, we provide some probabilistic lemmas.

	\begin{lemma}\label{lem:martingale_convergence}
		Let $(X_{n})_{n\ge 1}$ be a sequence of nonnegative random variables adapted to a filtration $(\mathcal{F}_{n})_{n\ge 1}$ and $\E[X_{n}]<\infty$ for $n\ge 1$.  Suppose there exists constants $\alpha\in [0,1)$ and $M>0$ such that for sufficiently large $n\ge 1$,   
		\begin{align}\label{eq:martingale_ineq}
			\E\left[ X_{n}\,|\, \mathcal{F}_{n-1} \right] \le \alpha X_{n-1} + M.
		\end{align}
		Then $\lim_{n\rightarrow \infty} X_{n}$ exists and finite almost surely. In particular, $(X_{n})_{n\ge 1}$  is bounded almost surely. 
	\end{lemma}
	
	\begin{proof}
		Fix a constant $c>0$ and let $\widetilde{X}_{n}:=X_{n}+c$ for $n\ge 1$. Substituting for $X_{n}$ in \eqref{eq:martingale_ineq}, we get 
		\begin{align}
			\E[\tilde{X}_{n}\,|\, \mathcal{F}_{n-1}] \le \alpha \tilde{X}_{n-1}  + (1-\alpha)c+ M
		\end{align}
		for large enough $n\ge 1$. 	Thus if we choose $c=-M/(1-\alpha)$, then $(\widetilde{X}_{n})_{n\ge N}$ is a supermartingale with respect to the filtration $(\mathcal{F}_{n})_{n\ge N}$ for some $N\ge 1$. Furthermore, $\widetilde{X}_{n}$ is uniformly lower bounded by $c=-M/(1-\alpha)$ since $X_{n}\ge 0$ for $n\ge 1$. By the martingale convergence theorem (see, e.g., \cite[Thm. 4.2.11]{durrett2019probability}), $\tilde{X}_{n}$ converges almost surely to some random variable, say, $\widetilde{X}_{\infty}$, such that $\E[\widetilde{X}_{\infty}] \le \E[\widetilde{X}_{0}]<\infty$. In particular, $\widetilde{X}_{\infty}<\infty$ almost surely. It follows that $X_{n}$ converges to $\widetilde{X}_{\infty}-c$ almost surely. Lastly, note that 
		\begin{align}
			\P( \textup{$(X_{n})_{n\ge 1}$ is not bounded } ) \le \P(\textup{$X_{n}$ does not converge to $\widetilde{X}_{\infty}-c$})=0.
		\end{align}
		This shows the assertion. 
	\end{proof}

	\begin{lemma}\label{lem:strongly_convex_surrogate_A6'}
		Let $(\param_{n})_{n\ge 1}$ denote the output of Algorithm \ref{algorithm:SMM} with Algorithm \ref{algorithm:BSM-PR} used for \eqref{eq:alg1_param_update} with $\lambda_{n}\equiv 0$. Assume  \ref{assumption:A1-ell_smooth}, \ref{assumption:A3-cvx_constraint}, and  \ref{assumption:A5'_sufficient_surrogate_decay} hold. Suppose that the data sequence $(\x_{n})_{n\ge 1}$ is contained in some compact subset $\mathfrak{X}_{0}\subseteq \mathfrak{X}$. Then  $\lVert \param_{n-1}-\param_{n} \rVert=O(w_{n})$ and  $\Delta_{n}=O(w_{n}^{2})$ for $n\ge 1$. In particular, if $\sum_{n=1}^{\infty}w_{n}^{2}<\infty$, then \ref{assumption:A5_sufficient_surrogate_decay} holds. 
	\end{lemma}
	
	\begin{proof}
		The argument here follows the one given in \cite{mensch2017stochastic}. Let $\param_{n}^{\star}$ be the exact minmizer of the $\rho$-strongly convex function $\bar{g}_{n}$ over $\Param$ for $n\ge 1$.  We introduce the following random variables 
		\begin{align}
			A_{n} = \lVert \param_{n}-\param_{n-1} \rVert,\quad B_{n}= \lVert\param_{n} - \param_{n}^{\star} \rVert, \quad C_{n}= \lVert\param_{n-1}^{\star} - \param_{n}^{\star} \rVert,\quad D_{n}=\bar{g}_{n}(\param_{n}) - \bar{g}_{n}(\param_{n}^{\star}). 
		\end{align}
		By $\rho$-strong convexity of $\bar{g}_{n}$, almost surely, $\frac{\rho}{2} B_{n}^{2} \le D_{n}$.  By the second-order growth property (Lemma \ref{lem:second_order_growh_univariate}) and using $L'$-Lipschitz continuity of $g_{n}$ (see Lemma \ref{prop:g_grad_Lipschitz}), almost surely,
		\begin{align}
			\frac{\rho}{2} C_{n}^{2}&\le \bar{g}_{n}(\param_{n-1}^{\star} ) - \bar{g}_{n}(\param_{n}^{\star} ) \\
			&= (1-w_{n} ) \left( \bar{g}_{n-1}(\param_{n-1}^{\star} ) - \bar{g}_{n-1}(\param_{n}^{\star} ) \right) + w_{n} \left( g_{n}(\param_{n-1}^{\star} ) - g_{n}(\param_{n}^{\star} ) \right) \\
			&\le w_{n} \left( g_{n}(\param_{n-1}^{\star} ) - g_{n}(\param_{n}^{\star} ) \right) \le w_{n} L' C_{n}.
		\end{align} 
		This implise $C_{n} \le c_{1}w_{n}$ for all $n\ge 1$ almost surely for some constant $c_{1}>0$. By triangle inequality, we can write $A_{n} \le  B_{n} + B_{n-1} + C_{n} $, so by Cauchy-Schwarz inequality, 
		\begin{align}\label{eq:ineq_A_D}
			A_{n}^{2} \le (B_{n} + B_{n-1} + C_{n} )^{2}  &\le 	3(B_{n}^{2} + B_{n-1}^{2} + C_{n}^{2} ) \\
			&\le  3( (2/\rho) (D_{n} + D_{n-1} )  + (2L'/\rho)^{2}w_{n}^{2} ). 
		\end{align}
		Next, we assume \ref{assumption:A5'_sufficient_surrogate_decay} and show that $D_{n}=O(w^{2})$.  Denote $\widetilde{D}_{n}:=D_{n}/w_{n}^{2}$. We will show that 
		\begin{align}\label{eq:pf_D_n_martingale_ineq}
			\E\left[ \widetilde{D}_{n} \,|\, \mathcal{F}_{n-1}\right] 
			&\le  \alpha \widetilde{D}_{n-1} + M
		\end{align}	
		for all sufficiently large $n\ge 1$, where $\alpha\in [0,1)$ and $M>0$ are constants. Then by Lemma \ref{lem:martingale_convergence}, it follows that $\widetilde{D}_{n}$ is uniformly bounded almost surely. This yields $D_{n}=O(w_{n}^{2})$. The assertion then follows by combining this with \eqref{eq:ineq_A_D}. 
		
		It suffices to verify \eqref{eq:pf_D_n_martingale_ineq}. First, by using \ref{assumption:A5'_sufficient_surrogate_decay} and the previous inequalities involving $A_{n}$  and $B_{n}$, we have 
		\begin{align}
			\E\left[ D_{n} \,|\, \mathcal{F}_{n-1}\right] &\le (1-\mu) \left(\bar{g}_{n}(\param_{n-1}) - \bar{g}_{n}(\param_{n}^{\star}) \right) \\
			&= (1-\mu)  w_{n} \left(g_{n}(\param_{n-1}) - g_{n}(\param_{n})  + g_{n}(\param_{n})-g_{n}(\param_{n}^{\star}) \right)  \\
			&\hspace{1cm} + (1-\mu)(1-w_{n}) \left(\bar{g}_{n-1}(\param_{n-1}) - \bar{g}_{n-1}(\param_{n-1}^{\star}) \right) + (1-\mu)(1-w_{n}) \left(\bar{g}_{n-1}(\param_{n-1}^{\star}) - \bar{g}_{n-1}(\param_{n}^{\star}) \right)  \\ 
			&\le  (1-\mu) L' w_{n} (A_{n} + B_{n} )  + (1-\mu) D_{n-1} \\ 
			&\le (1-\mu) c_{1} w_{n} \left( \sqrt{D_{n}+ D_{n-1} + c_{2} w_{n}^{2}} + \sqrt{D}_{n}    \right)  + (1-\mu) D_{n-1},
		\end{align}
		where $c_{1},c_{2}>0$ are constants. The second inequality above uses the fact that $\param_{n-1}^{\star}$ is an exact minimizer of $\bar{g}_{n-1}$ over $\Param$. Setting $\widetilde{D}_{n}:=D_{n}/w_{n}^{2}$, this gives 
		\begin{align}
			\E\left[ \widetilde{D}_{n} \,|\, \mathcal{F}_{n-1}\right] 
			&\le (1-\mu) c_{1} \left( \sqrt{\widetilde{D}_{n}+ \widetilde{D}_{n-1} (w_{n-1}/w_{n})^{2} + c_{2}} + \sqrt{\widetilde{D}_{n}} \right) + (1-\mu) \widetilde{D}_{n-1} (w_{n-1}/w_{n})^{2}.
		\end{align}
		Fix $\eps>0$. Using the inequality $\sqrt{t} \le \eps t + \eps^{-1}$ for all $t\ge 0$, this gives 
		\begin{align}
			\E\left[ \widetilde{D}_{n} \,|\, \mathcal{F}_{n-1}\right] 
			%&\le (1-\mu) c_{1} \left[ \eps \left( \widetilde{D}_{n}+ \widetilde{D}_{n-1} (w_{n-1}/w_{n})^{2} + c_{2}\right) + \eps^{-1} + \eps \widetilde{D}_{n} + \eps^{-1} \right] + (1-\mu) \widetilde{D}_{n-1} (w_{n-1}/w_{n})^{2} \\
			&\le  c_{3}\eps \widetilde{D}_{n}  +(1-\mu)(1+c_{1}\eps) (w_{n-1}/w_{n})^{2} \widetilde{D}_{n-1} + c_{4}\eps^{-1},
		\end{align}
		where $c_{3},c_{4}>0$ are constants. Recall  the hypothesis that $w_{n-1}/w_{n}\rightarrow 1$ as $n\rightarrow \infty$. Hence taking the expectation conditional on $\mathcal{F}_{n-1}$ and choosing $\eps>0$ sufficiently small, we verify \eqref{eq:pf_D_n_martingale_ineq}, as desired. 
	\end{proof}

	\begin{lemma}\label{lem:uniform_convergence_symmetric_weights}
		Fix a bounded and measurable function $\psi:\mathfrak{X}\times \Param \rightarrow \R$. Under Assumptions \ref{assumption:A2-MC} and \ref{assumption:A3-cvx_constraint}, 
		\begin{align}
			\E\left[ \sup_{\param\in \Param} \sqrt{n}\left\lvert \E_{\x\sim \pi}\left[ \psi(\x,\param)\right] - \frac{1}{n} \sum_{k=1}^{n} \psi(\x_{k},\param)\right\rvert \right] = O(1).
		\end{align}
		Furthermore, $\sup_{\param\in \Param} \left\lvert \E_{\x\sim \pi}\left[ \psi(\x,\param)\right] - \frac{1}{n} \sum_{k=1}^{n} \psi(\x_{k},\param)\right\rvert \rightarrow 0$ almost surely as $t\rightarrow \infty$. 
	\end{lemma}
	
	\begin{proof}
		Omitted. See \cite[Lem. 10]{lyu2020online}
	\end{proof}

	\begin{lemma}\label{lem:uniform_convergence_asymmetric_weights}
		Fix a continuous and measurable function $\psi:\mathfrak{X}\times \Param \rightarrow \R$. Fix a sequence $(w_{n})_{n\ge 1}$ in $(0,1]$ and   define functions $\bar{\psi}_{n}:\Param\rightarrow \R$ by
		\begin{align}
			\bar{\psi}_{n}(\param) = \sum_{k=1}^{n} \hat{w}_{k}^{n} \psi(\x_{k},\param),
		\end{align}
		where $w^{n}_{k}:=w_{k}\prod_{i=k+1}^{n} (1-w_{i})$. Suppose  \ref{assumption:A2-MC}-\ref{assumption:A3-cvx_constraint} hold. Assume that there exist an integer $T\ge 1$ such that $w_{T}\ge 1/2$ and $w_{n+1}^{-1}-w_{n}^{-1}\le 1$ for all $n\ge T$.  Suppose $w_{n} \ge cn^{-\gamma}$ for some constant $c>0$ and $\gamma\in (0,1]$ for all $n\ge 1$.

		Then there exists a constant $C=C(T)>0$ such that for all $n\ge 1$,
		\begin{align}\label{eq:uniform_CLT_bd}
			\E\left[ \sup_{\param\in \Param} \left\lvert \E_{\x\sim \pi}\left[ \psi(\x,\param)\right] - \bar{\psi}_{n}(\param) \right\rvert \right] \le Cw_{n}\sqrt{n}.
		\end{align}
		Furthermore, if $w_{n}\sqrt{n}=O(1/(\log n)^{1+\eps})$ for some $\eps>0$, then $\sup_{\param\in \Param} \left\lvert \E_{\x\sim \pi}\left[ \psi(\x,\param)\right] - \bar{\psi}_{n}(\param) \right\rvert\rightarrow 0$ almost surely as $t\rightarrow \infty$. 
	\end{lemma}

	The following lemma establishes the fluctuation bound in Lemma \ref{lem:f_n_concentration_L1_gen} for univariate observables. 
	
	\begin{lemma}\label{lem:uniform_convergence_asymmetric_weights}
		Fix compact subsets $\mathcal{X}\subseteq \R^{q}$, $\Param\subseteq \R^{p}$ and a bounded Borel measurable function $\psi:\mathcal{X}\times \Param\rightarrow \R$. Let $(\x_{n})_{n\ge 1}$ denote a sequence of points in $\mathcal{X}$ such that $\x_{n}=\varphi(X_{n})$ for $n\ge 1$, where $(X_{n})_{n\ge 1}$ is a Markov chain on a state space $\Omega$ and $\varphi:\Omega \rightarrow \mathcal{X}$ is a measurable function. Fix a sequence of weights $w_{n}\in (0,1]$, $n\ge 1$ and define functions $\bar{\psi}(\cdot):= \E_{\x\sim \pi}\left[ \psi(\x,\cdot) \right] $ and $\bar{\psi}_{n}:\Param\rightarrow \R$ recursively as $\bar{\psi}_{0}\equiv \mathbf{0}$ and 
		\begin{align}
			\bar{\psi}_{n}(\cdot) = (1-w_{n})\bar{\psi}_{n-1}(\cdot) + w_{n} \psi(\x_{n}, \cdot).
		\end{align}
		Assume the following: 
		\begin{description}
			\item[(a1)] The Markov chain $(X_{n})_{n\ge 1}$ mixes exponentially fast to its unique stationary distribution and the stochastic process $(\x_{n})_{n\ge 1}$ on $\mathcal{X}$ has a unique stationary distribution $\pi$. 
			\item[(a2)] $w_{n}$ is non-increasing in $n$ and   $w_{n}^{-1} - w_{n-1}^{-1}\le 1$ for all sufficiently large $n\ge 1$. 
		\end{description}
		
		\noindent Then there exists a constant $C>0$ such that for all $n\ge 1$,
		\begin{align}\label{eq:lem_f_fn_bd_gen_appendix}
			\E\left[ \sup_{\param\in \Param} \left\lVert \bar{\psi}(\param) - \bar{\psi}_{n}(\param) \right\rVert \right] \le Cw_{n} \sqrt{n}.
		\end{align}
		Furthermore, if $w_{n}\sqrt{n}=O(1/(\log n)^{1+\eps})$ for some $\eps>0$, then $\sup_{\param\in \Param} \left\lVert \bar{\psi}(\param) - \bar{\psi}_{n}(\param) \right\rVert\rightarrow 0$ as $t\rightarrow \infty$ almost surely. 
		
	\end{lemma}

	\begin{proof}
		We first argue for the special case when $w_{n}=1/n$ for $n\ge 1$. In this case, $w_{n}^{-1}-w_{n-1}^{-1}\equiv 1$ so \textbf{(a2)} is satisfied, and also $w^{n}_{k}\equiv 1/n$ for $1\le k \le n$ so $\bar{\psi}_{n}(\param)$ becomes the sample mean 
		\begin{align}
			\bar{\psi}_{n}(\param) = \frac{1}{n} \sum_{k=1}^{n} \psi(\x_{k}, \param). 
		\end{align}
		If $\x_{k}$'s are i.i.d. from $\pi$, then \eqref{eq:lem_f_fn_bd_gen_appendix} holds by a standard empirical process theory. More generally under the hypothesis \textbf{(a1)}, \eqref{eq:lem_f_fn_bd_gen_appendix} also holds by using a uniform functional central limit theorem for Markov chains. This statement has been shown in \cite[Lem. 10]{lyu2020online}.
		
		Now we consider a more general weighting scheme under \textbf{(a2)}. Fix $t\in \mathbb{N}$.   Define $\Psi_{i}(\param) = (t-i+1)^{-1}\sum_{j=1}^{t} \psi(\x_{j}, \param)$ for each $1\le i \le t$. Denote $\bar{\psi}(\param):=\E_{\x\sim \pi}[ \psi(\x,\param) ]$. By the previous case, there exists a constant $c_{1}>0$ such that 
		\begin{align}\label{eq:pf_uniform_convergnece_weighted}
			\E\left[ \sup_{\param\in \Param} |\Psi_{i}(\param) - \bar{\psi}(\param)| \right] \le \frac{c_{1}}{\sqrt{n-i+1}}
		\end{align} 
		for all $1\le i \le n$. A simple calculation shows the following  identity 
		\begin{align}
			\bar{\psi}_{n} - \bar{\psi} = \sum_{i=1}^{n} (w_{i}^{n} - w^{n}_{i-1}) (n-i+1) (\Psi_{i} - \bar{\psi}),
		\end{align} 
		with the convention of $w^{n}_{0}=0$. Also, suppose $T\ge 1$ is such that $w_{k}^{-1} - w_{k-1}^{-1} \le 1$ for $k\ge T$. Note that for $i\ge 2$, $w^{n}_{i-1} \le w^{n}_{i}$ if and only if $w_{i-1}(1-w_{i}) \le w_{i}$ if and only if $w_{i}^{-1}-w_{i-1}^{-1}\le 1$. Hence for each $n > T$ and $k\ge T$, we have $w^{n}_{k}\le w^{n}_{k+1}\le \dots \le w^{n}_{n}=w_{n}$. Then observe that 
		\begin{align}\label{eq:pf_uniform_convergnece_weighted2}
			\E\left[ \sup_{\param \in \Param} |\bar{\psi}_{n}(\param) - \bar{\psi}(\param)| \right] &\le \E\left[ \sum_{i=1}^{n} |w^{n}_{i} - w^{n}_{i-1}| (t-i+1) \sup_{W\in \Param} \left|\Psi_{i}(\param) - \bar{\psi}(\param) \right|  \right] \\
			&= \sum_{i=1}^{n} |w^{n}_{i} - w^{n}_{i-1}|  (n-i+1) \,  \E\left[ \sup_{\param\in \Param} \left|\Psi_{i}(\param) - \bar{\psi}(\param) \right|  \right] \\
			&\le  \sum_{i=1}^{n} |w^{n}_{i} - w^{n}_{i-1}|  c_{1}\sqrt{n-i+1} \\
			&\le  c_{1}\sqrt{n} \left( \sum_{i=1}^{T} |w_{i}^{n} -w_{i-1}^{n} | +  \sum_{i=T}^{n} (\hat{w}^{n}_{i} - \hat{w}^{n}_{i-1}) \right)  \\
			&\le c_{1}\sqrt{n}  \left( w_{n} +  \sum_{i=1}^{T} w_{i}^{n}  \right).
		\end{align} 
		By using Lemma \ref{lem:weight_sum_bound}, we have $\sum_{i=1}^{T} w_{i}^{n} =O(1/n)$. Hence the above shows \eqref{eq:lem_f_fn_bd_gen_appendix}. Lastly, note that \textbf{(a2)} implies $w_{n} \ge c_{2}/n$ for all $n\ge 1$ for some constant $c_{2}>0$ (see the proof of Lemma \ref{lem:weight_sum_bound}) for details. Therefore we can conclude \eqref{eq:lem_f_fn_bd_gen_appendix}.

		To show the second part of the assertion, suppose $w_{n}\sqrt{n}=O(1/(\log n)^{1+\eps})$ for some fixed $\eps>0$. Denote $X_{n}:=\sup_{\param\in \Param} | \bar{\psi}(\param) - \bar{\psi}_{n}(\param) |$, which is a nonnegative random variable. We wish to show $X_{n}\rightarrow 0$ almost surely as $n\rightarrow \infty$. Suppose for contradiction that there exists a diverging sequence $n(k)$ such that $X_{n(k)}$ does not converge to $0$ with some positive probability.  According to the first part, there exists a constant $c>0$ such that $\E[X_{n}] \le c (\log n)^{-1-\eps} $ for all $n\ge 1$. By Markov inequality, we have 
		\begin{align}
			\P\left( |X_{n}|\ge  \frac{1}{ (\log n)^{\eps/2}  }\right) \le \frac{ c }{(\log n)^{1+(\eps/2)}}.
		\end{align}
		We may choose a subsequence $m(r)=n(k(r))$ of $n(k)$ such that there exists a constant $c'>0$ for which $m(r) \ge c' \exp(r)$  for all $r\ge 1$. Then it follows that 
		\begin{align}
			\P\left(|X_{ m(r) }| \ge \frac{1}{ (\log m(r))^{\eps/2} }  \right) \le \frac{ c} { r^{1+(\eps/2)} }.
		\end{align}
		Since $m(r)$ diverves,  by the Borel-Cantelli lemma, $X_{m(r)}\rightarrow 0$ almost surely as $r\rightarrow \infty$. This contradicts the assumption that $X_{n(r)}$ does not converge to zero almost surely. This completes the proof. 
	\end{proof}

	The following lemma was used in the proof of Lemma \ref{lem:uniform_convergence_asymmetric_weights}.
	
	\begin{lemma}\label{lem:weight_sum_bound}
		Fix a sequence $(w_{n})_{n\ge 1}$ of numbers in $(0,1]$. Denote $w^{n}_{k}:=w_{k}\prod_{i=k+1}^{n} (1-w_{i})$ for $1\le k \le n$. Suppose $w_{n}^{-1} - w_{n-1}^{-1}\le 1$ for all sufficiently large $n\ge 1$.  Fix $T\ge 1$. Then for all $n\ge T$, 
		\begin{align}\label{eq:weight_sum_bound1}
			\sum_{i=1}^{T} w^{n}_{i}  = O(1/n).
		\end{align}
	\end{lemma}

	\begin{proof}
		Suppose $w_{n}^{-1} - w_{n-1}^{-1}\le 1$ for all $n\ge N$ for some $N\ge 1$. It follows that $w_{n}^{-1} - w_{N}^{-1} \le n-N$, so $w_{n} \ge \frac{1}{n-N+w_{N}^{-1}}$. Hence for some constant $c>0$, $w_{n}\ge \frac{1}{n+c}$ for all $n\ge N$. Denote $a\lor b = \max(a,b)$. Then note that 
		\begin{align}
			w^{n}_{k}=w_{k} \exp\left(  \sum_{i=k+1}^{n} \log (1-w_{i}) \right)& \le \exp\left( -\sum_{i=k+1}^{n} w_{i} \right)  \le \exp\left( -\int_{N\lor (k+1)}^{n} \frac{1}{x+c}\,dx \right) =  \frac{[N\lor (k+1)]+c }{n+c},
		\end{align}
		where the second inequality uses $w_{k}\le 1$ and the following inequality uses $\log(1-a)\le -a$ for $a<1$.  Hence for each fixed $1\le T\le n$, we have 
		\begin{align}
			\sum_{k=1}^{T} w^{n}_{k} \le T\left( (N\lor (T+1))  +c\right)\frac{1}{n+c} . 
		\end{align}
		This shows the assertion. 
	\end{proof}

	\begin{remark}\label{rmk:weight}
		In the original statement of \cite[Lem B.7]{mairal2013stochastic}, the assumption that $w_{n+1}^{-1}-w_{n}^{-1}\le 1$ for sufficiently large $n$ was not used, and it seems that the argument in \cite{mairal2013stochastic} needs this assumption. To give more detail, the argument begins with writing the empirical loss $f_{n}(\cdot) = \sum_{k=1}^{n} w^{n}_{k} \, \ell(\mathcal{X}_{k},\cdot)$, where $w^{n}_{k}:=w_{k}(1-w_{k-1})\cdots (1-w_{n})$, and proceeds with assuming the monotonicity $w^{n}_{1}\le \dots \le w^{n}_{n}$, which is equivalent to $w_{1}\ge 1-w_{1}$ and $w_{k}\ge w_{k-1}(1-w_{k})$ for $2\le k \le n$. In turn, this is equivalent to $w_{k}^{-1}-w_{k-1}^{-1}\le 1$ for $2\le k \le n$  and $w_{1}\ge 1/2$.  Note that this condition implies $w_{k}^{-1} \le k+1$, or $w_{k} \ge \frac{1}{k+1}$.  This means that, asymptotically, $w_{k}$ cannot decay faster than the balanced weight $1/k$. We have used this observation in the proof of Lemma \ref{lem:weight_sum_bound}. 
		
		Next, we will argue that \ref{assumption:A4'-w_t}  implies \ref{assumption:A4-w_t}. This observation was made in \cite[Remark 20]{lyu2020online_CP}, but we repeat the argument here for completeness. It is clear that if the sequence $w_{n}\in (0,1]$ satisfies \ref{assumption:A4'-w_t}, then $\sum_{n=1}^{\infty} w_{n}=\infty$ and $\sum_{n=1}^{\infty} w_{n}^{2}\sqrt{n}<\infty$. So it remains to verify $w_{n}^{-1}-w_{n-1}^{-1}\le 1$ for sufficiently large $t$. Suppose $w_{n}=\Theta(n^{-\beta} (\log t)^{-\delta} )$ for some $\beta\in [0,1]$ and $\delta\ge 0$. Let $c_{1},c_{2}>0$ be constants such that $w_{n}t^{\beta} (\log t)^{\delta} \in [c_{1},c_{2}] $ for all $n\ge 1$. Then by the mean value theorem, 
		\begin{align}
			w_{n+1}^{-1} - w_{n}^{-1} & \le c_{2} \left( (n+1)^{\beta}(\log (n+1))^{\delta} - n^{\beta}(\log n)^{\delta} \right)  \\
			&\le c_{2}\sup_{n\le s \le n+1} \left( \beta s^{\beta-1}(\log s)^{\delta} + \delta s^{\beta-1} (\log s)^{\delta-1}  \right)\\
			&\le c_{2}\sup_{n\le s \le n+1}  s^{\beta-1}(\log s)^{\delta-1} \left( (\log s)+\delta \right).
		\end{align}
		Since $n\ge 1$, the last expression is of $o(1)$ if $\beta<1$. Otherwise, $w_{n}=n^{-1}$ for $t\ge 1$ by \ref{assumption:A4'-w_t}. Then $w_{n+1}^{-1} - w_{n}^{-1} \equiv 1$ for all $t\ge 1$.
	\end{remark}

\end{document}